\documentclass[12pt,oneside]{amsart}
\usepackage[margin=1in]{geometry}
\usepackage{enumitem,verbatim}
\usepackage{microtype}
\allowdisplaybreaks

\usepackage[export]{adjustbox}
\usepackage{amsmath,amssymb,amsthm,mathtools}
\usepackage{graphicx}
\usepackage[hyphens]{url}
\usepackage[pdfusetitle]{hyperref}
\usepackage[nameinlink]{cleveref}
\usepackage{color}
\hypersetup{colorlinks,linkcolor=blue,citecolor=cyan}
\usepackage{physics}
\usepackage{tikz-cd, tikz-3dplot}
\tikzset{anchorbase/.style={baseline={([yshift=-0.5ex]current bounding box.center)}}}
\tikzstyle directed=[postaction={decorate,decoration={markings,
    mark=at position #1 with {\arrow{>}}}}]
\tikzset{cross/.style={cross out, draw=black, minimum size=2*(#1-\pgflinewidth), inner sep=0pt, outer sep=0pt},
cross/.default={1pt}}
\usetikzlibrary{patterns, knots, calc, arrows.meta, decorations.markings,shapes.misc,math,hobby,intersections,spath3}
\usetikzlibrary{shadings}
\usetikzlibrary{arrows.meta}
\usepackage{asymptote}
\usepackage{standalone}
\usepackage{subcaption}
\usepackage{wasysym}
\usepackage{stmaryrd}
\usepackage{float}
\usepackage{bbm}

\makeatletter
\renewcommand*\env@matrix[1][\arraystretch]{%
  \edef\arraystretch{#1}%
  \hskip -\arraycolsep
  \let\@ifnextchar\new@ifnextchar
  \array{*\c@MaxMatrixCols c}}
\makeatother

\newtheorem{thm}{Theorem}[section]
\newtheorem{cor}[thm]{Corollary}
\newtheorem{lem}[thm]{Lemma}
\newtheorem{prop}[thm]{Proposition}
\newtheorem{conj}[thm]{Conjecture}

\theoremstyle{definition}
\newtheorem{defn}[thm]{Definition}
\newtheorem{eg}[thm]{Example}

\newtheorem{rmk}[thm]{Remark}

\DeclareMathOperator{\Sk}{\mathrm{Sk}}
\DeclareMathOperator{\SkAlg}{\mathrm{SkAlg}}
\DeclareMathOperator{\dual}{d}
\DeclareMathOperator{\QGM}{\mathrm{QGM}}
\DeclareMathOperator{\SQGM}{\mathrm{SQGM}}

\title{3d quantum trace map}

\author[S. Panitch]{Samuel Panitch}
\address{Department of Mathematics, Yale University, New Haven, CT 06511, USA}
\email{\href{mailto:sam.panitch@yale.edu}{sam.panitch@yale.edu}}

\author[S. Park]{Sunghyuk Park}
\address{Department of Mathematics, Harvard University, Cambridge, MA 02138, USA}
\email{\href{mailto:sunghyukpark@math.harvard.edu}{sunghyukpark@math.harvard.edu}}

\begin{document}

\maketitle

\begin{abstract}
We construct the 3d quantum trace map, a homomorphism from the Kauffman bracket skein module of an ideally triangulated 3-manifold to its (square root) quantum gluing module, thereby giving a precise relationship between the two quantizations of the character variety of ideally triangulated 3-manifolds. 
This map, whose existence was conjectured earlier by \cite{AGLR}, is a natural 3-dimensional analog of the 2d quantum trace map of \cite{BW}. 
Our construction is based on the study of stated skein modules and their behavior under splitting, especially into face suspensions. 
\end{abstract}

\tableofcontents

\section{Introduction}

The goal of this paper is to establish a precise relationship between two different quantizations of the character variety
\[
X_{\mathrm{SL}_2(\mathbb{C})}(Y) := \mathrm{Hom}(\pi_1(Y) , \mathrm{SL}_2(\mathbb{C})) \sslash \mathrm{SL}_2(\mathbb{C})
\]
of a $3$-manifold $Y$ equipped with an ideal triangulation $\mathcal{T}$. 

The first is the \emph{skein module} $\Sk(Y)$ \cite{Tur, BFK, PS}, which is the free $\mathbb{Z}[A^{\pm 1}]$-module spanned by all isotopy classes of framed links in $Y$, modulo Kauffman bracket skein relations. 
When $Y = \Sigma \times I$ where $\Sigma$ is a surface and $I$ is the interval $[0,1]$, 
then the corresponding skein module is called the \emph{skein algebra} $\SkAlg(\Sigma) := \Sk(\Sigma \times I)$, where the algebra structure is given by stacking links along the $I$-direction. 
It provides a quantization (in the sense of deformation quantization) \cite{Bullock, PS} of the character variety $X_{\mathrm{SL}_2(\mathbb{C})}(\Sigma)$, which is equipped with a natural symplectic (hence Poisson) structure \cite{Pet, Weil, AB, Gol1, Gol2}. 

The second is the \emph{quantum gluing module} $\QGM_{\mathcal{T}}(Y)$ \cite{Dim,GKRY}, which is a 3-dimensional analog of the \emph{quantum Teichm\"uller space} \cite{Kashaev, CF1, CF2}. 
For an ideally triangulated surface $\Sigma$, 
the quantum Teichm\"uller space associates a quantum torus, consisting of Laurent polynomials in $q$-commuting variables called the quantized shear coordinates. 
For an ideally triangulated 3-manifold $Y$, 
the quantum gluing module consists of Laurent polynomials in variables called the quantized shape parameters, 
modulo certain relations quantizing the usual gluing relations and the 3-term relations for the shape parameters of an ideal triangulation $\mathcal{T}$. 

In the case of surfaces, a precise relationship between the two quantizations of the character variety was established by Bonahon and Wong \cite{BW}, 
where they constructed an injective algebra homomorphism from the skein algebra to quantum Teichm\"{u}ller space. 
This map is called the \emph{quantum trace}, 
as the correspondence between the classical limit of the skein algebra and the algebra of regular functions on the character variety $X_{\mathrm{SL}_2(\mathbb{C})}(\Sigma)$ is given in terms of the trace functions (i.e.\ trace of holonomy) associated to closed immersed curves $K$ in $\Sigma$, 
which, in turn, can be expressed as Laurent polynomials in \emph{square roots} of the shear coordinates. 
See also \cite{Gabella, NY} for related constructions developed from the physics point of view. 

Motivated by the construction of Bonahon and Wong, 
Le and collaborators \cite{Le, CL, CL2, LS} have developed the theory of \emph{stated skein modules}. 
They generalize the usual skein module by allowing one to consider not just framed links in $Y$, but also framed tangles in $Y$ ending on a \emph{boundary marking}, a disjoint union of intervals in $\partial Y$. 
The most important property of stated skein modules is that there is a \emph{splitting homomorphism}, a map from the stated skein module of a $3$-manifold to the tensor product of stated skein modules of simpler pieces obtained by cutting the $3$-manifold along disks. 
By splitting a surface into ideal triangles and passing to a certain quotient of the stated skein algebra called the \emph{reduced stated skein algebra}, 
Costantino and Le \cite{CL} provided a simpler proof of the well-definedness of the quantum trace of Bonahon and Wong.

We develop a generalized notion of stated skein modules which allows the boundary markings to be any embedded bipartite graph in $\partial Y$, rather than just disjoint unions of intervals. 
We then prove the existence of a splitting homomorphism in this setting, generalizing the splitting map of stated skein modules of 3-manifolds studied by Costantino and Le \cite{CL2}. 
Our splitting map allows us to cut a 3-manifold into simpler constituent pieces, such as ideal tetrahedra.

Using these new ideas, we establish a precise relationship between the two quantizations of the character variety of ideally triangulated 3-manifolds. 
Just as in the case of surfaces, the classical trace functions on the character variety $X_{\mathrm{SL}_2(\mathbb{C})}(Y)$ can be expressed as Laurent polynomials in \emph{square roots} of the shape parameters of the ideal triangulation. 
Therefore, it is natural to introduce a refinement of the quantum gluing module -- which we call the \emph{square root quantum gluing module} $\SQGM_{\mathcal{T}}(Y)$ --
which consists of Laurent polynomials in square roots of the quantized shape parameters of $\mathcal{T}$, modulo certain natural relations. 
Our main theorem is the following: 
\begin{thm}[3d quantum trace map; detailed version in Theorem \ref{thm:quantumTrace}]
    There is a $\mathbb{Z}[A^{\pm \frac{1}{2}}, (-A^2)^{\pm \frac{1}{2}}]$-module homomorphism 
    \[
    \Tr_{\mathcal{T}}: \Sk(Y) \rightarrow \SQGM_{\mathcal{T}}(Y).
    \]
\end{thm}

The construction of this map immediately settles part of the conjecture of \cite{AGLR}.

\subsection*{Organization of this paper}
This paper is organized as follows. 

In Section \ref{sec:classical_case}, we recall the classical story behind the 2d quantum trace, and develop an analagous one for the 3d quantum trace. 
We discover that the analogue of crossing an edge of the triangulation of a surface is not crossing a face of a tetrahedron, but rather crossing a geometric object we call \emph{edge cone}. 
This is summed up in Proposition \ref{prop:3d_classical_state_sum}.

In Section \ref{sec:SkeinModules}, we define stated skein modules and prove important properties that will be used throughout the rest of the paper. 
There are two differences between our definition of stated skein modules and the definitions in the literature: 
\begin{enumerate}
    \item\label{item:difference1} we include a skein relation that resolves half twists, and
    \item\label{item:difference2} we allow the boundary markings of our $3$-manifolds to be any embedded bipartite graph, rather than just disjoint unions of intervals. 
\end{enumerate}
The first \eqref{item:difference1} is mostly cosmetic; 
the resolution of half twists allows us to write other skein relations in a manifestly symmetric way. 
On the other hand, the second \eqref{item:difference2} is crucial: 
a critical observation underpinning our theory is that skein modules of $3$-manifolds are naturally modules over skein algebras associated to the vertices of their boundary markings. 
This fact is summarized in Proposition \ref{prop:skein_alg_structure}. 

The main theorem from this section is Theorem \ref{thm:SplittingMap} (and its generalization, Theorem \ref{thm:partialSplittingHomomorphism}), where we generalize the splitting map of \cite{CL2}. 
From this theorem, we obtain as easy corollaries a splitting map that cuts an ideally triangulated manifold $Y$ into ideal tetrahedra or into what we call \emph{face suspensions}. 

We also define \emph{reduced stated skein modules} \cite{CL} by taking a quotient with respect to \emph{bad arcs}. 
It is these reduced versions that will actually be used to construct the 3d quantum trace. 
The splitting map and module structures defined above descend to this quotient.  

Section \ref{sec:struc_theorems} is dedicated to finding an explicit presentation of the stated skein modules of $3$-balls with boundary markings, as summarized in Theorem \ref{thm:SkeinModuleOf3Ball}. 
As a corollary (Corollary \ref{cor:red_skein_module_fs}), we obtain a simple presentation for the reduced stated skein module of a face suspension. 

In Section \ref{sec:3dQuantumTrace}, we construct the 3d quantum trace map.
To do so, we first apply a version of the splitting map (Corollary \ref{cor:fs_splitting_reduced}) for reduced stated skein modules to cut an ideally triangulated 3-manifold $Y$ into face suspensions. 
Then, we use the presentation of the reduced stated skein module of a face suspension obtained in Corollary \ref{cor:red_skein_module_fs} to construct the quantum trace on a single face suspension, whose codomain is what we call the \textit{face suspension module}.
Just as quantum Teichm\"{u}ller space can be realized as the tensor product of triangle algebras, one for each triangle in the triangulation, 
the square root quantum gluing module can be constructed as some quotient of the tensor product of face suspension modules, one for each face suspension composing $Y$.
By composing the tensor product of the face suspension quantum trace with the splitting map from above, we obtain the 3d quantum trace.  

Section \ref{sec:examples} contains concrete examples of computations using our map. 
We compute the quantum trace of three elements of the skein module of the figure-8 knot complement.

Finally, in Section \ref{sec:future_directions}, we list some interesting avenues of future research.

\subsection*{Acknowledgements}
We are indebted to Andy Neitzke who generously shared many ideas with us. 
We also thank Francesco Costantino, Daniel Douglas, Tobias Ekholm, Dan Freed, Sergei Gukov, Ka Ho Wong, and Fei Yan for interesting discussions. 

S. Park gratefully acknowledges support from Simons Foundation through Simons Collaboration on Global Categorical Symmetries.


\section{Classical story} \label{sec:classical_case}
In this section, we remind the reader about the state sum formulation for the classical trace map in 2d, and develop the corresponding story in 3d using face suspensions. 

\subsection{State sum formulation of the 2d classical trace map}
The following well known result can be found in \cite{BW} and \cite{Bonahon}, among other sources. 
Let $\Sigma$ be an oriented punctured surface with boundary that admits an ideal triangulation $\lambda$. 
Let $\mathcal{E}$ denote the set of interior edges in the triangulation.
To each interior edge $e_i \in \mathcal{E}$, we associate a shear parameter $X_i \in \mathbb{R}_+$. 
Associated to this data, we construct a group homomorphism $r: \pi_1(\Sigma)\rightarrow \mathrm{PSL}_2(\mathbb{R})$ as follows:
\begin{enumerate}
    \item Lift the ideal triangulation $\lambda$ to an ideal triangulation $\widetilde{\lambda}$ of the universal cover $\widetilde{\Sigma}$. 
    \item\label{item:2dPSL2Rrep} Construct an orientation-preserving immersion $\widetilde{g}:\widetilde{\Sigma}\rightarrow \mathbb{H}^2$. 
    Having fixed the shear parameters of the ideal triangulation, this map is uniquely determined up to isotopy of $\widetilde{\Sigma}$ respecting $\widetilde{\lambda}$ once we have chosen some $\Delta \in \lambda$ to map to the ideal triangle with vertices at $0, 1$, and $\infty$. 
    This implies $\widetilde{g}$ is unique up to composition by an element of $\mathrm{PSL}_2(\mathbb{R})$. 
    \item From the construction in step \ref{item:2dPSL2Rrep}, we obtain a unique group homomorphism $r$ satisfying $\widetilde{g}(\gamma \widetilde{\Delta})=r(\gamma)\widetilde{g}(\widetilde{\Delta})$ for each $\Delta\in \lambda$. 
    More concretely, if $\gamma$ is a closed, immersed curve in $\Sigma$, then $r(\gamma)\in \mathrm{PSL}_2(\mathbb{R})$ is the unique orientation preserving isometry of $\mathbb{H}^2$ sending the point $\widetilde{g}\circ \widetilde{\gamma}(0)$ to $\widetilde{g}\circ \widetilde{\gamma}(1)$ and the vector $(\widetilde{g}\circ \widetilde{\gamma})'(0)$ to $(\widetilde{g}\circ \widetilde{\gamma})'(1)$. 
    Since the family of triangles $\widetilde{g}(\widetilde{\lambda})\subset \mathbb{H}^2$ was unique up to composition by an element of $\mathrm{PSL}_2(\mathbb{R})$, $r$ is unique up to conjugation by an element of $\mathrm{PSL}_2(\mathbb{R})$. 
\end{enumerate}
Such a group homomorphism can be constructed explicitly. Suppose the curve $\gamma$ transversely meets the edges $e_{1}, e_{2},\cdots, e_{k}, e_{k+1} = e_{1}$ in that order. 
After crossing the edge $e_{i}$, $\gamma$ enters a face $\Delta_i$ of $\lambda$, and then exits $\Delta_i$ across the edge $e_{i+1}$. 
There are three possibilities for the edge $e_{i+1}$:
\begin{enumerate}
    \item $e_{i+1}$ is the edge immediately to the left of $e_{i}$ after crossing into $\Delta_i$. 
    In this case, define 
    \[
    M_i = 
    \begin{bmatrix} 
    1 & 1 \\ 
    0 & 1
    \end{bmatrix}.
    \]
    \item $e_{i+1}$ is the edge immediately to the right of $e_{i}$ after crossing into $\Delta_i$. 
    In this case, define 
    \[
    M_i = 
    \begin{bmatrix} 
    1 & 0 \\ 
    1 & 1
    \end{bmatrix}.
    \]
    \item $e_{i+1}$ is $e_{i}$. 
    In this case, define 
    \[
    M_i = 
    \begin{bmatrix} 
    0 & 1 \\ 
    -1 & 0
    \end{bmatrix}.
    \]
\end{enumerate}
Also, for $X \in \mathbb{R}_+$, define
\[
S(X) = \begin{bmatrix} X^\frac{1}{2} & 0\\ 0 & X^{-\frac{1}{2}}\end{bmatrix}.
\]

\begin{lem} \label{lem:2dclassicaltrace}
    Up to conjugation by an element of $\mathrm{PSL}_2(\mathbb{R})$, 
    \[
    r(\gamma) = S(X_{1}) M_1 S(X_{2}) M_2 \cdots S(X_{k}) M_k,
    \]
    where the matrices $M_i$ and $S(X_{i})$ are defined as above, and $X_i$ is the shear parameter associated to the edge $e_i\in \mathcal{E}$.    
\end{lem}
From Lemma \ref{lem:2dclassicaltrace}, it is straightforward to obtain a state sum formula for the trace of $r(\gamma)$.
Let a state $s$ assign a number $s_1, s_2, \cdots, s_{k} \in \{\pm 1\}$ to each point where $\gamma$ crosses an edge $e_{i}$ in this order. 
For each $i$, write the matrix $M_i$ from above as 
\[
M_i = \begin{bmatrix} m_i^{++} & m_i^{+-}\\ m_i^{-+} & m_i^{--} \end{bmatrix}.
\]

\begin{lem}
    If $\hat{r}(\gamma) \in \mathrm{SL}_2(\mathbb{R})$ is a lift of $r(\gamma)$, then 
    we have, up to sign, 
    \[
    \Tr(\hat{r}(\gamma)) = \pm \sum_s m_1^{s_1s_2} m_2^{s_2s_3} \cdots m_k^{s_ks_1} X_{1}^{s_1/2} X_{2}^{s_2/2} \cdots X_{k}^{s_k/2} ,
    \]
    where the sum is over all possible states $s$.
\end{lem}

\subsection{State sum formulation of the 3d classical trace map}\label{subsec:3dClassicalTrace}

Now, we will develop a similar story for ideally triangulated $3$-manifolds. 
Surprisingly, a simple state sum formulation is most readily obtained not from cutting the $3$-manifold into tetrahedra, but into objects we call face suspensions. 

We begin with a brief reminder of the meaning of shape parameters of an ideal tetrahedron \cite{Thu}. 
An ideal tetrahedron is the convex hull of $4$ points in $\mathbb{H}^3$, all of which lie on the sphere at infinity. 
\begin{figure}[H]
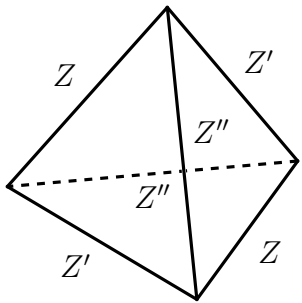

    \centering
    \includestandalone[scale=1]{figures/tetrahedron_no_leaf_space}
    \caption{Shape parameters of a tetrahedron}
    \label{fig:tetrahedron}
\end{figure}
To each edge, we associate a complex number called the \emph{shape parameter}. 
Opposite edges have the same label, and provided we label the edges in the appropriate order on each face, as shown in Figure \ref{fig:tetrahedron},  
these shape parameters must obey the relations 
\begin{equation} \label{eqn:classicalRels1}
    \begin{aligned}
    ZZ'Z'' &= -1, \\
    Z''^{-1}+Z &= 1,\\
    Z^{-1}+Z' &= 1, \\
    Z'^{-1}+Z'' &= 1.
    \end{aligned}
\end{equation}

By an orientation preserving isometry of $\mathbb{H}^3$, it is always possible to take $3$ of the vertices of an ideal tetrahedron to the points $0, 1$, and $\infty$, and the fourth to a complex number with positive imaginary part; by a relabeling if necessary, we can ensure this complex number is $Z$.

Throughout the paper, we will use the terms \emph{bare vertex} and \emph{bare edge} to reference a vertex or an edge of a tetrahedron before gluing. 

Let $Y$ be a $3$-manifold that admits an ideal triangulation $\mathcal{T}$. 
Let $\mathcal{F}$ denote the set of faces, $\mathcal{E}$ the set of internal edges, $\textbf{e}(T)$ the set of $6$ bare edges of each tetrahedron $T$, and $\textbf{v}(T)$ the set of $4$ bare vertices of each tetrahedron $T$. 

Associated to each tetrahedron $T\in \mathcal{T}$, there are $3$ shape parameters $Z_T, Z'_T, Z''_T$, obeying the relations in (\ref{eqn:classicalRels1}). 
In order for the tetrahedra in $\mathcal{T}$ to glue together to yield a hyperbolic structure, we require that the gluing relations hold. 
That is, if the bare edges $e_1, e_2, \cdots, e_k$ of tetrahedra $T_1, T_2, \cdots, T_k$ are glued around a single internal edge $e\in \mathcal{E}$, we require 
\begin{equation}\label{eqn:classicalRels2}
\log Z_{e_1} + \log Z_{e_2} + \cdots + \log Z_{e_k} = 2\pi i,
\end{equation}
where 
$Z_{e_i} \in \{Z_T, Z'_T, Z''_T \;\vert\; T \in  \mathcal{T}\}$ 
is the shape parameter associated to the bare edge $e_i$.\footnote{Here, we are taking the principal branch of the logarithm.} 

Having chosen a set of shape parameters obeying the relations in equations \eqref{eqn:classicalRels1} and \eqref{eqn:classicalRels2}, we construct a group homomorphism $r: \pi_1(Y)\rightarrow \mathrm{PSL}_2(\mathbb{C})$ as follows:
\begin{enumerate}
    \item Lift the ideal triangulation $\mathcal{T}$ to an ideal triangulation $\widetilde{\mathcal{T}}$ of the universal cover $\widetilde{Y}$. 
    \item\label{item:3dPSL2Crep} Construct an orientation-preserving immersion $\widetilde{g}:\widetilde{Y}\rightarrow \mathbb{H}^3$. 
    Having already chosen the shape parameters for each tetrahedron in $\mathcal{T}$, this map is uniquely determined up to isotopy of $\widetilde{Y}$ respecting $\widetilde{\mathcal{T}}$ once we have chosen a tetrahedron $T$ to map to the ideal tetrahedron with vertices at $0, 1, \infty$, and $Z_T$. 
    This implies $\widetilde{g}$ is unique up to composition with an element of $\mathrm{PSL}_2(\mathbb{C})$. 
    \item From the construction in step \ref{item:3dPSL2Crep}, we obtain a unique group homomorphism $r$ satisfying $\widetilde{g}(\gamma \widetilde{T}) = r(\gamma)\widetilde{g}(\widetilde{T})$ for every tetrahedron $T\in \mathcal{T}$. 
    More concretely, if $\gamma$ is a closed, immersed curve in $Y$, then $r(\gamma) \in \mathrm{PSL}_2(\mathbb{C})$ is the unique orientation preserving isometry of $\mathbb{H}^3$ sending the point $\widetilde{g}\circ \widetilde{\gamma}(0)$ to $\widetilde{g}\circ \widetilde{\gamma}(1)$ and the vector $(\widetilde{g}\circ \widetilde{\gamma})'(0)$ to $(\widetilde{g}\circ \widetilde{\gamma})'(1)$. 
    Since the family of tetrahedra $\widetilde{g}(\widetilde{\mathcal{T}}) \subset \mathbb{H}^3$ was unique up to composition with an element of $\mathrm{PSL}_2(\mathbb{C})$, $r$ is unique up to conjugation by an element of $\mathrm{PSL}_2(\mathbb{C})$. 
\end{enumerate}

To make the above construction explicit, we first introduce face suspensions.

\begin{defn}\label{defn:ConesAndSuspensions}
    \begin{enumerate} 
    \item The \emph{barycenter} of an ideal tetrahedron $T \in \mathcal{T}$ is the distinguished point in the interior of $T$ given by the intersection of the $3$ common perpendiculars of pairs of opposite edges of $T$.
    \item For an ideal tetrahedron $T \in \mathcal{T}$ and a bare vertex $v \in \mathbf{v}(T)$, the \emph{cone $Cv$ over $v$}, or \emph{vertex cone} for short, is the unique geodesic between the barycenter of $T$ and $v$. 
    A vertex cone is shown in Figure \ref{fig:vertex_cone}.
    \item For an ideal tetrahedron $T \in \mathcal{T}$ and a bare edge $e \in \mathbf{e}(T)$, 
    the \emph{cone $Ce$ over $e$}, or \emph{edge cone} for short, is the component of the geodesic hemisphere passing through the bare edge $e$ and the barycenter of $T$ bounded by the two vertex cones corresponding to the vertices of $e$. 
    An edge cone is shown in Figure \ref{fig:edge_cone}.
    
    Note that each edge cone naturally corresponds to a shape parameter.    
    \item For a face $f\in \mathcal{F}$ of the ideal triangulation, the \emph{suspension $Sf$ over $f$}, or \emph{face suspension} for short, is the $3$ manifold whose boundary is the $6$ edge cones associated to the three edges of the given face. 
    A face suspension is shown in Figure \ref{fig:faceSusp}.
    \end{enumerate}
\end{defn}

\begin{figure}[htbp]
    \centering
    \begin{subfigure}[t]{.45\linewidth}    
        \centering
        \includestandalone[scale=1]{figures/vertex_cone}
        \caption{A vertex cone corresponding to the bare vertex $v$.}
        \label{fig:vertex_cone}
    \end{subfigure}
    \hspace{5mm}
    \begin{subfigure}[t]{.45\linewidth}
        \includestandalone[scale=1]{figures/edge_cone}
        \caption{An edge cone corresponding to the bare edge $e$.}
        \label{fig:edge_cone}
    \end{subfigure}
    \caption{Vertex cones and edge cones in a tetrahedron $T$.}
\end{figure}

\begin{figure}[htbp]
    \centering
    \includegraphics[scale=.7]{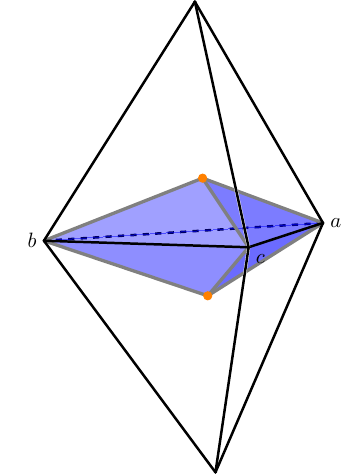}
    \caption{A face suspension corresponding to the face with vertices $a, b, c$. 
    The vertex cones are shown in gray and the edge cones are shown in blue. 
    The orange dots are the barycenters of the two tetrahedra.}
    \label{fig:faceSusp}
\end{figure}

Now suppose a curve $\gamma$ transversally meets the edge cones $Ce_{1}, Ce_{2}, \cdots, Ce_{k}, Ce_{k+1} = Ce_{1}$, in that order. 
After crossing an edge cone $Ce_i$, $\gamma$ enters a face suspension $Sf_{i}$, and exits across the edge cone $Ce_{i+1}$. 
Let $T_i$ be the tetrahedron containing $Ce_i$. 
There are six possibilities for the edge cone $Ce_{i+1}$, and thus $6$ different $\mathrm{PSL}_2(\mathbb{C})$ matrices to consider:
\begin{enumerate}
    \item\label{item:left_turn} When viewed from the center of $T_{i}$, $Ce_{i+1}$ is immediately to the left of $Ce_i$ after crossing into $Sf_i$. 
    In this case, define 
    \[
    M_i = 
    \begin{bmatrix} 
    1 & 1 \\ 
    0 & 1
    \end{bmatrix}.
    \]
    \item When viewed from the center of $T_{i}$, $Ce_{i+1}$ is immediately to the right of $Ce_i$ after crossing into $Sf_i$. 
    In this case, define 
    \[
    M_i = 
    \begin{bmatrix} 
    1 & 0 \\ 
    1 & 1
    \end{bmatrix}.
    \]
    \item $Ce_{i+1}$ is $Ce_i$. 
    In this case, define 
    \[
    M_i = 
    \begin{bmatrix} 
    0 & 1 \\ 
    -1 & 0
    \end{bmatrix}.
    \]
    \item\label{item:elIsom} When viewed from the center of $T_{i}$, $Ce_{i+1}$ is across $f_{i}$ and to the left. 
    In this case, define 
    \[
    M_i = 
    \begin{bmatrix} 
    (-1)^{\frac{1}{2}} & (-1)^{\frac{1}{2}} \\ 
    (-1)^{\frac{1}{2}} & 0
    \end{bmatrix}.
    \]
    \item\label{item:down_right_turn} When viewed from the center of $T_{i}$, $Ce_{i+1}$ is across $f_{i}$ and to the right. 
    In this case, define 
    \[
    M_i = 
    \begin{bmatrix} 
    0 & (-1)^{\frac{1}{2}} \\ 
    (-1)^{\frac{1}{2}} & (-1)^{\frac{1}{2}}
    \end{bmatrix}.
    \]
    \item When viewed from the center of $T_{i}$, $Ce_{i+1}$ is across $f_{i}$ and directly below. 
    In this case, define 
    \[
    M_i = 
    \begin{bmatrix} 
    (-1)^{-\frac{1}{2}} & 0 \\
    0 & (-1)^{\frac{1}{2}}
    \end{bmatrix}.
    \]
\end{enumerate}
We illustrate some of these possibilities in Figure \ref{fig:faceSuspTurns}. 

\begin{figure}[H]
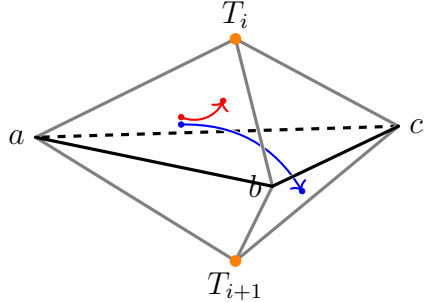

    \centering
    \includestandalone[scale=1.0]{figures/turnsInFaceSusp}
    \caption{Here, $Ce_{i}$ is the edge cone in tetrahedron $T_i$ corresponding to the edge with bare vertices $a$ and $b$. 
    In red, we show a left turn corresponding to case (\ref{item:left_turn}) above, which exits across the edge cone in $T_i$ corresponding the edge with bare vertices $a$ and $c$. 
    In blue, we show a turn corresponding to case (\ref{item:down_right_turn}), which exits across the edge cone in $T_{i+1}$ corresponding to the edge with bare vertices $b$ and $c$.}
    \label{fig:faceSuspTurns}
\end{figure}

Lastly, for $Z\in \mathbb{C}$, define the matrix 
\[
S(Z) = \begin{bmatrix} (-Z)^{\frac{1}{2}} & 0\\ 0 & (-Z)^{-\frac{1}{2}}\end{bmatrix}.
\]
\begin{prop}
    Up to conjugation by an element of $\mathrm{PSL}_2(\mathbb{C}),$
    \[
    r(\gamma) = S(Z_{1}) M_1 S(Z_{2}) M_2 \cdots S(Z_{k}) M_k,
    \]
    where the matrices $M_i$ and $S(Z_{i})$ are defined as above, and $Z_{i} \in \{Z_T, Z'_T, Z''_T \;\vert\; T \in \mathcal{T}\}$ is the shape parameter associated to the edge cone $Ce_i$. 
\end{prop}
\begin{proof}
    An orientation preserving isometry of $\mathbb{H}^3$ that maps a face $f_i \in \mathcal{F}$ to the face $f_j \in \mathcal{F}$ maps $Sf_i$ to $Sf_j$. 
    This follows immediately from the definitions of edge cones and hyperbolic barycenters. 
    
    Next, write $r(\gamma) = I_1 I_2 \cdots I_k$ where each $I_i$ is a hyperbolic isometry that takes $Sf_i$ to $Sf_{i+1}$, and arrange that $f_{1}$ is the face with vertices at $0, 1$, and $\infty$. 
    Suppose that the tetrahedra abutting $f_1$ have their fourth vertices at $Z_1^{-1}$ and $Z_2$, where $\Im Z_2>0$ and $\Im Z_1^{-1}<0$.
    Further arrange that $\gamma$ enters $Sf_1$ through the edge cone corresponding to the bare edge with vertices at $0$ and $1$ inside the tetrahedron with its fourth vertex at $Z_1^{-1}$.
    
    Observe that we can write $I_i = I_{i-1} I_{i-2} \cdots I_1 I_i^* I_1^{-1} \cdots I_{i-2}^{-1} I_{i-1}^{-1}$, where $I_i^*$ is an isometry from a face suspension corresponding to the face with vertices at $0, 1, $ and $\infty$ to one of its neighbors. 
    Furthermore, this face suspension lives inside tetrahedra with their fourth vertices at $Z_{i+1}$ and $Z_{i}^{-1}$. 
    Substituting these expressions into $r(\gamma) = I_1 I_2 \cdots I_k,$ we see it is enough to compute all of the elementary isometries that map a face suspension associated to the face with vertices at $0,1,$ and $\infty$ to one of its neighbors.
    
    These elementary isometries are an explicit computation. 
    For instance, the hyperbolic isometry depicted in Figure \ref{fig:faceSuspInH3} is the isometry that maps $0 \mapsto 1$, $1 \mapsto Z_{i+1}$, and $\infty \mapsto 1$.
\end{proof}

\begin{figure}[H]
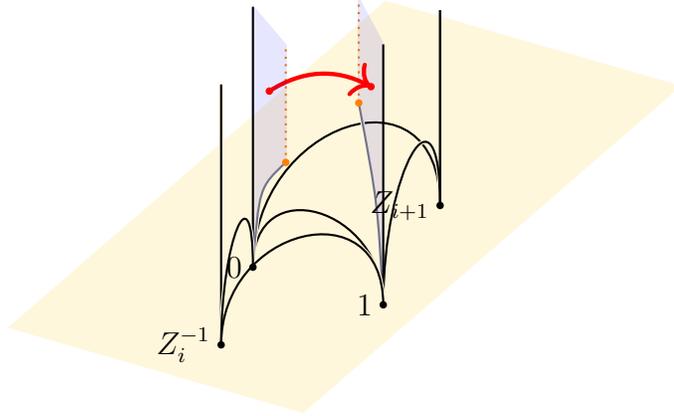

    \centering
    \includestandalone[scale=2]{figures/faceSuspInUniCov}
    \caption{A curve corresponding to the elementary isometry described in (\ref{item:elIsom}) above.}
    \label{fig:faceSuspInH3}
\end{figure}

Now, let a state $s$ assign a number $s_1, s_2, \cdots, s_{k} \in \{\pm 1\}$ to each point where $\gamma$ crosses an edge cone $Ce_{i}$ in this order. 
We write 
\[
M_i = 
\begin{bmatrix} 
m_i^{++} & m_i^{+-} \\ 
m_i^{-+} & m_i^{--} 
\end{bmatrix},
\]
as before and again obtain the following proposition from some easy linear algebra. 
\begin{prop}\label{prop:3d_classical_state_sum}
If $\hat{r}(\gamma) \in \mathrm{SL}_2(\mathbb{C})$ is a lift of $r(\gamma)$, then 
we have, up to sign, 
\[
\Tr(\hat{r}(\gamma)) = \pm \sum_s m_1^{s_1s_2} m_2^{s_2s_3} \cdots m_k^{s_ks_1} Z_{1}^{s_1/2} Z_{2}^{s_2/2} \cdots Z_{k}^{s_k/2} ,
\]
where the sum is over all possible states $s$.
\end{prop}


\section{Stated skein modules of 3-manifolds} \label{sec:SkeinModules}
In this section, we define stated skein modules and prove general facts about them that we will use in later sections. 

Let $R := \mathbb{Z}[A^{\pm \frac{1}{2}}, (-A^2)^{\pm \frac{1}{2}}]$. 

\begin{defn}
Let $Y$ be a 3-manifold with boundary. 
A \emph{boundary marking} is a smoothly embedded oriented graph $\Gamma\subset \partial Y$ where every vertex is either a source or a sink.\footnote{Sometimes, it will be useful to allow 2-valent vertices whose adjacent edges are oriented in such a way that one edge is pointing toward the vertex and the other edge is pointing away from the vertex. 
These boundary markings are equivalent to the boundary markings obtained by removing such vertices and smoothly connecting the adjacent edges.} 
\end{defn}
\begin{rmk}
The definition of stated skein modules we are about to give is more general than the ones appearing in the literature (see for instance \cite{CL2}) as we allow the boundary marking $\Gamma \subset \partial Y$ to be any embedded bipartite graph, not just a disjoint union of intervals. 
This generalization is crucial in our construction of the 3d quantum trace map. 
\end{rmk}

\begin{defn}
Let $(Y, \Gamma)$ be a boundary marked 3-manifold. 
The \emph{(stated $\mathrm{SL}_2$) skein module} 
$\Sk(Y, \Gamma)$ is the free $R$-module spanned by the isotopy classes of unoriented ribbon tangles in $Y$, each of whose boundary components lies flat in the interior of $\Gamma \setminus V(\Gamma)$ and carries a sign $\in \{\pm 1 \}$ (called a \emph{state}), modulo the following stated skein relations:
\begin{gather*}
\vcenter{\hbox{
\begin{tikzpicture}
\draw[dotted] (0,0) circle (1);
\draw[line width=5] ({sqrt(2)/2},{-sqrt(2)/2}) -- ({-sqrt(2)/2},{sqrt(2)/2});
\draw[white, line width=10] ({-sqrt(2)/2},{-sqrt(2)/2}) -- ({sqrt(2)/2},{sqrt(2)/2});
\draw[line width=5] ({-sqrt(2)/2},{-sqrt(2)/2}) -- ({sqrt(2)/2},{sqrt(2)/2});
\end{tikzpicture}
}}
\;\;=\;\;
A\;
\vcenter{\hbox{
\begin{tikzpicture}
\draw[dotted] (0,0) circle (1);
\draw[line width=5] ({sqrt(2)/2},{sqrt(2)/2}) arc (135:225:1);
\draw[line width=5] ({-sqrt(2)/2},{-sqrt(2)/2}) arc (-45:45:1);
\end{tikzpicture}
}}
\;\;+\;\;
A^{-1}\;
\vcenter{\hbox{
\begin{tikzpicture}
\draw[dotted] (0,0) circle (1);
\draw[line width=5] ({sqrt(2)/2},{sqrt(2)/2}) arc (-45:-135:1);
\draw[line width=5] ({-sqrt(2)/2},{-sqrt(2)/2}) arc (135:45:1);
\end{tikzpicture}
}}
\;,
\\
\vcenter{\hbox{
\begin{tikzpicture}
\draw[dotted] (0,0) circle (1);
\draw[line width=5] (0,0) circle (0.5);
\end{tikzpicture}
}}
\;\;=\;\;
(-A^2 - A^{-2})\;
\vcenter{\hbox{
\begin{tikzpicture}
\draw[dotted] (0,0) circle (1);
\end{tikzpicture}
}}
\;,
\\
\vcenter{\hbox{
\tdplotsetmaincoords{20}{30}
\begin{tikzpicture}[tdplot_main_coords, rotate=30]
\begin{scope}[scale = 1.0, tdplot_main_coords]
\fill[color=lightgray, opacity=0.3] (0, -1, -1) -- (0, -1, 1) -- (0, 1, 1) -- (0, 1, -1) -- cycle;
\draw[ultra thick, orange, ->] (0, -1, 0) -- (0, 0, 0);
\draw[ultra thick, orange] (0, -1, 0) -- (0, 1, 0);
\draw[dotted] (0, -1, -1) -- (0, -1, 1) -- (0, 1, 1) -- (0, 1, -1) -- cycle;
\draw[dotted] (-1.5, -1, -1) -- (-1.5, -1, 1) -- (-1.5, 1, 1) -- (-1.5, 1, -1) -- cycle;
\draw[dotted] (0, -1, -1) -- (-1.5, -1, -1);
\draw[dotted] (0, -1, 1) -- (-1.5, -1, 1);
\draw[dotted] (0, 1, -1) -- (-1.5, 1, -1);
\draw[dotted] (0, 1, 1) -- (-1.5, 1, 1);
\draw[line width=5] (0, 0.5, 0) .. controls (-0.8, 0.5, 0) and (-0.8, -0.5, 0) .. (0, -0.5, 0);
\node[anchor = west] at (0, 0.5, 0) {$\mu$};
\node[anchor = west] at (0, -0.5, 0) {$\nu$};
\end{scope}
\end{tikzpicture}
}}
\;\;=\;\;
\delta_{\mu, -\nu}\,(-A^2)^{\frac{\mu}{2}}
\;,
\quad\quad \mu, \nu \in \{\pm 1\}\;,
\\
\vcenter{\hbox{
\tdplotsetmaincoords{20}{30}
\begin{tikzpicture}[tdplot_main_coords, rotate=30]
\begin{scope}[scale = 1.0, tdplot_main_coords]
\fill[color=lightgray, opacity=0.3] (0, -1, -1) -- (0, -1, 1) -- (0, 1, 1) -- (0, 1, -1) -- cycle;
\draw[ultra thick, orange, ->] (0, -1, 0) -- (0, 0, 0);
\draw[ultra thick, orange] (0, -1, 0) -- (0, 1, 0);
\draw[dotted] (0, -1, -1) -- (0, -1, 1) -- (0, 1, 1) -- (0, 1, -1) -- cycle;
\draw[dotted] (-1.5, -1, -1) -- (-1.5, -1, 1) -- (-1.5, 1, 1) -- (-1.5, 1, -1) -- cycle;
\draw[dotted] (0, -1, -1) -- (-1.5, -1, -1);
\draw[dotted] (0, -1, 1) -- (-1.5, -1, 1);
\draw[dotted] (0, 1, -1) -- (-1.5, 1, -1);
\draw[dotted] (0, 1, 1) -- (-1.5, 1, 1);
\draw[line width=5] (-1.5, 0.5, 0) .. controls (-0.5, 0.5, 0) and (-0.5, -0.5, 0) .. (-1.5, -0.5, 0);
\end{scope}
\end{tikzpicture}
}}
\;\;=\;\;
\sum_{\mu \in \{\pm\}}\,(-A^2)^{\frac{\mu}{2}}
\vcenter{\hbox{
\tdplotsetmaincoords{20}{30}
\begin{tikzpicture}[tdplot_main_coords, rotate=30]
\begin{scope}[scale = 1.0, tdplot_main_coords]
\fill[color=lightgray, opacity=0.3] (0, -1, -1) -- (0, -1, 1) -- (0, 1, 1) -- (0, 1, -1) -- cycle;
\draw[ultra thick, orange, ->] (0, -1, 0) -- (0, 0, 0);
\draw[ultra thick, orange] (0, -1, 0) -- (0, 1, 0);
\draw[dotted] (0, -1, -1) -- (0, -1, 1) -- (0, 1, 1) -- (0, 1, -1) -- cycle;
\draw[dotted] (-1.5, -1, -1) -- (-1.5, -1, 1) -- (-1.5, 1, 1) -- (-1.5, 1, -1) -- cycle;
\draw[dotted] (0, -1, -1) -- (-1.5, -1, -1);
\draw[dotted] (0, -1, 1) -- (-1.5, -1, 1);
\draw[dotted] (0, 1, -1) -- (-1.5, 1, -1);
\draw[dotted] (0, 1, 1) -- (-1.5, 1, 1);
\draw[line width=5] (0, 0.5, 0) -- (-1.5, 0.5, 0);
\draw[line width=5] (0, -0.5, 0) -- (-1.5, -0.5, 0);
\node[anchor = west] at (0, 0.5, 0) {$\mu$};
\node[anchor = west] at (0, -0.5, 0) {$-\mu$};
\end{scope}
\end{tikzpicture}
}}
\;,
\\
\vcenter{\hbox{
\begin{tikzpicture}
\fill[top color=black, bottom color=white] (-0.085, 1) to[out=-90, in=100] (0, 0) to[out=80, in=-90] (0.085, 1)--cycle;
\fill[bottom color=black, top color=white] (-0.085, -1) to[out=90, in=-100] (0, 0) to[out=-80, in=90] (0.085, -1)--cycle;
\draw[line width=1] (-0.075, -1) to[out=90, in=-90] (0.075, 1);
\draw[dotted] (0,0) circle (1);
\end{tikzpicture}
}}
\;\;=\;\;
(-A^3)^{\frac{1}{2}}\;
\vcenter{\hbox{
\begin{tikzpicture}
\draw[dotted] (0,0) circle (1);
\draw[line width=5] (0, -1) -- (0, 1);
\end{tikzpicture}
}}
\;.
\end{gather*}
\end{defn}

\begin{rmk}\label{rmk:OrientationReversingSymmetry}
Notice that, in the definition of $\Sk(Y,\Gamma)$, if we reverse the orientation of $\Gamma$ while flipping all the states, then all the skein relations are preserved. 
This means any property of $\Sk(Y,\Gamma)$ must be preserved under this symmetry. 
\end{rmk}

\begin{defn}
Let $\Sigma$ be a bordered, punctured surface, and let $P \subset \partial \Sigma$ be a finite collection of points. 

The \emph{(stated $\mathrm{SL}_2$) skein algebra} $\SkAlg(\Sigma, P)$ is defined to be
\[
\SkAlg(\Sigma, P) := \Sk(\Sigma \times I, P\times I),
\] 
where $P\times I$ is oriented in the positive direction of the interval $I$. 

 $\SkAlg(\Sigma, P)$ has a natural algebra structure given by stacking along the $I$-direction.\footnote{In our convention, stacking on top (i.e.\ adding a tangle whose $I$-coordinate is bigger than the rest) corresponds to \emph{left} multiplication.} 
\end{defn}

When $\partial \Sigma$ is a disjoint union of intervals (e.g.\ when $\Sigma$ is an $n$-gon), we will often choose $P$ to be a collection of points in $\partial \Sigma$, one for each interval, 
and in such cases we will often omit $P$ from the notation and just write $\SkAlg(\Sigma) = \SkAlg(\Sigma, P)$; it should be clear from the context. 
This notation agrees with the usual definition of stated skein algebras \cite{CL}. 

\begin{eg}
An \emph{$n$-gon}, denoted by $D_n$, is the disk $D^2$ with $n$ marked points on the boundary. 
Its skein algebra $\SkAlg(D_n)$ is by definition the skein module of the $n$-gonal cylinder with $n$ vertical boundary markings. 
We give an example for $n=6$ below:
\[
\includestandalone[scale=1.0,valign=c]{figures/hexagonal_cylinder}
\;.
\]
\end{eg}

\begin{rmk}\label{rmk:SkeinModuleIsGraded}
The skein module $\Sk(Y,\Gamma)$ is naturally $\mathbb{Z}^{E(\Gamma)}$-graded; the grading is given by the sum of states (i.e.\ signs) on each edge of $\Gamma$. 
Likewise, $\SkAlg(\Sigma, P)$ is a $\mathbb{Z}^{P}$-graded algebra. 
\end{rmk}

\begin{rmk}
Note, if we use the opposite orientation of $P\times I$, then we get the opposite algebra $\SkAlg(\Sigma, P)^{\mathrm{op}}$. 

Recall from Remark \ref{rmk:OrientationReversingSymmetry} that reversing the orientation of the boundary marking while flipping all the states is a symmetry of skein modules. 
Therefore, 
$\SkAlg(\Sigma, P)$ is isomorphic to $\SkAlg(\Sigma, P)^{\mathrm{op}}$, 
with one isomorphism given by
\begin{align*}
\SkAlg(\Sigma, P) &\rightarrow \SkAlg(\Sigma, P)^{\mathrm{op}} \\
[L_{\vec{\epsilon}}] &\mapsto [L_{-\vec{\epsilon}}],
\end{align*}
where $L_{\vec{\epsilon}}$ is any stated tangle in $\SkAlg(\Sigma, P)$, and 
$L_{-\vec{\epsilon}}$ denotes the same tangle but with all the states flipped. 

In fact, there is a whole family of such isomorphisms, as we can rescale the map in each graded component. 
As we will see later, especially in Theorem \ref{thm:SplittingMap}, there is a natural choice of such rescaling, which is to rescale $[L]$ by $(-A^2)^{\frac{\epsilon}{4}}$ for each boundary state $\epsilon \in \{\pm 1\}$ of $L$. 
We summarize this discussion in the following proposition: 
\begin{prop}[Self-duality]\label{prop:SelfDuality}
The skein algebra $\SkAlg(\Sigma, P)$ is isomorphic to its opposite algebra, via the following isomorphism: 
\begin{align}\label{eq:duality}
\dual : \SkAlg(\Sigma, P) &\rightarrow \SkAlg(\Sigma, P)^{\mathrm{op}} \\
(-A^2)^{\frac{\Sigma\vec{\epsilon}}{4}}[L_{\vec{\epsilon}}] &\mapsto (-A^2)^{-\frac{\Sigma\vec{\epsilon}}{4}}[L_{-\vec{\epsilon}}], \nonumber
\end{align}
where $\Sigma\vec{\epsilon}$ denotes the total sum of all the boundary states of $L_{\vec{\epsilon}}$. 
\end{prop}
\end{rmk}

\begin{rmk}
Adding a puncture to a 3-manifold doesn't change its skein module, as we can always isotope any isotopy of a tangle away from that puncture. 
Using this fact, sometimes we will freely add or remove punctures from a 3-manifold. 
Note, this also means that any theorem on skein modules of ideally triangulated 3-manifolds holds also for (non-ideally) triangulated 3-manifolds, as any non-ideal vertex of a triangulation can be made ideal by adding a puncture there. 
\end{rmk}

\begin{rmk}
While in this paper we focus on stated $\mathrm{SL}_2$ skein modules, the notion of stated skein modules can be generalized to any quantum group. 
See e.g. \cite{LS} for explicit descriptions of stated $\mathrm{SL}_N$ skein algebras of surfaces. 

A slightly different approach toward cutting and gluing of skein modules is via factorization homology and skein categories, which can be defined for any ribbon category; see \cite{BBJ, Cooke}. 
The two approaches are known to be equivalent in the case of $\mathrm{SL}_2$ skein algebras of surfaces \cite{Haioun}
and are expected to be equivalent in general. 
Roughly, stated skein modules are a ``strictification'' of skein categories in the sense that we require the ends of a ribbon tangle (or, more generally, a ribbon graph) to lie on a 1-skeleton which is a strong deformation retraction of the boundary $\partial Y$. 
\end{rmk}

\subsection{Module structure}
In our study of stated skein modules, especially to be able to cut a 3-manifold into elementary pieces, it is crucial to view $\Sk(Y, \Gamma)$ as a (bi)module over tensor products of stated skein algebras at the vertices of $\Gamma$. 
In this subsection, we explain why $\Sk(Y, \Gamma)$ has this natural (bi)module structure. 

Let $V(\Gamma)$, $V(\Gamma)^{+}$, and $V(\Gamma)^{-}$ be respectively the set of vertices of $\Gamma$, the set of sinks, and the set of sources. 

We claim that, for each $v\in V(\Gamma)^+$ (resp. for each $v\in V(\Gamma)^-$), $\Sk(Y, \Gamma)$ has a natural left (resp. right) $\SkAlg(D_{\deg v})$-module structure. 
This can be seen easily from the following figure: 
\[
\vcenter{\hbox{
\includestandalone[scale=0.7]{figures/module_structureLHS}
}}
\;\approx\;
\vcenter{\hbox{
\includestandalone[scale=0.7]{figures/module_structureRHS}
}}.
\]
In this figure, we are showing an example of a vertex which is a sink of degree $3$, viewed from outside of the 3-manifold $Y$.\footnote{By viewing from ``outside of the 3-manifold'', we mean locally identifying the 3-manifold with the half-space in $\mathbb{R}^3$, and viewing the boundary of the half-space from outside of the half-space (i.e.\ from the other half-space).} 
Whenever we have a vertex of degree $n$, we can first delete a small neighborhood of the vertex, which does not change the skein module, and then ``pull the vertex out'' to get a homeomorphic picture with a cylinder over $D_n$, as in the figure above. 

Once we have such a cylinder over $D_n$, it is immediate that we have a $\SkAlg(D_n)$-module structure on $\Sk(Y, \Gamma)$; 
the action of a stated skein in $\SkAlg(D_n)$ is simply given by stacking it on top of the cylinder. 
From the orientation of $\Gamma$, it is clear that this is a left action if the vertex is a sink, and a right action if the vertex is a source. 

Our discussion so far can be summarized as: 
\begin{prop}[Bimodule structure]\label{prop:skein_alg_structure}
The stated skein module $\Sk(Y, \Gamma)$ has a natural 
$\otimes_{v \in V(\Gamma)^{+}}\SkAlg(D_{\deg v})$-$\otimes_{w \in V(\Gamma)^{-}}\SkAlg(D_{\deg w})$-bimodule structure. 
\end{prop}

\begin{eg}\label{eg:SkAlg(D_n)AsABimod}
Let $B$ be a 3-ball with boundary marking $\Gamma = \Gamma_n$ whose vertices are the north pole and the south pole and whose edges are $n$ geodesic curves oriented from the south pole to the north pole. 
Then $(B, \Gamma_n)$ is topologically equivalent to $D_n \times I$ with the standard boundary marking, so $\Sk(B, \Gamma_n)$ is isomorphic to $\SkAlg(D_n)$. 
In this case, the bimodule structure is just 
$\SkAlg(D_n)$ as a regular $\SkAlg(D_n)$-$\SkAlg(D_n)$ bimodule. 
\end{eg}

\begin{rmk}
The bimodule structure on $\Sk(Y,\Gamma)$ respects the grading mentioned in Remark \ref{rmk:SkeinModuleIsGraded}. 
That is, $\Sk(Y,\Gamma)$ is a graded bimodule. 
\end{rmk}

\begin{rmk}
One can think of these module structures as something giving new ``skein relations'' associated to sliding an end point of a tangle across a vertex of the boundary marking $\Gamma$. 
For instance, we have
\begin{align*}
\vcenter{\hbox{
\begin{tikzpicture}
\coordinate (o) at (0, 0);
\coordinate (a) at ({sqrt(3)}, -1);
\coordinate (a1) at ({sqrt(3)/3}, -1/3);
\coordinate (b) at (-{sqrt(3)}, -1);
\coordinate (b1) at (-{sqrt(3)/3}, -1/3);
\draw[line width=5] (b1) to[out=-60, in=90](0, -2);
\begin{scope}[very thick,decoration={
    markings,
    mark=at position 0.5 with {\arrow{>}}}
    ] 
    \draw[postaction={decorate}, orange] (a) -- (o);
    \draw[postaction={decorate}, orange] (b) -- (o);
\end{scope}
\filldraw[orange] (o) circle (0.05);
\node[anchor = south] at (b1) {$\mu$};
\end{tikzpicture}
}}
\;\;&=\;\;
\sum_{\nu \in \{\pm\}}
(-A^2)^{-\frac{\nu}{2}}
\vcenter{\hbox{
\begin{tikzpicture}
\coordinate (o) at (0, 0);
\coordinate (a) at ({sqrt(3)}, -1);
\coordinate (a1) at ({sqrt(3)/3}, -1/3);
\coordinate (a2) at ({sqrt(3)*2/3}, -2/3);
\coordinate (b) at (-{sqrt(3)}, -1);
\coordinate (b1) at (-{sqrt(3)/3}, -1/3);
\draw[line width=5] (a1) arc (-30:-150:2/3);
\draw[line width=5] (a2) to[out=-120, in=90](0, -2);
\begin{scope}[very thick,decoration={
    markings,
    mark=at position 0.5 with {\arrow{>}}}
    ] 
    \draw[postaction={decorate}, orange] (a) -- (o);
    \draw[postaction={decorate}, orange] (b) -- (o);
\end{scope}
\filldraw[orange] (o) circle (0.05);
\node[anchor = south] at (a1) {$-\nu$};
\node[anchor = south] at (a2) {$\nu$};
\node[anchor = south] at (b1) {$\mu$};
\end{tikzpicture}
}}\\
\;\;&=\;\;
\sum_{\nu \in \{\pm\}}
\qty(
(-A^2)^{-\frac{\nu}{2}}
\vcenter{\hbox{
\begin{tikzpicture}
\coordinate (o) at (0, 0);
\coordinate (a) at ({sqrt(3)}, -1);
\coordinate (a1) at ({sqrt(3)/3}, -1/3);
\coordinate (a2) at ({sqrt(3)*2/3}, -2/3);
\coordinate (b) at (-{sqrt(3)}, -1);
\coordinate (b1) at (-{sqrt(3)/3}, -1/3);
\draw[line width=5] (a1) arc (-30:-150:2/3);
\begin{scope}[very thick,decoration={
    markings,
    mark=at position 0.5 with {\arrow{>}}}
    ] 
    \draw[postaction={decorate}, orange] (a) -- (o);
    \draw[postaction={decorate}, orange] (b) -- (o);
\end{scope}
\filldraw[orange] (o) circle (0.05);
\node[anchor = south] at (a1) {$-\nu$};
\node[anchor = south] at (b1) {$\mu$};
\end{tikzpicture}
}}
)
\cdot
\vcenter{\hbox{
\begin{tikzpicture}
\coordinate (o) at (0, 0);
\coordinate (a) at ({sqrt(3)}, -1);
\coordinate (a1) at ({sqrt(3)/3}, -1/3);
\coordinate (a2) at ({sqrt(3)*2/3}, -2/3);
\coordinate (b) at (-{sqrt(3)}, -1);
\coordinate (b1) at (-{sqrt(3)/3}, -1/3);
\draw[line width=5] (a1) to[out=-120, in=90](0, -2);
\begin{scope}[very thick,decoration={
    markings,
    mark=at position 0.5 with {\arrow{>}}}
    ] 
    \draw[postaction={decorate}, orange] (a) -- (o);
    \draw[postaction={decorate}, orange] (b) -- (o);
\end{scope}
\filldraw[orange] (o) circle (0.05);
\node[anchor = south] at (a2) {$\nu$};
\end{tikzpicture}
}}
\;,
\end{align*}
where the term in parenthesis is considered as an element of $\SkAlg(D_{\deg v})$, with $v$ being the vertex of $\Gamma$ involved in this figure. 
Therefore, even though the end points of tangles are not allowed to freely slide across vertices of the boundary marking, 
we still get ``skein relations'' under such sliding moves, with coefficients in the skein algebras $\SkAlg(D_{\deg v})$. 
\end{rmk}

\subsection{Splitting map}\label{subsec:SplittingMap}
In this subsection, we formulate a splitting map for our stated skein modules and prove its well-definedness. 
Our splitting map is a generalization of that of \cite{CL}. 

\begin{defn}\label{defn:CombinatorialFoliation}
Let $\Sigma$ be a bordered, punctured surface. 
A \emph{combinatorial foliation} of $\Sigma$ is a decomposition of $\Sigma$ into pieces, each of which is homeomorphic to the \emph{elementary quadrilateral}: a quadrilateral with 2 opposite vertices removed and a diagonal marking:
\[
\vcenter{\hbox{
\begin{tikzpicture}
\filldraw[very thick, draw=darkgray, fill=lightgray] (-1, 0) -- (0, 1) -- (1, 0) -- (0, -1) -- cycle;
\filldraw[draw=black, fill=white] (1, 0) circle (0.1);
\filldraw[draw=black, fill=white] (-1, 0) circle (0.1);
\draw[very thick, orange] (0, -1) -- (0, 1);
\draw[very thick, orange, ->] (0, -1) -- (0, 0);
\filldraw[orange] (0, 1) circle (0.05);
\filldraw[orange] (0, -1) circle (0.05);
\end{tikzpicture}
}}.
\]
That is, it is a presentation of $\Sigma$ as some number of copies of elementary quadrilaterals, with some pairs of edges glued together. 
When two edges are glued together, we require that the orange vertices adjacent to the edges must both be sources or must both be sinks. 
We also assume that no two edges of the same elementary quadrilateral are glued together. 
\end{defn}

\begin{eg}
Here are some examples of combinatorially foliated surfaces: 
\begin{gather*}
\vcenter{\hbox{
\begin{tikzpicture}
\filldraw[very thick, draw=darkgray, fill=lightgray] (-1, 0) -- (0, 1) -- (1, 0) -- (0, -1) -- cycle;
\filldraw[draw=black, fill=white] (1, 0) circle (0.1);
\filldraw[draw=black, fill=white] (-1, 0) circle (0.1);
\draw[very thick, orange] (0, -1) -- (0, 1);
\draw[very thick, orange, ->] (0, -1) -- (0, 0);
\filldraw[orange] (0, 1) circle (0.05);
\filldraw[orange] (0, -1) circle (0.05);
\end{tikzpicture}
}}
\;,\quad
\vcenter{\hbox{
\begin{tikzpicture}[scale=1.5]
\filldraw[very thick, draw=darkgray, fill=lightgray] (0, 0) -- (1, 0) -- (1, 1) -- (0, 1) -- (-1, 1) -- (-1, 0) -- cycle;
\draw[very thick, darkgray] (0, 0) -- (0, 1);
\filldraw[draw=black, fill=white] (0, 0) circle (0.1);
\filldraw[draw=black, fill=white] (1, 1) circle (0.1);
\filldraw[draw=black, fill=white] (-1, 1) circle (0.1);
\draw[very thick, orange] (1, 0) -- (0, 1);
\draw[very thick, orange, ->] (1, 0) -- (1/2, 1/2);
\draw[very thick, orange] (-1, 0) -- (0, 1);
\draw[very thick, orange, ->] (-1, 0) -- (-1/2, 1/2);
\filldraw[orange] (0, 1) circle (0.05);
\filldraw[orange] (1, 0) circle (0.05);
\filldraw[orange] (-1, 0) circle (0.05);
\end{tikzpicture}
}}
\;,\quad
\vcenter{\hbox{
\begin{tikzpicture}[scale=1.5]
\filldraw[very thick, draw=darkgray, fill=lightgray] (-1, 0) -- (0, 1) -- (1, 0) -- (0, -1) -- cycle;
\draw[very thick, darkgray] (1, 0) -- (-1, 0);
\filldraw[draw=black, fill=white] (1, 0) circle (0.1);
\filldraw[draw=black, fill=white] (-1, 0) circle (0.1);
\draw[very thick, orange] (0, 0) -- (0, 1);
\draw[very thick, orange, ->] (0, 0) -- (0, 1/2);
\draw[very thick, orange] (0, 0) -- (0, -1);
\draw[very thick, orange, ->] (0, 0) -- (0, -1/2);
\filldraw[orange] (0, 1) circle (0.05);
\filldraw[orange] (0, 0) circle (0.05);
\filldraw[orange] (0, -1) circle (0.05);
\end{tikzpicture}
}}
\;,\quad
\vcenter{\hbox{
\begin{tikzpicture}[scale=1.7]
\filldraw[very thick, draw=darkgray, fill=lightgray] (0, 1) -- ({sqrt(3)/2}, -1/2) -- ({-sqrt(3)/2}, -1/2) -- cycle;
\draw[very thick, darkgray] (0, 0) -- (0, 1);
\draw[very thick, darkgray] (0, 0) -- ({sqrt(3)/2}, -1/2);
\draw[very thick, darkgray] (0, 0) -- ({-sqrt(3)/2}, -1/2);
\filldraw[draw=black, fill=white] (0, 1) circle (0.08);
\filldraw[draw=black, fill=white] ({sqrt(3)/2}, -1/2) circle (0.08);
\filldraw[draw=black, fill=white] ({-sqrt(3)/2}, -1/2) circle (0.08);
\filldraw[orange] (0, 0) circle (0.05);
\filldraw[orange] ({sqrt(3)/4}, {1/4}) circle (0.05);
\filldraw[orange] ({-sqrt(3)/4}, {1/4}) circle (0.05);
\filldraw[orange] (0, -1/2) circle (0.05);
\draw[very thick, orange] ({sqrt(3)/4}, {1/4}) -- (0, 0);
\draw[very thick, orange, ->] ({sqrt(3)/4}, {1/4}) -- ({sqrt(3)/8}, {1/8});
\draw[very thick, orange] ({-sqrt(3)/4}, {1/4}) -- (0, 0);
\draw[very thick, orange, ->] ({-sqrt(3)/4}, {1/4}) -- ({-sqrt(3)/8}, {1/8});
\draw[very thick, orange] (0, -1/2) -- (0, 0);
\draw[very thick, orange, ->] (0, -1/2) -- (0, -1/4);
\end{tikzpicture}
}}
\;.
\end{gather*}
\end{eg}

\begin{rmk} \label{rem:leaf_space}
We call such a decomposition a combinatorial foliation because one can construct a 1-dimensional foliation of $\Sigma$ using the foliation on the elementary quadrilateral depicted in the following figure: 
\[
\vcenter{\hbox{
\begin{tikzpicture}[scale=1.3]
\filldraw[very thick, draw=darkgray, fill=lightgray] (-1, 0) -- (0, 1) -- (1, 0) -- (0, -1) -- cycle;
\draw (1, 0) .. controls (0, 1) and (0, 1) .. (-1, 0);
\draw (1, 0) .. controls (0.2, 0.8) and (-0.2, 0.8) .. (-1, 0);
\draw (1, 0) .. controls (0.4, 0.6) and (-0.4, 0.6) .. (-1, 0);
\draw (1, 0) .. controls (0.6, 0.4) and (-0.6, 0.4) .. (-1, 0);
\draw (1, 0) .. controls (0.8, 0.2) and (-0.8, 0.2) .. (-1, 0);
\draw (1, 0) .. controls (1.0, 0.0) and (-1.0, 0.0) .. (-1, 0);
\draw (1, 0) .. controls (0.8, -0.2) and (-0.8, -0.2) .. (-1, 0);
\draw (1, 0) .. controls (0.6, -0.4) and (-0.6, -0.4) .. (-1, 0);
\draw (1, 0) .. controls (0.4, -0.6) and (-0.4, -0.6) .. (-1, 0);
\draw (1, 0) .. controls (0.2, -0.8) and (-0.2, -0.8) .. (-1, 0);
\draw (1, 0) .. controls (0, -1) and (0, -1) .. (-1, 0);
\filldraw[draw=black, fill=white] (1, 0) circle (0.1);
\filldraw[draw=black, fill=white] (-1, 0) circle (0.1);
\draw[very thick, orange] (0, -1) -- (0, 1);
\draw[very thick, orange, ->] (0, -1) -- (0, 0);
\filldraw[orange] (0, 1) circle (0.05);
\filldraw[orange] (0, -1) circle (0.05);
\end{tikzpicture}
}}.
\]
Note, the orange graph can be thought of as the \emph{leaf space} (i.e.\ the moduli space of leaves of the foliation), and the edges of the quadrilaterals are exactly the \emph{singular leaves} of the foliation. 
\end{rmk}

\begin{rmk}
Let $P$ be any finite set of points in the elementary quadrilateral, equipped with the foliation as in Remark \ref{rem:leaf_space}. 
Assume that no leaf of the foliation contains more than 1 point of $P$. 
Then, the projection of $P$ to the leaf space, which is just an oriented interval in this case, is injective and therefore induces a total ordering on $P$: for any $p, p' \in P$, we define $p < p'$ if the projection of $p'$ to the leaf space is closer to the sink vertex than that of $p$. 
We call this ordering the \emph{height ordering}. 
\end{rmk}

\begin{thm}[Splitting homomorphism]\label{thm:SplittingMap}
Let $(Y_1, \Gamma_1)$ and $(Y_2, \Gamma_2)$ be boundary marked (bordered punctured) 3-manifolds. 
Suppose that $\Sigma_1 \subset \partial Y_1$ and $\Sigma_2 \subset \partial Y_2$, along with their markings, are homeomorphic combinatorially foliated surfaces of opposite orientations, so that we can glue $Y_1$ and $Y_2$ by identifying $\Sigma_1$ with $\Sigma_2$. 
Let $Y = Y_1 \cup_{\Sigma} Y_2$ be the glued 3-manifold, with boundary marking $\Gamma = (\Gamma_1 \cup \Gamma_2) \setminus \mathrm{int}\;\Sigma $. 

Then, there is an $R$-module homomorphism
\begin{align*}
\sigma \;:\; \Sk(Y, \Gamma) &\rightarrow \Sk(Y_1, \Gamma_1) \;\overline{\otimes}\; \Sk(Y_2, \Gamma_2) \\
[L] &\mapsto \qty[
\sum_{\vec{\epsilon} \;\in\; \{\substack{\mathrm{compatible}\\ \mathrm{states}}\}}  [L_1^{\vec{\epsilon}}] \otimes [L_2^{\vec{\epsilon}}] 
],
\end{align*}
where the \emph{reduced tensor product} $\Sk(Y_1, \Gamma_1) \;\overline{\otimes}\; \Sk(Y_2, \Gamma_2)$ denotes the quotient of the usual tensor product $\Sk(Y_1, \Gamma_1) \otimes \Sk(Y_2, \Gamma_2)$ (as $R$-modules) by the following relations: 
\begin{enumerate}[label= (S\arabic*)]
\item\label{item:S+} For each internal edge $e$ of $\Sigma$ adjacent to a sink, we have the following relations among left actions: 
\begin{align*}
&
\text{the left action of }\;
(-A^2)^{\frac{\mu + \nu}{4}}
\vcenter{\hbox{
\begin{tikzpicture}[scale=2]
\filldraw[very thick, draw=darkgray, fill=lightgray] (0, 0) -- (1, 0) -- (1, 1) -- (0, 1) -- (-1, 1) -- (-1, 0) -- cycle;
\draw[very thick, darkgray] (0, 0) -- (0, 1);
\filldraw[draw=black, fill=white] (0, 0) circle (0.07);
\filldraw[draw=black, fill=white] (1, 1) circle (0.07);
\filldraw[draw=black, fill=white] (-1, 1) circle (0.07);
\draw[line width=3] (1/4, 3/4) to[out=-135, in=-45] (-1/4, 3/4);
\node[anchor=south east] at (-1/4, 3/4){$\mu$};
\node[anchor=south west] at (1/4, 3/4){$\nu$};
\node[anchor=west] at (0, 1/3){$e$};
\draw[very thick, orange] (1, 0) -- (0, 1);
\draw[very thick, orange, ->] (1, 0) -- (1/2, 1/2);
\draw[very thick, orange] (-1, 0) -- (0, 1);
\draw[very thick, orange, ->] (-1, 0) -- (-1/2, 1/2);
\filldraw[orange] (0, 1) circle (0.05);
\filldraw[orange] (1, 0) circle (0.05);
\filldraw[orange] (-1, 0) circle (0.05);
\end{tikzpicture}
}}
\text{ on }\Sk(Y_1, \Gamma_1)\\
=\;\;&\text{the left action of }\;
(-A^2)^{-\frac{\mu + \nu}{4}}
\vcenter{\hbox{
\begin{tikzpicture}[scale=2]
\filldraw[very thick, draw=darkgray, fill=lightgray] (0, 0) -- (1, 0) -- (1, 1) -- (0, 1) -- (-1, 1) -- (-1, 0) -- cycle;
\draw[very thick, darkgray] (0, 0) -- (0, 1);
\filldraw[draw=black, fill=white] (0, 0) circle (0.07);
\filldraw[draw=black, fill=white] (1, 1) circle (0.07);
\filldraw[draw=black, fill=white] (-1, 1) circle (0.07);
\draw[line width=3] (1/4, 3/4) to[out=-135, in=-45] (-1/4, 3/4);
\node[anchor=south east] at (-1/4, 3/4){$-\mu$};
\node[anchor=south west] at (1/4, 3/4){$-\nu$};
\node[anchor=west] at (0, 1/3){$e$};
\draw[very thick, orange] (1, 0) -- (0, 1);
\draw[very thick, orange, ->] (1, 0) -- (1/2, 1/2);
\draw[very thick, orange] (-1, 0) -- (0, 1);
\draw[very thick, orange, ->] (-1, 0) -- (-1/2, 1/2);
\filldraw[orange] (0, 1) circle (0.05);
\filldraw[orange] (1, 0) circle (0.05);
\filldraw[orange] (-1, 0) circle (0.05);
\end{tikzpicture}
}}
\text{ on }\Sk(Y_2, \Gamma_2). 
\end{align*}

\item\label{item:S-} Likewise, for each internal edge $e$ of $\Sigma$ adjacent to a source, we have the following relations among right actions: 
\begin{align*}
&
\text{the right action of }\;
(-A^2)^{-\frac{\mu + \nu}{4}}
\vcenter{\hbox{
\begin{tikzpicture}[scale=2]
\filldraw[very thick, draw=darkgray, fill=lightgray] (0, 0) -- (1, 0) -- (1, 1) -- (0, 1) -- (-1, 1) -- (-1, 0) -- cycle;
\draw[very thick, darkgray] (0, 0) -- (0, 1);
\filldraw[draw=black, fill=white] (0, 0) circle (0.07);
\filldraw[draw=black, fill=white] (1, 1) circle (0.07);
\filldraw[draw=black, fill=white] (-1, 1) circle (0.07);
\draw[line width=3] (1/4, 3/4) to[out=-135, in=-45] (-1/4, 3/4);
\node[anchor=south east] at (-1/4, 3/4){$\mu$};
\node[anchor=south west] at (1/4, 3/4){$\nu$};
\node[anchor=west] at (0, 1/3){$e$};
\draw[very thick, orange] (1, 0) -- (0, 1);
\draw[very thick, orange, ->] (0, 1) -- (1/2, 1/2);
\draw[very thick, orange] (-1, 0) -- (0, 1);
\draw[very thick, orange, ->] (0, 1) -- (-1/2, 1/2);
\filldraw[orange] (0, 1) circle (0.05);
\filldraw[orange] (1, 0) circle (0.05);
\filldraw[orange] (-1, 0) circle (0.05);
\end{tikzpicture}
}}
\text{ on }\Sk(Y_1, \Gamma_1) \\
=\;\;&\text{the right action of }\;
(-A^2)^{\frac{\mu + \nu}{4}}
\vcenter{\hbox{
\begin{tikzpicture}[scale=2]
\filldraw[very thick, draw=darkgray, fill=lightgray] (0, 0) -- (1, 0) -- (1, 1) -- (0, 1) -- (-1, 1) -- (-1, 0) -- cycle;
\draw[very thick, darkgray] (0, 0) -- (0, 1);
\filldraw[draw=black, fill=white] (0, 0) circle (0.07);
\filldraw[draw=black, fill=white] (1, 1) circle (0.07);
\filldraw[draw=black, fill=white] (-1, 1) circle (0.07);
\draw[line width=3] (1/4, 3/4) to[out=-135, in=-45] (-1/4, 3/4);
\node[anchor=south east] at (-1/4, 3/4){$-\mu$};
\node[anchor=south west] at (1/4, 3/4){$-\nu$};
\node[anchor=west] at (0, 1/3){$e$};
\draw[very thick, orange] (1, 0) -- (0, 1);
\draw[very thick, orange, ->] (0, 1) -- (1/2, 1/2);
\draw[very thick, orange] (-1, 0) -- (0, 1);
\draw[very thick, orange, ->] (0, 1) -- (-1/2, 1/2);
\filldraw[orange] (0, 1) circle (0.05);
\filldraw[orange] (1, 0) circle (0.05);
\filldraw[orange] (-1, 0) circle (0.05);
\end{tikzpicture}
}}
\text{ on }\Sk(Y_2, \Gamma_2). 
\end{align*}
\end{enumerate}
\end{thm}

\begin{rmk}\label{rmk:ModuleQuotient}
Let us make a brief comment on what we mean by identifying left actions. 
Let $A_i$ be a unital associative algebra and $M_i$ a left $A_i$-module, for $i = 1, 2$. 
Then by identifying the action of $\alpha_1 \in A_1$ with the action of $\alpha_2 \in A_2$ we mean to take the quotient
\[
M_1 \;\overline{\otimes}\; M_2 
=
(M_1 \otimes M_2)/(\alpha_1\otimes 1 - 1\otimes \alpha_2)(M_1 \otimes M_2)
\]
as an $R$-module. 
That is, elements of $M_1 \;\overline{\otimes}\; M_2$ are cosets of the form
\[
m_1 \otimes m_2 + (\alpha_1\otimes 1 - 1\otimes \alpha_2)(M_1 \otimes M_2),\quad m_1 \in M_1,\; m_2\in M_2.
\]
Identifying right actions should be interpreted in a similar way. 
\end{rmk}

\begin{rmk}
Using the isomorphism in Proposition \ref{prop:SelfDuality}, 
we can think of any left (resp. right) $\SkAlg(D_n)$-module structure -- which is equivalent to a right (resp. left) $\SkAlg(D_n)^{\mathrm{op}}$-module structure -- as a right (resp. left) $\SkAlg(D_n)$-module structure. 
Then, the relations \ref{item:S+} and \ref{item:S-} are nothing but the usual relations for tensor products of left and right modules over the same algebra. 
\end{rmk}


\begin{proof}[Proof of Theorem \ref{thm:SplittingMap}]
Let $L$ be a stated tangle, representing an element $[L]$ of $\Sk(Y, \Gamma)$. 
We will say that $L$ is in \emph{general position} with respect to the foliation of $\Sigma$ if
\begin{enumerate}
    \item\label{item:genpos1} its intersection with $\Sigma$ is transverse, 
    \item\label{item:genpos2} at the intersection points, the foliation is transverse to the 2-dimensional tangent space to the ribbon tangle, 
    \item\label{item:genpos3} no singular leaf meets $L$, and
    \item\label{item:genpos4} no leaf of the foliation meets $L$ at 2 or more points. 
\end{enumerate}
Note that, in the moduli space of smoothly embedded tangles, generic points are in general position with respect to the foliation, and tangles which are in non-general positions are of real codimension $\geq 1$ in the moduli space. 

If $L$ is in general position near $\Sigma$ with respect to the foliation, we can isotope the part of $L$ near $\Sigma$ along the foliation so that $L \cap \Sigma$ lies in the marking $\Gamma_{\Sigma}$ of $\Sigma$. 
Then $L_1 := L \cap Y_1$ and $L_2 := L \cap Y_2$ are tangles in $(Y_1, \Gamma_1)$ and $(Y_2, \Gamma_2)$. 
We need to introduce states to the newly created ends of $L_1$ and $L_2$ in order to consider them as elements of $\Sk(Y_1, \Gamma_1)$ and $\Sk(Y_2, \Gamma_2)$, respectively. 
We say that a state on $L_1 \sqcup L_2$ is \emph{compatible} if for each intersection point $p \in L \cap \Sigma$, the state of $L_1$ at $p$ is equal to that of $L_2$ at $p$. 
For each $\vec{\epsilon} : L\cap \Sigma \rightarrow \{\pm\}$, we denote by $L_1^{\vec{\epsilon}}$ and $L_2^{\vec{\epsilon}}$ the corresponding tangles with compatible states assigned to the newly created ends. 
By summing over all possible compatible states, we get an element
\[
\widetilde{\sigma}(L) := \sum_{\vec{\epsilon} \;\in\; \{\substack{\mathrm{compatible}\\ \mathrm{states}}\}} [L_1^{\vec{\epsilon}}] \otimes [L_2^{\vec{\epsilon}}] \;\in\; \Sk(Y_1, \Gamma_1) \otimes \Sk(Y_2, \Gamma_2). 
\]
We need to study how $\widetilde{\sigma}(L)$ behaves under isotopy of $L$. 

For any generic isotopy $\{L_t\}_{t\in [0,1]}$, $L_t$ must be in general position with respect to the foliation of $\Sigma$ except for finitely many $t$, 
and for each such $t$ for which $L_t$ is in non-general position, it must violate exactly 1 of the conditions (\ref{item:genpos1})-(\ref{item:genpos4}) for it to be in general position. 
Therefore, any isotopy of $L$ is a finite composition of isotopies of the following types:
\begin{enumerate}[label= (\Roman*)]
\item\label{item:typeIisotopy} An isotopy in the class of tangles in general position with respect to the foliation of $\Sigma$. 
\item\label{item:typeIIisotopy} Half twist of the ribbon tangle near $\Sigma$ -- at some point of the isotopy, the ribbon becomes parallel to the foliation of $\Sigma$. 
\[
\vcenter{\hbox{
\begin{tikzpicture}
\draw[line width=5] (-1, 0) -- (1, 0);
\draw[very thick, orange] (0, -1) -- (0, 0);
\draw[very thick, orange] (0, 1) -- (0, 0);
\node[anchor=south] at (0, 1){$\Sigma$};
\node[anchor=south] at (-1, 0.7){$Y_1$};
\node[anchor=south] at (1, 0.7){$Y_2$};
\end{tikzpicture}
}}
\;\leftrightarrow\;
\vcenter{\hbox{
\begin{tikzpicture}
\begin{scope}[shift={(-0.5, 0)}, xscale=-0.5, rotate=90]
    \fill[right color=black, left color=white] (-0.085, 1) to[out=-90, in=100] (0, 0) to[out=80, in=-90] (0.085, 1)--cycle;
    \fill[left color=black, right color=white] (-0.085, -1) to[out=90, in=-100] (0, 0) to[out=-80, in=90] (0.085, -1)--cycle;
    \draw[line width=1] (-0.075, -1) to[out=90, in=-90] (0.075, 1);
\end{scope}
\begin{scope}[shift={(0.5, 0)}, xscale=0.5, rotate=90]
    \fill[left color=black, right color=white] (-0.085, 1) to[out=-90, in=100] (0, 0) to[out=80, in=-90] (0.085, 1)--cycle;
    \fill[right color=black, left color=white] (-0.085, -1) to[out=90, in=-100] (0, 0) to[out=-80, in=90] (0.085, -1)--cycle;
    \draw[line width=1] (-0.075, -1) to[out=90, in=-90] (0.075, 1);
\end{scope}
\draw[very thick, orange] (0, -1) -- (0, 0);
\draw[very thick, orange] (0, 1) -- (0, 0);
\node[anchor=south] at (0, 1){$\Sigma$};
\node[anchor=south] at (-1, 0.7){$Y_1$};
\node[anchor=south] at (1, 0.7){$Y_2$};
\end{tikzpicture}
}}
\]
\item\label{item:typeIIIisotopy} Height exchange -- at some point of the isotopy, there is a leaf of the foliation of $\Sigma$ which contains two points of the tangle. 
\[
\vcenter{\hbox{
\begin{tikzpicture}
\draw[line width=5] (-1, 0.4) -- (1, 0.4);
\draw[line width=5] (-1, -0.4) -- (1, -0.4);
\draw[very thick, orange] (0, -1) -- (0, 0);
\draw[very thick, orange] (0, 1) -- (0, 0);
\node[anchor=south] at (0, 1){$\Sigma$};
\node[anchor=south] at (-1, 0.7){$Y_1$};
\node[anchor=south] at (1, 0.7){$Y_2$};
\end{tikzpicture}
}}
\;\leftrightarrow\;
\vcenter{\hbox{
\begin{tikzpicture}
\draw[line width=5] (-1, 0.4) to[out=0, in=180] (0, -0.2) to[out=0, in=180] (1, 0.4);
\draw[white, line width=10] (-1, -0.4) to[out=0, in=180] (0, 0.2) to[out=0, in=180] (1, -0.4);
\draw[line width=5] (-1, -0.4) to[out=0, in=180] (0, 0.2) to[out=0, in=180] (1, -0.4);
\draw[very thick, orange] (0, -1) -- (0, 0);
\draw[very thick, orange] (0, 1) -- (0, 0);
\node[anchor=south] at (0, 1){$\Sigma$};
\node[anchor=south] at (-1, 0.7){$Y_1$};
\node[anchor=south] at (1, 0.7){$Y_2$};
\end{tikzpicture}
}}
\]
\item\label{item:typeIVisotopy} Birth or annihilation of a pair of intersection points with $\Sigma$ -- at some point of the isotopy, the tangle becomes tangent to $\Sigma$. 
\[
\vcenter{\hbox{
\begin{tikzpicture}
\draw[line width=5] (-1, 0.4) -- (-0.8, 0.4) to[out=0, in=90] (-0.4, 0) to[out=-90, in=0] (-0.8, -0.4) -- (-1, -0.4);
\draw[very thick, orange] (0, -1) -- (0, 0);
\draw[very thick, orange] (0, 1) -- (0, 0);
\node[anchor=south] at (0, 1){$\Sigma$};
\node[anchor=south] at (-1, 0.7){$Y_1$};
\node[anchor=south] at (1, 0.7){$Y_2$};
\end{tikzpicture}
}}
\;\leftrightarrow\;
\vcenter{\hbox{
\begin{tikzpicture}
\draw[line width=5] (-1, 0.4) -- (0, 0.4) to[out=0, in=90] (0.4, 0) to[out=-90, in=0] (0, -0.4) -- (-1, -0.4);
\draw[very thick, orange] (0, -1) -- (0, 0);
\draw[very thick, orange] (0, 1) -- (0, 0);
\node[anchor=south] at (0, 1){$\Sigma$};
\node[anchor=south] at (-1, 0.7){$Y_1$};
\node[anchor=south] at (1, 0.7){$Y_2$};
\end{tikzpicture}
}}
\]
\item\label{item:typeVisotopy} An isotopy passing the tangle through a singular leaf (i.e.\ an internal edge of the combinatorial foliation of $\Sigma$). 
\begin{equation*}
\vcenter{\hbox{
\tdplotsetmaincoords{20}{40}
\begin{tikzpicture}[tdplot_main_coords, rotate=40]
\begin{scope}[scale = 1.0, tdplot_main_coords]
\draw[dotted] (1.5, 0, 2) -- (-1.5, 0, 2);
\draw[line width=5] (0, 0.8, 0.6) -- (1.5, 0.8, 0.6);
\fill[color=lightgray, opacity=0.6] (0, 0, 0) -- (0, 0, 2) -- (0, 2, 2) -- (0, 2, 0) -- cycle;
\fill[color=lightgray, opacity=0.6] (0, 0, -2) -- (0, 0, 0) -- (0, 2, 0) -- (0, 2, -2) -- cycle;
\draw[ultra thick, orange] (0, 0, 2) -- (0, 2, 0) -- (0, 0, -2);
\draw[dotted] (1.5, 0, 2) -- (1.5, 2, 2) -- (1.5, 2, -2) -- (1.5, 0, -2) -- cycle;
\draw[dotted] (-1.5, 0, 2) -- (-1.5, 2, 2) -- (-1.5, 2, -2) -- (-1.5, 0, -2) -- cycle;
\draw[very thick, darkgray] (0, 0, 0) -- (0, 2, 0);
\draw[dotted] (1.5, 2, 2) -- (-1.5, 2, 2);
\draw[dotted] (1.5, 2, -2) -- (-1.5, 2, -2);
\draw[dotted] (1.5, 0, -2) -- (-1.5, 0, -2);
\filldraw[orange] (0, 2, 0) circle (0.07);
\draw[line width=5] (0, 0.8, 0.6) -- (-1.5, 0.8, 0.6);
\node[anchor=south] at (0, 2.1, 0){$\Sigma$};
\node[anchor=south] at (-1, 1.8, 0){$Y_1$};
\node[anchor=south] at (1, 1.8, 0){$Y_2$};
\end{scope}
\end{tikzpicture}
}}
\;\leftrightarrow\;
\vcenter{\hbox{
\tdplotsetmaincoords{20}{40}
\begin{tikzpicture}[tdplot_main_coords, rotate=40]
\begin{scope}[scale = 1.0, tdplot_main_coords]
\draw[dotted] (1.5, 0, 2) -- (-1.5, 0, 2);
\draw[line width=5] (0, 0.8, -0.6) -- (1.5, 0.8, -0.6);
\fill[color=lightgray, opacity=0.6] (0, 0, 0) -- (0, 0, 2) -- (0, 2, 2) -- (0, 2, 0) -- cycle;
\fill[color=lightgray, opacity=0.6] (0, 0, -2) -- (0, 0, 0) -- (0, 2, 0) -- (0, 2, -2) -- cycle;
\draw[ultra thick, orange] (0, 0, 2) -- (0, 2, 0) -- (0, 0, -2);
\draw[dotted] (1.5, 0, 2) -- (1.5, 2, 2) -- (1.5, 2, -2) -- (1.5, 0, -2) -- cycle;
\draw[dotted] (-1.5, 0, 2) -- (-1.5, 2, 2) -- (-1.5, 2, -2) -- (-1.5, 0, -2) -- cycle;
\draw[very thick, darkgray] (0, 0, 0) -- (0, 2, 0);
\draw[dotted] (1.5, 2, 2) -- (-1.5, 2, 2);
\draw[dotted] (1.5, 2, -2) -- (-1.5, 2, -2);
\draw[dotted] (1.5, 0, -2) -- (-1.5, 0, -2);
\filldraw[orange] (0, 2, 0) circle (0.07);
\draw[line width=5] (0, 0.8, -0.6) -- (-1.5, 0.8, -0.6);
\node[anchor=south] at (0, 2.1, 0){$\Sigma$};
\node[anchor=south] at (-1, 1.8, 0){$Y_1$};
\node[anchor=south] at (1, 1.8, 0){$Y_2$};
\end{scope}
\end{tikzpicture}
}}
\end{equation*}
\end{enumerate}
The moves \ref{item:typeIIisotopy}, \ref{item:typeIIIisotopy}, \ref{item:typeIVisotopy}, \ref{item:typeVisotopy} each correspond, respectively, to the violation of conditions (\ref{item:genpos2}), (\ref{item:genpos4}), (\ref{item:genpos1}), (\ref{item:genpos3}). 

It is clear that $\widetilde{\sigma}(L)$ does not change under the isotopies of type \ref{item:typeIisotopy}. 
In fact, it does not change under isotopies of type \ref{item:typeIIisotopy}, \ref{item:typeIIIisotopy}, and \ref{item:typeIVisotopy} either, thanks to the stated skein relations. 
This follows from the proof of well-definedness of the 2d splitting map; see \cite{Le, CL}. 
Therefore, we only need to check how $\widetilde{\sigma}(L)$ changes under isotopies of type \ref{item:typeVisotopy}. 

Thanks to invariance under type \ref{item:typeIisotopy}, \ref{item:typeIIisotopy}, \ref{item:typeIIIisotopy}, and \ref{item:typeIVisotopy} moves, we can assume that the part of the tangle that passes through an internal edge $e$ of $\Sigma$ under the type \ref{item:typeVisotopy} move is closest to the orange vertex adjacent to $e$ in the height ordering. 
We will show that the relations we get are exactly the relations \ref{item:S+} and \ref{item:S-} defining the reduced tensor product. 

Let's assume that the orange vertex adjacent to the internal edge $e$ is a sink. (The argument for the case when it's a source is exactly the same, with the replacement $(-A^2)^{\frac{1}{2}} \mapsto (-A^2)^{-\frac{1}{2}}$.)
Then, the image of the LHS of \ref{item:typeVisotopy} under $\widetilde{\sigma}$ is 
\begin{gather*}
\sum_{\mu \in \{\pm\}}
\vcenter{\hbox{
\tdplotsetmaincoords{20}{40}

}}
\;,
\end{gather*}
whereas the image of the RHS of \ref{item:typeVisotopy} under $\widetilde{\sigma}$ is
\begin{gather*}
\sum_{\nu \in \{\pm\}}
\vcenter{\hbox{
\tdplotsetmaincoords{20}{40}
%
}}
\;.
\end{gather*}
Comparing the two images, it follows immediately from relation \ref{item:S+} that they represent the same element in the quotient $\Sk(Y_1, \Gamma_1) \;\overline{\otimes}\; \Sk(Y_2, \Gamma_2)$. 
\end{proof}

Sometimes, it is also useful to split a 3-manifold $Y$ partially -- that is, splitting along a surface $\Sigma$ whose boundary may not be contained in $\partial Y$. 
Theorem \ref{thm:SplittingMap} can be easily extended to include partial splittings: 
\begin{thm}[Partial splitting homomorphism] \label{thm:partialSplittingHomomorphism}
Let $(Y', \Gamma')$ be a boundary marked 3-manifold. 
Suppose that $\Sigma_1, \Sigma_2 \subset \partial Y$, along with their boundary markings, are homeomorphic (fixing their common edges) combinatorially foliated surfaces of opposite orientations which are disjoint in the interior, so that we can glue them. 
Let $Y = Y'/(\Sigma_1 \sim \Sigma_2)$ be the glued 3-manifold, with boundary marking $\Gamma = \Gamma' \setminus \mathrm{int}\;\Sigma$. 
Then, there is a well-defined $R$-module homomorphism
\begin{align*}
\sigma \; : \; \Sk(Y, \Gamma) &\rightarrow \Sk(Y', \Gamma')/M \\
[L] &\mapsto \qty[\sum_{\vec{\epsilon} \;\in\; \{\substack{\mathrm{compatible}\\ \mathrm{states}}\}} [{L'}^{\vec{\epsilon}}]],
\end{align*}
where $M$ denotes the $R$-submodule generated by relations \ref{item:S+} and \ref{item:S-}, 
as well as the following relations, one for each edge $e\in \Sigma_1 \cap \Sigma_2$ (i.e.\ an internal edge of $\Sigma_1 \cup \Sigma_2$ that is not an internal edge of either $\Sigma_1$ or $\Sigma_2$): 
\begin{enumerate}[label= (PS\arabic*)]
\item\label{item:PS+} For each $e \in \Sigma_1 \cap \Sigma_2$ adjacent to a sink, we have the following relations among left actions: 
\begin{align*}
&
\text{the left action of }\;
\vcenter{\hbox{
\begin{tikzpicture}[scale=2]
\filldraw[very thick, draw=darkgray, fill=lightgray] (0, 0) -- (1, 0) -- (1, 1) -- (0, 1) -- (-1, 1) -- (-1, 0) -- cycle;
\draw[very thick, darkgray] (0, 0) -- (0, 1);
\filldraw[draw=black, fill=white] (0, 0) circle (0.07);
\filldraw[draw=black, fill=white] (1, 1) circle (0.07);
\filldraw[draw=black, fill=white] (-1, 1) circle (0.07);
\draw[line width=3] (1/4, 3/4) to[out=-135, in=-45] (-1/4, 3/4);
\node[anchor=south east] at (-1/4, 3/4){$\mu$};
\node[anchor=south west] at (1/4, 3/4){$\nu$};
\node[anchor=west] at (0, 1/3){$e$};
\draw[very thick, orange] (1, 0) -- (0, 1);
\draw[very thick, orange, ->] (1, 0) -- (1/2, 1/2);
\draw[very thick, orange] (-1, 0) -- (0, 1);
\draw[very thick, orange, ->] (-1, 0) -- (-1/2, 1/2);
\filldraw[orange] (0, 1) circle (0.05);
\filldraw[orange] (1, 0) circle (0.05);
\filldraw[orange] (-1, 0) circle (0.05);
\end{tikzpicture}
}}
\text{ on }\Sk(Y', \Gamma')\\
=\;\;&
\delta_{\mu,\nu}\,(-A^2)^{-\mu}
\end{align*}

\item\label{item:PS-} Likewise, for each $e \in \Sigma_1 \cap \Sigma_2$ adjacent to a source, we have the following relations among right actions: 
\begin{align*}
&
\text{the right action of }\;
\vcenter{\hbox{
\begin{tikzpicture}[scale=2]
\filldraw[very thick, draw=darkgray, fill=lightgray] (0, 0) -- (1, 0) -- (1, 1) -- (0, 1) -- (-1, 1) -- (-1, 0) -- cycle;
\draw[very thick, darkgray] (0, 0) -- (0, 1);
\filldraw[draw=black, fill=white] (0, 0) circle (0.07);
\filldraw[draw=black, fill=white] (1, 1) circle (0.07);
\filldraw[draw=black, fill=white] (-1, 1) circle (0.07);
\draw[line width=3] (1/4, 3/4) to[out=-135, in=-45] (-1/4, 3/4);
\node[anchor=south east] at (-1/4, 3/4){$\mu$};
\node[anchor=south west] at (1/4, 3/4){$\nu$};
\node[anchor=west] at (0, 1/3){$e$};
\draw[very thick, orange] (1, 0) -- (0, 1);
\draw[very thick, orange, ->] (0, 1) -- (1/2, 1/2);
\draw[very thick, orange] (-1, 0) -- (0, 1);
\draw[very thick, orange, ->] (0, 1) -- (-1/2, 1/2);
\filldraw[orange] (0, 1) circle (0.05);
\filldraw[orange] (1, 0) circle (0.05);
\filldraw[orange] (-1, 0) circle (0.05);
\end{tikzpicture}
}}
\text{ on }\Sk(Y', \Gamma') \\
=\;\;&
\delta_{\mu,\nu}\,(-A^2)^{\mu}
\;.
\end{align*}
\end{enumerate}
\end{thm}
\begin{proof}
The proof is completely analagous to that of Theorem \ref{thm:SplittingMap}. 
Whenever a stated tangle $L$ in $Y$ representing an element $[L] \in \Sk(Y,\Gamma)$ is in general position with respect to the foliation of $\Sigma$, 
then we can isotope the part of $L$ near $\Sigma$ along the foliation so that $L \cap \Sigma$ lies in the marking $\Gamma_{\Sigma}$ of $\Sigma$. 
Then, by splitting $Y$ along $\Sigma$, we get a tangle $L'$ in $Y'$, and by summing over all possible compatible states on $L' \cap (\Sigma_1 \cup \Sigma_2)$, we get an element
\[
\widetilde{\sigma}(L) := \sum_{\vec{\epsilon} \in \{\substack{\mathrm{compatible}\\ \mathrm{states}}\}} [{L'}^{\vec{\epsilon}}] \in \Sk(Y', \Gamma'). 
\]
We need to study how $\widetilde{\sigma}(L)$ behaves under isotopy of $L$. 

Any isotopy of $L$ is a finite composition of isotopies of the following types: 
\begin{enumerate}
\item An isotopy of types \ref{item:typeIisotopy}-\ref{item:typeVisotopy} listed in the proof of Theorem \ref{thm:SplittingMap}. 
\item\label{item:PartialSplittingIsotopy} An isotopy passing the tangle through an edge $e \in \partial \Sigma \setminus \partial Y$. 
\begin{equation*}
\vcenter{\hbox{
\tdplotsetmaincoords{20}{40}
\begin{tikzpicture}[tdplot_main_coords, rotate=40]
\begin{scope}[scale = 1.0, tdplot_main_coords]
\draw[dotted] (1.5, 0, 2) -- (-1.5, 0, 2);
\draw[line width=5] (0, 0.8, 0.6) -- (1.5, 0.8, 0.6);
\fill[color=lightgray, opacity=0.6] (0, 0, 0) -- (0, 0, 2) -- (0, 2, 2) -- (0, 2, 0) -- cycle;
\draw[ultra thick, orange] (0, 0, 2) -- (0, 2, 0);
\draw[dotted] (1.5, 0, 2) -- (1.5, 2, 2) -- (1.5, 2, -2) -- (1.5, 0, -2) -- cycle;
\draw[dotted] (-1.5, 0, 2) -- (-1.5, 2, 2) -- (-1.5, 2, -2) -- (-1.5, 0, -2) -- cycle;
\draw[very thick, darkgray] (0, 0, 0) -- (0, 2, 0);
\draw[dotted] (1.5, 2, 2) -- (-1.5, 2, 2);
\draw[dotted] (1.5, 2, -2) -- (-1.5, 2, -2);
\draw[dotted] (1.5, 0, -2) -- (-1.5, 0, -2);
\filldraw[orange] (0, 2, 0) circle (0.07);
\draw[line width=5] (0, 0.8, 0.6) -- (-1.5, 0.8, 0.6);
\node[anchor=south] at (0, 2.1, 0){$\Sigma$};
\node[anchor=south] at (1, 1.8, 0){$Y$};
\node[anchor=west] at (0, 0.3, 0){$e$};
\end{scope}
\end{tikzpicture}
}}
\;\leftrightarrow\;
\vcenter{\hbox{
\tdplotsetmaincoords{20}{40}
\begin{tikzpicture}[tdplot_main_coords, rotate=40]
\begin{scope}[scale = 1.0, tdplot_main_coords]
\draw[dotted] (1.5, 0, 2) -- (-1.5, 0, 2);
\draw[line width=5] (0, 0.8, -0.6) -- (1.5, 0.8, -0.6);
\fill[color=lightgray, opacity=0.6] (0, 0, 0) -- (0, 0, 2) -- (0, 2, 2) -- (0, 2, 0) -- cycle;
\draw[ultra thick, orange] (0, 0, 2) -- (0, 2, 0);
\draw[dotted] (1.5, 0, 2) -- (1.5, 2, 2) -- (1.5, 2, -2) -- (1.5, 0, -2) -- cycle;
\draw[dotted] (-1.5, 0, 2) -- (-1.5, 2, 2) -- (-1.5, 2, -2) -- (-1.5, 0, -2) -- cycle;
\draw[very thick, darkgray] (0, 0.75, 0) -- (0, 2, 0);
\draw[very thick, darkgray] (0, 0, 0) -- (0, 0.52, 0);
\draw[dotted] (1.5, 2, 2) -- (-1.5, 2, 2);
\draw[dotted] (1.5, 2, -2) -- (-1.5, 2, -2);
\draw[dotted] (1.5, 0, -2) -- (-1.5, 0, -2);
\filldraw[orange] (0, 2, 0) circle (0.07);
\draw[line width=5] (0, 0.8, -0.6) -- (-1.5, 0.8, -0.6);
\node[anchor=south] at (0, 2.1, 0){$\Sigma$};
\node[anchor=south] at (1, 1.8, 0){$Y$};
\node[anchor=west] at (0, 0.3, 0){$e$};
\end{scope}
\end{tikzpicture}
}}
\end{equation*}
\end{enumerate}
As shown in the proof of Theorem \ref{thm:SplittingMap}, 
$\widetilde{\sigma}(L)$ is invariant under isotopies of types \ref{item:typeIisotopy}-\ref{item:typeVisotopy} once we impose the relations \ref{item:S+} and \ref{item:S-}. 

We claim that $\widetilde{\sigma}(L)$ is invariant under the new type of isotopy (\ref{item:PartialSplittingIsotopy}) as well, once we further impose relations \ref{item:PS+} and \ref{item:PS-}. 
Let's assume that the orange vertex adjacent to the edge $e$ is a sink. 
(The argument for the case when it is a source is exactly the same.) 
Then, the image of the LHS of (\ref{item:PartialSplittingIsotopy}) under $\widetilde{\sigma}$ is
\[
\sum_{\mu \in \{\pm\}}
\vcenter{\hbox{
\tdplotsetmaincoords{20}{40}
\begin{tikzpicture}[tdplot_main_coords, rotate=40]
\begin{scope}[scale = 1.0, tdplot_main_coords]
\draw[dotted] (1.5, 0, 2) -- (0, 0, 2);
\fill[black] (1.5, {0.8-0.085}, 0.6) .. controls (1, {0.8-0.085}, 0.6) and (1, {0.8-0.085}, 0.6) .. (0, {0.8-0.085}, {1.2+0.085}) -- (0, {0.8+0.085}, {1.2-0.085}) .. controls (1, {0.8+0.085}, 0.6) and (1, {0.8+0.085}, 0.6) .. (1.5, {0.8+0.085}, 0.6) -- cycle;
\fill[black] (1.5, {0.8-0.085}, -0.6) .. controls (1, {0.8-0.085}, -0.6) and (1, {0.8-0.085}, -0.6) .. (0, {0.8-0.085}, {-1.2-0.085}) -- (0, {0.8+0.085}, {-1.2+0.085}) .. controls (1, {0.8+0.085}, -0.6) and (1, {0.8+0.085}, -0.6) .. (1.5, {0.8+0.085}, -0.6) -- cycle;
\fill[color=lightgray, opacity=0.6] (0, 0, 0) -- (0, 0, 2) -- (0, 2, 2) -- (0, 2, 0) -- cycle;
\fill[color=lightgray, opacity=0.6] (0, 0, -2) -- (0, 0, 0) -- (0, 2, 0) -- (0, 2, -2) -- cycle;
\draw[ultra thick, orange] (0, 0, 2) -- (0, 2, 0) -- (0, 0, -2);
\draw[thick, orange, ->] (0, 0, 2) -- (0, 1, 1);
\draw[thick, orange, ->] (0, 0, -2)-- (0, 1, -1);
\draw[dotted] (1.5, 0, 2) -- (1.5, 2, 2) -- (1.5, 2, -2) -- (1.5, 0, -2) -- cycle;
\draw[dotted] (0, 0, 2) -- (0, 2, 2) -- (0, 2, -2) -- (0, 0, -2) -- cycle;
\draw[very thick, darkgray] (0, 0, 0) -- (0, 2, 0);
\draw[dotted] (1.5, 2, 2) -- (0, 2, 2);
\draw[dotted] (1.5, 2, -2) -- (0, 2, -2);
\draw[dotted] (1.5, 0, -2) -- (0, 0, -2);
\filldraw[orange] (0, 2, 0) circle (0.07);
\node at (-0.35, 0.8, 1.2){$-\mu$};
\node at (-0.35, 0.8, -1.2){$-\mu$};
\node[anchor=south] at (1, 1.8, 0){$Y$};
\node[anchor=south] at (0, 2.1, 0.5){$\Sigma_1$};
\node[anchor=south] at (0, 2.1, -1.5){$\Sigma_2$};
\end{scope}
\end{tikzpicture}
}},
\]
whereas the image of RHS of (\ref{item:PartialSplittingIsotopy}) under $\widetilde{\sigma}$ is
\begin{gather*}
\vcenter{\hbox{
\tdplotsetmaincoords{20}{40}
\begin{tikzpicture}[tdplot_main_coords, rotate=40]
\begin{scope}[scale = 1.0, tdplot_main_coords]
\draw[dotted] (1.5, 0, 2) -- (0, 0, 2);
\fill[black] (1.5, {0.8-0.085}, 0.6) -- (0.8, {0.8-0.085}, 0) -- (0.8, {0.8+0.085}, 0) -- (1.5, {0.8+0.085}, 0.6) -- cycle;
\filldraw[draw=white, thick, fill=black] (1.5, {0.8-0.085}, -0.6) -- (0.8, {0.8-0.085}, 0) -- (0.8, {0.8+0.085}, 0) -- (1.5, {0.8+0.085}, -0.6) -- cycle;
\fill[color=lightgray, opacity=0.6] (0, 0, 0) -- (0, 0, 2) -- (0, 2, 2) -- (0, 2, 0) -- cycle;
\fill[color=lightgray, opacity=0.6] (0, 0, -2) -- (0, 0, 0) -- (0, 2, 0) -- (0, 2, -2) -- cycle;
\draw[ultra thick, orange] (0, 0, 2) -- (0, 2, 0) -- (0, 0, -2);
\draw[thick, orange, ->] (0, 0, 2) -- (0, 1, 1);
\draw[thick, orange, ->] (0, 0, -2)-- (0, 1, -1);
\draw[dotted] (1.5, 0, 2) -- (1.5, 2, 2) -- (1.5, 2, -2) -- (1.5, 0, -2) -- cycle;
\draw[dotted] (0, 0, 2) -- (0, 2, 2) -- (0, 2, -2) -- (0, 0, -2) -- cycle;
\draw[very thick, darkgray] (0, 0, 0) -- (0, 2, 0);
\draw[dotted] (1.5, 2, 2) -- (0, 2, 2);
\draw[dotted] (1.5, 2, -2) -- (0, 2, -2);
\draw[dotted] (1.5, 0, -2) -- (0, 0, -2);
\filldraw[orange] (0, 2, 0) circle (0.07);
\node[anchor=south] at (1, 1.8, 0){$Y$};
\node[anchor=south] at (0, 2.1, 0.5){$\Sigma_1$};
\node[anchor=south] at (0, 2.1, -1.5){$\Sigma_2$};
\end{scope}
\end{tikzpicture}
}}
\;\;=\;\;
\sum_{\mu,\nu \in \{\pm\}}
(-A^2)^{\frac{\mu+\nu}{2}}\;
\vcenter{\hbox{
\tdplotsetmaincoords{20}{40}
\begin{tikzpicture}[tdplot_main_coords, rotate=40]
\begin{scope}[scale = 1.0, tdplot_main_coords]
\draw[dotted] (1.5, 0, 2) -- (0, 0, 2);
\fill[black] (1.5, {0.8-0.085}, 0.6) .. controls (1, {0.8-0.085}, 0.6) and (1, {0.8-0.085}, 0.6) .. (0, {0.4-0.085}, {1.6+0.085}) -- (0, {0.4+0.085}, {1.6-0.085}) .. controls (1, {0.8+0.085}, 0.6) and (1, {0.8+0.085}, 0.6) .. (1.5, {0.8+0.085}, 0.6) -- cycle;
\fill[black] (0, {1.2+0.085}, {-0.8+0.085}) .. controls (0.05, {1.2+0.085}, 0) and (0.05, {1.2+0.085}, 0) .. (0, {1.2+0.085}, {0.8-0.085}) -- (0, {1.2-0.085}, {0.8+0.085}) .. controls (0.05, {1.2-0.085}, 0) and (0.05, {1.2-0.085}, 0) .. (0, {1.2-0.085}, {-0.8-0.085}) -- cycle;
\fill[black] (1.5, {0.8-0.085}, -0.6) .. controls (1, {0.8-0.085}, -0.6) and (1, {0.8-0.085}, -0.6) .. (0, {0.8-0.085}, {-1.2-0.085}) -- (0, {0.8+0.085}, {-1.2+0.085}) .. controls (1, {0.8+0.085}, -0.6) and (1, {0.8+0.085}, -0.6) .. (1.5, {0.8+0.085}, -0.6) -- cycle;
\fill[color=lightgray, opacity=0.6] (0, 0, 0) -- (0, 0, 2) -- (0, 2, 2) -- (0, 2, 0) -- cycle;
\fill[color=lightgray, opacity=0.6] (0, 0, -2) -- (0, 0, 0) -- (0, 2, 0) -- (0, 2, -2) -- cycle;
\draw[ultra thick, orange] (0, 0, 2) -- (0, 2, 0) -- (0, 0, -2);
\draw[thick, orange, ->] (0, 0, 2) -- (0, 1, 1);
\draw[thick, orange, ->] (0, 0, -2)-- (0, 1, -1);
\draw[dotted] (1.5, 0, 2) -- (1.5, 2, 2) -- (1.5, 2, -2) -- (1.5, 0, -2) -- cycle;
\draw[dotted] (0, 0, 2) -- (0, 2, 2) -- (0, 2, -2) -- (0, 0, -2) -- cycle;
\draw[very thick, darkgray] (0, 0, 0) -- (0, 2, 0);
\draw[dotted] (1.5, 2, 2) -- (0, 2, 2);
\draw[dotted] (1.5, 2, -2) -- (0, 2, -2);
\draw[dotted] (1.5, 0, -2) -- (0, 0, -2);
\filldraw[orange] (0, 2, 0) circle (0.07);
\node at (-0.35, 0.4, 1.6){$-\mu$};
\node at (-0.35, 0.8, -1.2){$-\nu$};
\node at (-0.3, 1.3, 0.7){$\mu$};
\node at (0.2, 1.4, -0.7){$\nu$};
\node[anchor=south] at (1, 1.8, 0){$Y$};
\node[anchor=south] at (0, 2.1, 0.5){$\Sigma_1$};
\node[anchor=south] at (0, 2.1, -1.5){$\Sigma_2$};
\end{scope}
\end{tikzpicture}
}}
\\
\;\;=\;\;
\sum_{\mu,\nu \in \{\pm\}}
(-A^2)^{\frac{\mu+\nu}{2}}
\qty(
\vcenter{\hbox{
\tdplotsetmaincoords{20}{40}
\begin{tikzpicture}[tdplot_main_coords, rotate=40]
\begin{scope}[scale = 1.0, tdplot_main_coords]
\fill[black] (0, {1.2+0.085}, {-0.8+0.085}) .. controls (0.05, {1.2+0.085}, 0) and (0.05, {1.2+0.085}, 0) .. (0, {1.2+0.085}, {0.8-0.085}) -- (0, {1.2-0.085}, {0.8+0.085}) .. controls (0.05, {1.2-0.085}, 0) and (0.05, {1.2-0.085}, 0) .. (0, {1.2-0.085}, {-0.8-0.085}) -- cycle;
\fill[color=lightgray, opacity=0.6] (0, 0, 0) -- (0, 0, 2) -- (0, 2, 2) -- (0, 2, 0) -- cycle;
\fill[color=lightgray, opacity=0.6] (0, 0, -2) -- (0, 0, 0) -- (0, 2, 0) -- (0, 2, -2) -- cycle;
\draw[ultra thick, orange] (0, 0, 2) -- (0, 2, 0) -- (0, 0, -2);
\draw[thick, orange, ->] (0, 0, 2) -- (0, 1, 1);
\draw[thick, orange, ->] (0, 0, -2)-- (0, 1, -1);
\draw[dotted] (0, 0, 2) -- (0, 2, 2) -- (0, 2, -2) -- (0, 0, -2) -- cycle;
\draw[very thick, darkgray] (0, 0, 0) -- (0, 2, 0);
\filldraw[orange] (0, 2, 0) circle (0.07);
\node at (-0.37, 1.3, 0.7){$\mu$};
\node at (0.2, 1.1, -0.7){$\nu$};
\end{scope}
\end{tikzpicture}
}}
)
\cdot
\vcenter{\hbox{
\tdplotsetmaincoords{20}{40}
\begin{tikzpicture}[tdplot_main_coords, rotate=40]
\begin{scope}[scale = 1.0, tdplot_main_coords]
\draw[dotted] (1.5, 0, 2) -- (0, 0, 2);
\fill[black] (1.5, {0.8-0.085}, 0.6) .. controls (1, {0.8-0.085}, 0.6) and (1, {0.8-0.085}, 0.6) .. (0, {0.8-0.085}, {1.2+0.085}) -- (0, {0.8+0.085}, {1.2-0.085}) .. controls (1, {0.8+0.085}, 0.6) and (1, {0.8+0.085}, 0.6) .. (1.5, {0.8+0.085}, 0.6) -- cycle;
\fill[black] (1.5, {0.8-0.085}, -0.6) .. controls (1, {0.8-0.085}, -0.6) and (1, {0.8-0.085}, -0.6) .. (0, {0.8-0.085}, {-1.2-0.085}) -- (0, {0.8+0.085}, {-1.2+0.085}) .. controls (1, {0.8+0.085}, -0.6) and (1, {0.8+0.085}, -0.6) .. (1.5, {0.8+0.085}, -0.6) -- cycle;
\fill[color=lightgray, opacity=0.6] (0, 0, 0) -- (0, 0, 2) -- (0, 2, 2) -- (0, 2, 0) -- cycle;
\fill[color=lightgray, opacity=0.6] (0, 0, -2) -- (0, 0, 0) -- (0, 2, 0) -- (0, 2, -2) -- cycle;
\draw[ultra thick, orange] (0, 0, 2) -- (0, 2, 0) -- (0, 0, -2);
\draw[thick, orange, ->] (0, 0, 2) -- (0, 1, 1);
\draw[thick, orange, ->] (0, 0, -2)-- (0, 1, -1);
\draw[dotted] (1.5, 0, 2) -- (1.5, 2, 2) -- (1.5, 2, -2) -- (1.5, 0, -2) -- cycle;
\draw[dotted] (0, 0, 2) -- (0, 2, 2) -- (0, 2, -2) -- (0, 0, -2) -- cycle;
\draw[very thick, darkgray] (0, 0, 0) -- (0, 2, 0);
\draw[dotted] (1.5, 2, 2) -- (0, 2, 2);
\draw[dotted] (1.5, 2, -2) -- (0, 2, -2);
\draw[dotted] (1.5, 0, -2) -- (0, 0, -2);
\filldraw[orange] (0, 2, 0) circle (0.07);
\node at (-0.35, 0.8, 1.2){$-\mu$};
\node at (-0.35, 0.8, -1.2){$-\nu$};
\node[anchor=south] at (1, 1.8, 0){$Y$};
\node[anchor=south] at (0, 2.1, 0.5){$\Sigma_1$};
\node[anchor=south] at (0, 2.1, -1.5){$\Sigma_2$};
\end{scope}
\end{tikzpicture}
}}
\end{gather*}
Comparing the two images, it follows immediately from relation \ref{item:PS+} that they represent the same element in the quotient $\Sk(Y', \Gamma')/M$. 
\end{proof}

\begin{rmk}\label{rmk:NewGradingAfterSplitting}
Here we make a quick remark on grading. 
Recall from Remark \ref{rmk:SkeinModuleIsGraded} that $\Sk(Y,\Gamma)$ is $\mathbb{Z}^{E(\Gamma)}$-graded. 
It is clear that the reduced tensor product $\Sk(Y_1, \Gamma_1) \;\overline{\otimes}\; \Sk(Y_2, \Gamma_2)$ is also $\mathbb{Z}^{E(\Gamma)}$-graded and that the splitting map $\sigma : \Sk(Y, \Gamma) \rightarrow \Sk(Y_1, \Gamma_1) \;\overline{\otimes}\; \Sk(Y_2, \Gamma_2)$ respects this grading. 

What is more, $\Sk(Y_1, \Gamma_1) \;\overline{\otimes}\; \Sk(Y_2, \Gamma_2)$ carries an extra $\mathbb{Z}^{E(\Gamma_{\Sigma})}$-grading, where $\Gamma_{\Sigma}$ denotes the marking on the cutting surface $\Sigma$.
To see this, first note that $\Sk(Y_1 \sqcup Y_2, \Gamma_1 \sqcup \Gamma_2) \cong \Sk(Y_1, \Gamma_1) \otimes \Sk(Y_2, \Gamma_2)$ carries a natural $\mathbb{Z}^{E(\Gamma_1 \cap \Sigma)} \times \mathbb{Z}^{E(\Gamma_2 \cap \Sigma)}$-grading, which induces a $\mathbb{Z}^{E(\Gamma_{\Sigma})}$-grading;
for each stated tangle in $(Y_1 \sqcup Y_2, \Gamma_1 \sqcup \Gamma_2)$, 
its grading with respect to an edge $e\in E(\Gamma_{\Sigma})$ is its grading with respect to a copy of $e$ in $\Gamma_1 \cap \Sigma$ minus its grading with respect to a copy of $e$ in $\Gamma_2 \cap \Sigma$. 
The relations \ref{item:S+} and \ref{item:S-} are homogeneous with respect to this $\mathbb{Z}^{E(\Gamma_{\Sigma})}$-grading, so this grading descends to the reduced tensor product. 
From the definition of the splitting map $\sigma$, it is clear that the image of $\sigma$ is contained in the $0^{E(\Gamma_{\Sigma})}$-graded piece with respect to this $\mathbb{Z}^{E(\Gamma_{\Sigma})}$-grading. 
\end{rmk}

While in Theorem \ref{thm:SplittingMap} we stated the splitting map just as an $R$-module homomorphism
\[
\sigma : \Sk(Y, \Gamma) \rightarrow \Sk(Y_1, \Gamma_1) \;\overline{\otimes}\; \Sk(Y_2, \Gamma_2),
\]
it is easy to see that it preserves the bimodule structure as well. 
Recall that $\Sk(Y, \Gamma)$ is a $\otimes_{v \in V(\Gamma)^{+}}\SkAlg(D_{\deg v})$-$\otimes_{w \in V(\Gamma)^{-}}\SkAlg(D_{\deg w})$-bimodule. 
To explain why $\sigma$ is a bimodule homomorphism, we need to first discuss the corresponding bimodule structure on the reduced tensor product. 

\begin{lem}\label{lem:BimodStructureOnRedTensProd}
There is a natural
$\otimes_{v \in V(\Gamma)^{+}}\SkAlg(D_{\deg v})$-$\otimes_{w \in V(\Gamma)^{-}}\SkAlg(D_{\deg w})$-bimodule structure on $\Sk(Y_1, \Gamma_1) \;\overline{\otimes}\; \Sk(Y_2, \Gamma_2)$. 
\end{lem}
\begin{proof}
There are two cases, depending on whether each vertex $v\in \Gamma$ is in $\partial \Sigma$ or not. 
\begin{enumerate}[label= (BM\arabic*)]
\item For each vertex $v\in \Gamma$ that is not in $\partial \Sigma$, 
it is obvious that the reduced tensor product $\Sk(Y_1, \Gamma_1) \;\overline{\otimes}\; \Sk(Y_2, \Gamma_2)$
has a $\SkAlg(D_{\deg v})$-module structure
induced from that of $\Sk(Y_1, \Gamma_1)$ (resp.\ $\Sk(Y_2, \Gamma_2)$) if $v\in Y_1$ (resp.\ if $v\in Y_2$). 
\item\label{item:vInPartialSigma} Let $v$ be a vertex in $\Gamma \cap \partial \Sigma$. 
Then we define the action of $\alpha \in \SkAlg(D_{\deg v})$ on $x\in \Sk(Y_1, \Gamma_1) \;\overline{\otimes}\; \Sk(Y_2, \Gamma_2)$ to be
\[
\sigma(\alpha)\cdot x 
\quad
\qty(\text{resp. }x \cdot \sigma(\alpha) 
\,),
\]
if $v$ is a sink (resp. source). 
\end{enumerate}
Let us elaborate on what we mean by $\sigma(\alpha)$ in \ref{item:vInPartialSigma}. 
Let $n = n_1 + n_2$, $n_1$, $n_2$, and $m$ be the degree of $v$ in $\Gamma$, $\Gamma \cap Y_1$, $\Gamma \cap Y_2$, and in the marking $\Gamma_\Sigma$ of $\Sigma$, respectively. 
Consider a cone over $D_{n}$, with the usual boundary marking consisting of one interval on each triangular side face, connecting the cone point with the midpoint of the corresponding edge of the base $D_{n}$. 
Let's call this cone $CD_{n}$. 
For example, the cone $CD_{5}$ would look like the following figure: 
\[
\vcenter{\hbox{
\tdplotsetmaincoords{55}{80}
\begin{tikzpicture}[tdplot_main_coords]
\begin{scope}[scale = 0.8, tdplot_main_coords]
    \coordinate (o) at (0, 0, 0);
    \coordinate (a) at (3, 0, 0);
    \coordinate (b) at ({3*cos(72)}, {3*sin(72)}, 0);
    \coordinate (c) at ({3*cos(2*72)}, {3*sin(2*72)}, 0);
    \coordinate (d) at ({3*cos(3*72)}, {3*sin(3*72)}, 0);
    \coordinate (e) at ({3*cos(4*72)}, {3*sin(4*72)}, 0);
    \coordinate (up) at (0, 0, {3*sqrt(3)/2});
    \coordinate (ab) at ({3*(cos(72)+1)/2}, {3*sin(72)/2}, 0);
    \coordinate (bc) at ({3*(cos(72)+cos(2*72))/2}, {3*(sin(72)+sin(2*72))/2}, 0);
    \coordinate (cd) at ({3*(cos(2*72)+cos(3*72))/2}, {3*(sin(2*72)+sin(3*72))/2}, 0);
    \coordinate (de) at ({3*(cos(3*72)+cos(4*72))/2}, {3*(sin(3*72)+sin(4*72))/2}, 0);
    \coordinate (ea) at ({3*(cos(4*72)+cos(5*72))/2}, {3*(sin(4*72)+sin(5*72))/2}, 0);
    \fill[lightgray, opacity=0.3] (up) -- (c) -- (d) -- cycle;
    \draw[very thick, dashed] (c) -- (d);
    \fill[lightgray, opacity=0.3] (up) -- (a) -- (b) -- cycle;
    \fill[lightgray, opacity=0.3] (up) -- (b) -- (c) -- cycle;
    \fill[lightgray, opacity=0.3] (up) -- (d) -- (e) -- cycle;
    \fill[lightgray, opacity=0.3] (up) -- (e) -- (a) -- cycle;
    \draw[orange, very thick] (up) -- (ab);
    \draw[orange, very thick] (up) -- (bc);
    \draw[orange, very thick, dashed] (up) -- (cd);
    \draw[orange, very thick] (up) -- (de);
    \draw[orange, very thick] (up) -- (ea);
    \draw[very thick] (a) -- (b);
    \draw[very thick] (b) -- (c);
    \draw[very thick] (d) -- (e);
    \draw[very thick] (e) -- (a);
    \draw[very thick] (up) -- (a);
    \draw[very thick] (up) -- (b);
    \draw[very thick] (up) -- (c);
    \draw[very thick] (up) -- (d);
    \draw[very thick] (up) -- (e);
    \filldraw (a) circle (0.05em);
    \filldraw (b) circle (0.05em);
    \filldraw (c) circle (0.05em);
    \filldraw (d) circle (0.05em);
    \filldraw (e) circle (0.05em);
    \filldraw (up) circle (0.05em);
\end{scope}
\end{tikzpicture}
}}
\;.
\]
Then, $\Sk(CD_n)$ is just $\SkAlg(D_n)$ as the left (resp.\ right) regular module over itself, if the cone point is a sink (resp.\ source). 

This cone can be thought of as a local model of $Y$ near $v$. 
The local model of the splitting map near $v$ is given by splitting $CD_{n_1+n_2}$ into $CD_{n_1+m}$ and $CD_{n_2+m}$. 
That is, we split the base $D_{n_1+n_2}$ into $D_{n_1+1}$ and $D_{n_2+1}$, add $m$ markings on the newly created edge, and then take the cone over it. 
The following figure illustrates this splitting ($n_1=3$, $n_2=2$, $m=2$ in this example). 
\[
\vcenter{\hbox{
\tdplotsetmaincoords{55}{80}
\begin{tikzpicture}[tdplot_main_coords]
\begin{scope}[scale = 0.8, tdplot_main_coords]
    \coordinate (o) at (0, 0, 0);
    \coordinate (a) at (3, 0, 0);
    \coordinate (b) at ({3*cos(72)}, {3*sin(72)}, 0);
    \coordinate (c) at ({3*cos(2*72)}, {3*sin(2*72)}, 0);
    \coordinate (d) at ({3*cos(3*72)}, {3*sin(3*72)}, 0);
    \coordinate (e) at ({3*cos(4*72)}, {3*sin(4*72)}, 0);
    \coordinate (up) at (0, 0, {3*sqrt(3)/2});
    \coordinate (ab) at ({3*(cos(72)+1)/2}, {3*sin(72)/2}, 0);
    \coordinate (bc) at ({3*(cos(72)+cos(2*72))/2}, {3*(sin(72)+sin(2*72))/2}, 0);
    \coordinate (cd) at ({3*(cos(2*72)+cos(3*72))/2}, {3*(sin(2*72)+sin(3*72))/2}, 0);
    \coordinate (de) at ({3*(cos(3*72)+cos(4*72))/2}, {3*(sin(3*72)+sin(4*72))/2}, 0);
    \coordinate (ea) at ({3*(cos(4*72)+cos(5*72))/2}, {3*(sin(4*72)+sin(5*72))/2}, 0);
    \fill[lightgray, opacity=0.3] (up) -- (c) -- (d) -- cycle;
    \draw[very thick, dashed] (c) -- (d);
    \draw[orange, very thick, dashed] (up) -- (cd);
    \fill[lightgray, opacity=0.3] (up) -- (a) -- (b) -- cycle;
    \fill[lightgray, opacity=0.3] (up) -- (b) -- (c) -- cycle;
    \fill[lightgray, opacity=0.3] (up) -- (d) -- (e) -- cycle;
    \fill[lightgray, opacity=0.3] (up) -- (e) -- (a) -- cycle;
    \draw[orange, very thick] (up) -- (ab);
    \draw[orange, very thick] (up) -- (bc);
    \draw[orange, very thick] (up) -- (de);
    \draw[orange, very thick] (up) -- (ea);
    \draw[very thick] (a) -- (b);
    \draw[very thick] (b) -- (c);
    \draw[very thick] (d) -- (e);
    \draw[very thick] (e) -- (a);
    \draw[very thick] (up) -- (a);
    \draw[very thick] (up) -- (b);
    \draw[very thick] (up) -- (c);
    \draw[very thick] (up) -- (d);
    \draw[very thick] (up) -- (e);
    \filldraw (a) circle (0.05em);
    \filldraw (b) circle (0.05em);
    \filldraw (c) circle (0.05em);
    \filldraw (d) circle (0.05em);
    \filldraw (e) circle (0.05em);
    \filldraw (up) circle (0.05em);
\end{scope}
\end{tikzpicture}
}}
\;\rightarrow\;
\vcenter{\hbox{
\tdplotsetmaincoords{55}{80}
\begin{tikzpicture}[tdplot_main_coords]
\begin{scope}[scale = 0.8, tdplot_main_coords]
    \coordinate (o) at (0, 0, 0);
    \coordinate (a) at (3, 0, 0);
    \coordinate (b) at ({3*cos(72)}, {3*sin(72)}, 0);
    \coordinate (c) at ({3*cos(2*72)}, {3*sin(2*72)}, 0);
    \coordinate (d) at ({3*cos(3*72)}, {3*sin(3*72)}, 0);
    \coordinate (e) at ({3*cos(4*72)}, {3*sin(4*72)}, 0);
    \coordinate (up) at (0, 0, {3*sqrt(3)/2});
    \coordinate (ab) at ({3*(cos(72)+1)/2}, {3*sin(72)/2}, 0);
    \coordinate (bc) at ({3*(cos(72)+cos(2*72))/2}, {3*(sin(72)+sin(2*72))/2}, 0);
    \coordinate (cd) at ({3*(cos(2*72)+cos(3*72))/2}, {3*(sin(2*72)+sin(3*72))/2}, 0);
    \coordinate (de) at ({3*(cos(3*72)+cos(4*72))/2}, {3*(sin(3*72)+sin(4*72))/2}, 0);
    \coordinate (ea) at ({3*(cos(4*72)+cos(5*72))/2}, {3*(sin(4*72)+sin(5*72))/2}, 0);
    \coordinate (aac) at ({2+cos(2*72)}, {sin(2*72)}, 0);
    \coordinate (acc) at ({1+2*cos(2*72)}, {2*sin(2*72)}, 0);
    \fill[lightgray, opacity=0.3] (up) -- (c) -- (d) -- cycle;
    \draw[very thick, dashed] (c) -- (d);
    \draw[orange, very thick, dashed] (up) -- (cd);
    \fill[lightgray, opacity=0.3] (up) -- (a) -- (c) -- cycle;
    \fill[lightgray, opacity=0.3] (up) -- (d) -- (e) -- cycle;
    \fill[lightgray, opacity=0.3] (up) -- (e) -- (a) -- cycle;
    \draw[orange, very thick] (up) -- (aac);
    \draw[orange, very thick] (up) -- (acc);
    \draw[orange, very thick] (up) -- (de);
    \draw[orange, very thick] (up) -- (ea);
    \draw[very thick] (a) -- (c);
    \draw[very thick] (d) -- (e);
    \draw[very thick] (e) -- (a);
    \draw[very thick] (up) -- (a);
    \draw[very thick] (up) -- (c);
    \draw[very thick] (up) -- (d);
    \draw[very thick] (up) -- (e);
    \filldraw (a) circle (0.05em);
    \filldraw (c) circle (0.05em);
    \filldraw (d) circle (0.05em);
    \filldraw (e) circle (0.05em);
    \filldraw (up) circle (0.05em);
\end{scope}
\end{tikzpicture}
}}
\vcenter{\hbox{
\tdplotsetmaincoords{55}{80}
\begin{tikzpicture}[tdplot_main_coords]
\begin{scope}[scale = 0.8, tdplot_main_coords]
    \coordinate (o) at (0, 0, 0);
    \coordinate (a) at (3, 0, 0);
    \coordinate (b) at ({3*cos(72)}, {3*sin(72)}, 0);
    \coordinate (c) at ({3*cos(2*72)}, {3*sin(2*72)}, 0);
    \coordinate (d) at ({3*cos(3*72)}, {3*sin(3*72)}, 0);
    \coordinate (e) at ({3*cos(4*72)}, {3*sin(4*72)}, 0);
    \coordinate (up) at (0, 0, {3*sqrt(3)/2});
    \coordinate (ab) at ({3*(cos(72)+1)/2}, {3*sin(72)/2}, 0);
    \coordinate (bc) at ({3*(cos(72)+cos(2*72))/2}, {3*(sin(72)+sin(2*72))/2}, 0);
    \coordinate (cd) at ({3*(cos(2*72)+cos(3*72))/2}, {3*(sin(2*72)+sin(3*72))/2}, 0);
    \coordinate (de) at ({3*(cos(3*72)+cos(4*72))/2}, {3*(sin(3*72)+sin(4*72))/2}, 0);
    \coordinate (ea) at ({3*(cos(4*72)+cos(5*72))/2}, {3*(sin(4*72)+sin(5*72))/2}, 0);
    \coordinate (aac) at ({2+cos(2*72)}, {sin(2*72)}, 0);
    \coordinate (acc) at ({1+2*cos(2*72)}, {2*sin(2*72)}, 0);
    \fill[lightgray, opacity=0.3] (up) -- (a) -- (c) -- cycle;
    \draw[orange, very thick, dashed] (up) -- (aac);
    \draw[orange, very thick, dashed] (up) -- (acc);
    \draw[very thick, dashed] (c) -- (a);
    \fill[lightgray, opacity=0.3] (up) -- (a) -- (b) -- cycle;
    \fill[lightgray, opacity=0.3] (up) -- (b) -- (c) -- cycle;
    \draw[orange, very thick] (up) -- (ab);
    \draw[orange, very thick] (up) -- (bc);
    \draw[very thick] (a) -- (b);
    \draw[very thick] (b) -- (c);
    \draw[very thick] (up) -- (a);
    \draw[very thick] (up) -- (b);
    \draw[very thick] (up) -- (c);
    \filldraw (a) circle (0.05em);
    \filldraw (b) circle (0.05em);
    \filldraw (c) circle (0.05em);
    \filldraw (up) circle (0.05em);
\end{scope}
\end{tikzpicture}
}}
\]
By Theorem \ref{thm:SplittingMap}, we have the associated splitting map
\[
\sigma : \Sk(CD_{n_1+n_2}) \rightarrow \Sk(CD_{n_1+m})\;\overline{\otimes}\;\Sk(CD_{n_2+m}).
\]
Using the isomorphism $\Sk(CD_n) \cong \SkAlg(D_n)$ (as regular modules over $\SkAlg(D_n)$), 
we can also write this splitting map as
\begin{equation}\label{eq:LocalSplittingMap}
\sigma : \SkAlg(D_{n_1+n_2}) \rightarrow \SkAlg(D_{n_1+m}) \;\overline{\otimes}\; \SkAlg(D_{n_2+m}),
\end{equation}
which is a priori just an $R$-module map. 

We claim that the image $\sigma(\SkAlg(D_{n_1+n_2}))$ of this map is an algebra and that $\sigma$ is an algebra homomorphism onto its image. 
Let $I$ be the right (resp.\ left) ideal of $\SkAlg(D_{n_1+m}) \;\otimes\; \SkAlg(D_{n_2+m})$ which we quotient out by to get the reduced tensor product when $v$ is a sink (resp.\ source); 
that is, $I$ is the right (resp.\ left) ideal generated by the relations \ref{item:S+} (resp.\ \ref{item:S-}) near $v$, in the case $v$ is a sink (resp.\ source).
Then, elements of $\SkAlg(D_{n_1+m}) \;\overline{\otimes}\; \SkAlg(D_{n_2+m})$ are cosets of the form $a+I$, where $a\in \SkAlg(D_{n_1+m}) \otimes \SkAlg(D_{n_2+m})$. 
To show that $\sigma(\SkAlg(D_{n_1+n_2}))$ has a natural algebra structure, we need to show that
$(a + I)(b + I) = ab + I$,
for any $a + I, b + I$ in $\sigma(\SkAlg(D_{n_1+n_2}))$,
or equivalently, $aI \subset I$ (resp.\ $Ib \subset I$) in the case $v$ is a sink (resp.\ source). 

Let's assume that $v$ is a sink; the proof for the case when $v$ is a source is analogous. 
Let's show that $aI \subset I$. 
The following lemma, which easily follows from the stated skein relations, will be useful: 
\begin{lem}\label{lem:StatedSkeinRelLemma}
We have the following skein relations:\footnote{All the diagrams drawn here are seen from \emph{outside} of the 3-manifold.} 
\begin{enumerate}
\item
\[
\vcenter{\hbox{

\;.
\end{align*}
\end{enumerate}
\end{lem}

Recall that the generators of the right ideal $I$ are of the form
\begin{gather*}
(-A^2)^{\frac{\mu + \nu}{4}}\;
\vcenter{\hbox{
\tdplotsetmaincoords{20}{40}
%
}}
\;,
\end{gather*}
where we have drawn only two of the markings that are next to each other on the surface we are cutting along. 
Let's call this element $\alpha_{\mu,\nu}$ for now, for simplicity of notation. 
An important property of this element is the following identity, which follows from Lemma \ref{lem:StatedSkeinRelLemma}:
\begin{gather*}
\qty(
\sum_{\epsilon \in \{\pm\}}
\vcenter{\hbox{
\tdplotsetmaincoords{20}{40}
%
}}
)
\;.
\end{gather*}
Likewise, we also have
\begin{gather*}
\qty(
\sum_{\epsilon \in \{\pm\}}
\vcenter{\hbox{
\tdplotsetmaincoords{20}{40}
%
}}
)
\;.
\end{gather*}
It follows that, for any $a\in \SkAlg(D_{n_1+m}) \otimes \SkAlg(D_{n_2+m})$ in the image of the splitting map $\sigma$, 
we can repeatedly use the above two identities to reexpress $a\cdot \alpha_{\mu,\nu}$ as
\[
a\cdot 
\alpha_{\mu,\nu} = 
\sum_{\epsilon_1,\epsilon_2 \in \{\pm\}}
\alpha_{\epsilon_1,\epsilon_2}\cdot
b_{\epsilon_1,\epsilon_2}
\]
for some $b_{\epsilon_1,\epsilon_2} \in \SkAlg(D_{n_1+m}) \otimes \SkAlg(D_{n_2+m})$. 
It follows that $aI \subset I$ for any such $a$. 
It is immediate that the splitting map \eqref{eq:LocalSplittingMap}, when the codomain is restricted to the image, is an algebra homomorphism. 

Finally, there is a natural action of $\SkAlg(D_{n_1+m}) \;\overline{\otimes}\; \SkAlg(D_{n_2+m})$ on $\Sk(Y_1, \Gamma_1) \;\overline{\otimes}\; \Sk(Y_2, \Gamma_2)$, 
and it is this action we were using in \ref{item:vInPartialSigma}. 
\end{proof}

\begin{rmk}
In the special case $m=1$, there are no relations imposed on the reduced tensor product, so 
\[
\SkAlg(D_{n_1+1}) \;\overline{\otimes}\; \SkAlg(D_{n_2+1}) = \SkAlg(D_{n_1+1}) \otimes \SkAlg(D_{n_2+1}),
\]
and the splitting map \eqref{eq:LocalSplittingMap} is just the usual 2d splitting map \cite{CL}. 
\end{rmk}

\begin{thm}\label{thm:SplittingMapIsABimodHom}
The splitting map $\sigma$ in Theorem \ref{thm:SplittingMap} is not just an $R$-module homomorphism, but a $\otimes_{v \in V(\Gamma)^{+}}\SkAlg(D_{\deg v})$-$\otimes_{w \in V(\Gamma)^{-}}\SkAlg(D_{\deg w})$-bimodule homomorphism. 
\end{thm}
\begin{proof}
This is immediate from the way we defined the bimodule structure on the reduced tensor product in Lemma \ref{lem:BimodStructureOnRedTensProd}. 
\end{proof}

By iteratively cutting a 3-manifold into pieces along combinatorially foliated surfaces and using Theorems \ref{thm:SplittingMap} and \ref{thm:SplittingMapIsABimodHom}, we get the following corollaries: 

\begin{cor}\label{cor:SplittingIntoTetrahedra} 
Let $Y = \cup_{T \in \mathcal{T}} T$ be an ideal triangulation of a 3-manifold $Y$. 
Then, there is a well-defined splitting map
\begin{align*}
\Sk(Y) &\rightarrow \overline{\bigotimes}_{T \in \mathcal{T}}\Sk(T)\\
[L] &\mapsto \qty[\sum_{\vec{\epsilon} \;\in\; \{\substack{\mathrm{compatible}\\ \mathrm{states}}\}}\otimes_{T \in \mathcal{T}} [L_{T}^{\vec{\epsilon}}]],
\end{align*}
where $\overline{\bigotimes}_{T \in \mathcal{T}}\Sk(T)$ denotes the quotient of the usual tensor product $\bigotimes_{T \in \mathcal{T}}\Sk(T)$ (as $R$-modules) by the following relations:
\begin{enumerate}
    \item\label{item:FaceRelation} For each internal face $f \in \mathcal{F}$ 
    \[
    \vcenter{\hbox{
    \includestandalone[scale=0.7]{figures/cutting_around_a_faceLHS}
    }}
    \;,
    \]
    we have the following relations among left actions of $\bigotimes_{\substack{T\in \mathcal{T} \\ f\in \mathbf{f}(T)}}\SkAlg(D_3)$ on $\bigotimes_{T \in \mathcal{T}}\Sk(T)$:
    \[
    (-A^2)^{\frac{\mu+\nu}{4}}\;
    \vcenter{\hbox{
    \includestandalone[scale=0.7]{figures/face_splitting_relationLHS}
    }}
    \;\;=\;\;
    (-A^2)^{-\frac{\mu+\nu}{4}}\;
    \vcenter{\hbox{
    \includestandalone[scale=0.7]{figures/face_splitting_relationRHS}
    }}
    \;.
    \]
    \item\label{item:EdgeRelation} For each internal edge $e \in \mathcal{E}$ 
    \[
    \vcenter{\hbox{
    \includestandalone[scale=0.8]{figures/cutting_around_an_edgeLHS}
    }}
    \;,
    \]
    we have the following relations among right actions of $\bigotimes_{\substack{T\in \mathcal{T} \\ e\in \mathbf{e}(T)}}\SkAlg(D_2)$ on $\bigotimes_{T \in \mathcal{T}}\Sk(T)$: 
    \begin{gather*}
    (-A^2)^{-\frac{\mu + \nu}{4}}
    \vcenter{\hbox{
    \includestandalone[scale=0.8]{figures/edge_splitting_relationLHS}
    }}
    \;\; = \;\;
    (-A^2)^{\frac{\mu + \nu}{4}}
    \sum_{\epsilon_1, \cdots, \epsilon_{m-2} \in \{\pm\}}
    \vcenter{\hbox{
    \includestandalone[scale=0.8]{figures/edge_splitting_relationRHS}
    }}
    \;, 
    \end{gather*}
    where $m$ is the (local) number of tetrahedra around the edge $e$. 
\end{enumerate}
\end{cor}

\begin{rmk}\label{rmk:SplittingRelationSymmetricalForm}
The above relations around each internal edge are equivalent to the following relation that is written in a more symmetric form: 
\begin{gather*}
    \sum_{\epsilon_1, \cdots, \epsilon_{m-1} \in \{\pm\}}
    \vcenter{\hbox{
    \begin{tikzpicture}[scale=0.8]
    \newcommand*{\defcoords}{
    \coordinate (o) at (0, 0);
    \coordinate (a) at ({sqrt(3)}, 1);
    \coordinate (oa) at ({sqrt(3)/3}, 1/3);
    \coordinate (ab) at ({2*sqrt(3)/3}, 4/3);
    \coordinate (ba) at ({sqrt(3)/3}, 5/3);
    \coordinate (b) at (0, 2);
    \coordinate (ob) at (0, 2/3);
    \coordinate (bc) at (-{sqrt(3)/3}, 5/3);
    \coordinate (cb) at ({-2*sqrt(3)/3}, 4/3);
    \coordinate (c) at (-{sqrt(3)}, 1);
    \coordinate (oc) at ({-sqrt(3)/3}, 1/3);
    \coordinate (cd) at (-{sqrt(3)}, 1/3);
    \coordinate (dc) at (-{sqrt(3)}, -1/3);
    \coordinate (d) at (-{sqrt(3)}, -1);
    \coordinate (od) at (-{sqrt(3)/3}, -1/3);
    \coordinate (de) at ({-2*sqrt(3)/3}, -4/3);
    \coordinate (ed) at (-{sqrt(3)/3}, -5/3);
    \coordinate (e) at (0, -2);
    \coordinate (oe) at (0, -2/3);
    \coordinate (ef) at ({sqrt(3)/3}, -5/3);
    \coordinate (fe) at ({2*sqrt(3)/3}, -4/3);
    \coordinate (f) at ({sqrt(3)}, -1);
    \coordinate (of) at ({sqrt(3)/3}, -1/3);
    \coordinate (fa) at ({sqrt(3)}, -1/3);
    \coordinate (af) at ({sqrt(3)}, 1/3);
    }
    \begin{scope}[scale=1.0]
    \begin{scope}[shift={({1/2*1/2}, {1/2*sqrt(3)/2})}, very thick, decoration={
        markings, 
        mark=at position 0.5 with {\arrow{>}}}
        ] 
        \defcoords
        \draw[line width=5] (oa) arc (30:90: 2/3);
        \draw[postaction={decorate}, orange] (o) -- (a);
        \draw[postaction={decorate}, orange] (o) -- (b);
        \filldraw[orange, thick] (o) circle (0.01);
        \node[anchor = south] at (ob) {$\epsilon_1$};
        \node[anchor = south west] at (oa) {$\epsilon_2$};
    \end{scope}
    \begin{scope}[shift={({-1/2*1/2}, {1/2*sqrt(3)/2})}, very thick, decoration={
        markings,
        mark=at position 0.5 with {\arrow{>}}}
        ] 
        \defcoords
        \draw[line width=5] (ob) arc (90:150: 2/3);
        \draw[postaction={decorate}, orange] (o) -- (b);
        \draw[postaction={decorate}, orange] (o) -- (c);
        \filldraw[orange, thick] (o) circle (0.01);
        \node[anchor = south east] at (oc) {$\mu$};
        \node[anchor = south] at (ob) {$\epsilon_1$};
    \end{scope}
    \begin{scope}[shift={({-1/2}, 0)}, very thick, decoration={
        markings,
        mark=at position 0.5 with {\arrow{>}}}
        ] 
        \defcoords
        \draw[line width=5] (oc) arc (150:210: 2/3);
        \draw[postaction={decorate}, orange] (o) -- (c);
        \draw[postaction={decorate}, orange] (o) -- (d);
        \filldraw[orange, thick] (o) circle (0.01);
        \node[anchor = south east] at (oc) {$\nu$};
        \node[anchor = north east] at (od) {$\epsilon_{m-1}$};
    \end{scope}
    \begin{scope}[shift={({-1/2*1/2}, {-1/2*sqrt(3)/2})}, very thick, decoration={
        markings,
        mark=at position 0.5 with {\arrow{>}}}
        ] 
        \defcoords
        \draw[line width=5] (od) arc (210:270: 2/3);
        \draw[postaction={decorate}, orange] (o) -- (d);
        \draw[postaction={decorate}, orange] (o) -- (e);
        \filldraw[orange, thick] (o) circle (0.01);
        \node[anchor = north east] at (od) {$\epsilon_{m-1}$};
    \end{scope}
    \begin{scope}[shift={({1/2*1/2}, {-1/2*sqrt(3)/2})}, very thick, decoration={
        markings,
        mark=at position 0.5 with {\arrow{>}}}
        ] 
        \defcoords
        \draw[line width=5] (oe) arc (270:330: 2/3);
        \draw[postaction={decorate}, orange] (o) -- (e);
        \draw[postaction={decorate}, orange] (o) -- (f);
        \filldraw[orange, thick] (o) circle (0.01);
        \node[anchor = north] at (0, -1) {$\cdots$};
    \end{scope}
    \begin{scope}[shift={({1/2}, 0)}, very thick, decoration={
        markings,
        mark=at position 0.5 with {\arrow{>}}}
        ] 
        \defcoords
        \draw[line width=5] (of) arc (-30:30: 2/3);
        \draw[postaction={decorate}, orange] (o) -- (f);
        \draw[postaction={decorate}, orange] (o) -- (a);
        \filldraw[orange, thick] (o) circle (0.01);
        \node[anchor = south west] at (oa) {$\epsilon_2$};
        \node[anchor = west] at (1, 0.1) {$\vdots$};
    \end{scope}
    \end{scope}
    \end{tikzpicture}
    }}
    \;\;=\;\;
    \delta_{\mu,\nu}\,(-A^2)^{\mu}
    \;.
\end{gather*}
\end{rmk}

\begin{proof}[Proof of Corollary \ref{cor:SplittingIntoTetrahedra}]
The boundary of each ideal tetrahedron $T \in \mathcal{T}$ admits the following canonical combinatorial foliation:
\[
\vcenter{\hbox{
\includestandalone[scale=1.0]{figures/tetrahedron_combinatorial_foliation}
}}
\;.
\]
That is, each face $f\in \mathbf{f}(T)$ is divided into three elementary quadrilaterals, 
and we orient the associated boundary marking $\Gamma$ so that 
the vertex of $\Gamma$ at the center of each face $f\in \mathbf{f}(T)$ is a sink, and
the vertex of $\Gamma$ in the middle of each bare edge $e\in \mathbf{e}(T)$ is a source. 

We can split $Y$ into ideal tetrahedra by splitting one tetrahedron at a time: 
\[
Y = Y_0
\;\;\rightarrow\;\; 
T_1 \sqcup Y_1 
\;\;\rightarrow\;\; 
T_1 \sqcup T_2 \sqcup Y_2 
\;\;\rightarrow 
\cdots 
\rightarrow\;\; 
T_1 \sqcup T_2 \sqcup \cdots \sqcup T_n.
\]
Here, $\{T_1, \cdots, T_n\} = \mathcal{T}$ is the set of tetrahedra in the ideal triangulation of $Y$, and $Y_i$ denotes $T_{i+1} \cup \cdots \cup T_n$. 

Let's assume for simplicity that, around each edge $e\in \mathcal{E}$, each tetrahedron $T \in \mathcal{T}$ appears at most once.\footnote{It is not hard to show that this assumption is not necessary. 
We leave this as an exercise to the reader. } 
Every time we split a tetrahedron $T = T_{i+1}$ from $Y_{i}$, 
Theorems \ref{thm:SplittingMap} and \ref{thm:SplittingMapIsABimodHom} tell us exactly what relations we get: 
\begin{enumerate}
\item On each face of $T$ we are cutting along, there are three internal edges of the combinatorial foliation, and the relations we get from Theorem \ref{thm:SplittingMap} are the relations (\ref{item:FaceRelation}). 
\item \label{item:tetSplittingEdgeRelProof}For each bare edge $e$ of $T$, there are two possibilities:
    \begin{enumerate}
    \item If $e$ is an internal edge of $Y_{i}$, the two halves of $e$ are internal edges of the combinatorial foliation of $\partial T$, and from Theorem \ref{thm:SplittingMap}, they both give the following relation among right actions of $\SkAlg(D_2) \otimes \SkAlg(D_2)$ on $\Sk(T) \otimes \Sk(Y_{i+1})$: 
    \[
    (-A^2)^{-\frac{\mu + \nu}{4}}
    \vcenter{\hbox{
    \begin{tikzpicture}[scale=0.8]
    \newcommand*{\defcoords}{
    \coordinate (o) at (0, 0);
    \coordinate (a) at ({sqrt(3)}, 1);
    \coordinate (oa) at ({sqrt(3)/3}, 1/3);
    \coordinate (ab) at ({2*sqrt(3)/3}, 4/3);
    \coordinate (ba) at ({sqrt(3)/3}, 5/3);
    \coordinate (b) at (0, 2);
    \coordinate (ob) at (0, 2/3);
    \coordinate (bc) at (-{sqrt(3)/3}, 5/3);
    \coordinate (cb) at ({-2*sqrt(3)/3}, 4/3);
    \coordinate (c) at (-{sqrt(3)}, 1);
    \coordinate (oc) at ({-sqrt(3)/3}, 1/3);
    \coordinate (cd) at (-{sqrt(3)}, 1/3);
    \coordinate (dc) at (-{sqrt(3)}, -1/3);
    \coordinate (d) at (-{sqrt(3)}, -1);
    \coordinate (od) at (-{sqrt(3)/3}, -1/3);
    \coordinate (de) at ({-2*sqrt(3)/3}, -4/3);
    \coordinate (ed) at (-{sqrt(3)/3}, -5/3);
    \coordinate (e) at (0, -2);
    \coordinate (oe) at (0, -2/3);
    \coordinate (ef) at ({sqrt(3)/3}, -5/3);
    \coordinate (fe) at ({2*sqrt(3)/3}, -4/3);
    \coordinate (f) at ({sqrt(3)}, -1);
    \coordinate (of) at ({sqrt(3)/3}, -1/3);
    \coordinate (fa) at ({sqrt(3)}, -1/3);
    \coordinate (af) at ({sqrt(3)}, 1/3);
    }
    \begin{scope}[scale=1.0]
    \begin{scope}[shift={({-1/2}, 0)}, very thick, decoration={
        markings,
        mark=at position 0.5 with {\arrow{>}}}
        ] 
        \defcoords
        \draw[line width=5] (oc) arc (150:210: 2/3);
        \draw[postaction={decorate}, orange] (o) -- (c);
        \draw[postaction={decorate}, orange] (o) -- (d);
        \filldraw[orange, thick] (o) circle (0.01);
        \node[anchor = south east] at (oc) {$\mu$};
        \node[anchor = north east] at (od) {$\nu$};
    \end{scope}
    \begin{scope}[shift={(1/2, 0)}, very thick, decoration={
        markings,
        mark=at position 0.5 with {\arrow{>}}}
        ] 
        \defcoords
        \draw[postaction={decorate}, orange] (o) -- (c);
        \draw[postaction={decorate}, orange] (o) -- (d);
        \filldraw[orange, thick] (o) circle (0.01);
    \end{scope}
    \end{scope}
    \end{tikzpicture}
    }}
    \;\; = \;\;
    (-A^2)^{\frac{\mu + \nu}{4}}
    \vcenter{\hbox{
    \begin{tikzpicture}[scale=0.8]
    \newcommand*{\defcoords}{
    \coordinate (o) at (0, 0);
    \coordinate (a) at ({sqrt(3)}, 1);
    \coordinate (oa) at ({sqrt(3)/3}, 1/3);
    \coordinate (ab) at ({2*sqrt(3)/3}, 4/3);
    \coordinate (ba) at ({sqrt(3)/3}, 5/3);
    \coordinate (b) at (0, 2);
    \coordinate (ob) at (0, 2/3);
    \coordinate (bc) at (-{sqrt(3)/3}, 5/3);
    \coordinate (cb) at ({-2*sqrt(3)/3}, 4/3);
    \coordinate (c) at (-{sqrt(3)}, 1);
    \coordinate (oc) at ({-sqrt(3)/3}, 1/3);
    \coordinate (cd) at (-{sqrt(3)}, 1/3);
    \coordinate (dc) at (-{sqrt(3)}, -1/3);
    \coordinate (d) at (-{sqrt(3)}, -1);
    \coordinate (od) at (-{sqrt(3)/3}, -1/3);
    \coordinate (de) at ({-2*sqrt(3)/3}, -4/3);
    \coordinate (ed) at (-{sqrt(3)/3}, -5/3);
    \coordinate (e) at (0, -2);
    \coordinate (oe) at (0, -2/3);
    \coordinate (ef) at ({sqrt(3)/3}, -5/3);
    \coordinate (fe) at ({2*sqrt(3)/3}, -4/3);
    \coordinate (f) at ({sqrt(3)}, -1);
    \coordinate (of) at ({sqrt(3)/3}, -1/3);
    \coordinate (fa) at ({sqrt(3)}, -1/3);
    \coordinate (af) at ({sqrt(3)}, 1/3);
    }
    \begin{scope}[scale=1.0]
    \begin{scope}[shift={({-1/2}, 0)}, very thick, decoration={
        markings,
        mark=at position 0.5 with {\arrow{>}}}
        ] 
        \defcoords
        \draw[postaction={decorate}, orange] (o) -- (c);
        \draw[postaction={decorate}, orange] (o) -- (d);
        \filldraw[orange, thick] (o) circle (0.01);
    \end{scope}
    \begin{scope}[shift={(1/2, 0)}, very thick, decoration={
        markings,
        mark=at position 0.5 with {\arrow{>}}}
        ] 
        \defcoords
        \draw[line width=5] (oc) arc (150:-150: 2/3);
        \draw[postaction={decorate}, orange] (o) -- (c);
        \draw[postaction={decorate}, orange] (o) -- (d);
        \filldraw[orange, thick] (o) circle (0.01);
        \node[anchor = south east] at (oc) {$-\mu$};
        \node[anchor = north east] at (od) {$-\nu$};
    \end{scope}
    \end{scope}
    \end{tikzpicture}
    }},
    \]
    where the smaller angle represents $T$ and the bigger angle represents $Y_{i+1}$. 
    \item If $e$ is already on the boundary of $Y_{i}$, then there must be some $j < i$ such that $e$ is an internal edge of $Y_{j}$ but not of $Y_{j+1}$. 
    Theorem \ref{thm:SplittingMapIsABimodHom} identifies the right $\SkAlg(D_2)$-module structure on $\Sk(Y_i)$ associated to the midpoint of the edge $e$ with that on $\Sk(T) \overline{\otimes} \Sk(Y_{i+1})$. 
    After we have cut everything into tetrahedra, the right $\SkAlg(D_2)$-module structure on
    $\Sk(T_{j+1}) \overline{\otimes} \cdots \overline{\otimes} \Sk(T_{n})$ is induced from the following splitting homomorphism $\SkAlg(D_2) \rightarrow \SkAlg(D_2)^{\otimes (m-1)}$, where, $m$ is the number of tetrahedra (in $Y$) around the edge $e$: 
    \[
    \vcenter{\hbox{
    \begin{tikzpicture}[scale=0.8]
    \newcommand*{\defcoords}{
    \coordinate (o) at (0, 0);
    \coordinate (a) at ({sqrt(3)}, 1);
    \coordinate (oa) at ({sqrt(3)/3}, 1/3);
    \coordinate (ab) at ({2*sqrt(3)/3}, 4/3);
    \coordinate (ba) at ({sqrt(3)/3}, 5/3);
    \coordinate (b) at (0, 2);
    \coordinate (ob) at (0, 2/3);
    \coordinate (bc) at (-{sqrt(3)/3}, 5/3);
    \coordinate (cb) at ({-2*sqrt(3)/3}, 4/3);
    \coordinate (c) at (-{sqrt(3)}, 1);
    \coordinate (oc) at ({-sqrt(3)/3}, 1/3);
    \coordinate (cd) at (-{sqrt(3)}, 1/3);
    \coordinate (dc) at (-{sqrt(3)}, -1/3);
    \coordinate (d) at (-{sqrt(3)}, -1);
    \coordinate (od) at (-{sqrt(3)/3}, -1/3);
    \coordinate (de) at ({-2*sqrt(3)/3}, -4/3);
    \coordinate (ed) at (-{sqrt(3)/3}, -5/3);
    \coordinate (e) at (0, -2);
    \coordinate (oe) at (0, -2/3);
    \coordinate (ef) at ({sqrt(3)/3}, -5/3);
    \coordinate (fe) at ({2*sqrt(3)/3}, -4/3);
    \coordinate (f) at ({sqrt(3)}, -1);
    \coordinate (of) at ({sqrt(3)/3}, -1/3);
    \coordinate (fa) at ({sqrt(3)}, -1/3);
    \coordinate (af) at ({sqrt(3)}, 1/3);
    }
    \begin{scope}[scale=1.0]
    \begin{scope}[shift={(1/2, 0)}, very thick, decoration={
        markings,
        mark=at position 0.5 with {\arrow{>}}}
        ] 
        \defcoords
        \draw[line width=5] (oc) arc (150:-150: 2/3);
        \draw[postaction={decorate}, orange] (o) -- (c);
        \draw[postaction={decorate}, orange] (o) -- (d);
        \filldraw[orange, thick] (o) circle (0.01);
        \node[anchor = south east] at (oc) {$-\mu$};
        \node[anchor = north east] at (od) {$-\nu$};
    \end{scope}
    \end{scope}
    \end{tikzpicture}
    }}
    \;\;\mapsto\;\;
    \sum_{\epsilon_1, \cdots, \epsilon_m \in \{\pm\}}
    \vcenter{\hbox{
    \begin{tikzpicture}[scale=0.8]
    \newcommand*{\defcoords}{
    \coordinate (o) at (0, 0);
    \coordinate (a) at ({sqrt(3)}, 1);
    \coordinate (oa) at ({sqrt(3)/3}, 1/3);
    \coordinate (ab) at ({2*sqrt(3)/3}, 4/3);
    \coordinate (ba) at ({sqrt(3)/3}, 5/3);
    \coordinate (b) at (0, 2);
    \coordinate (ob) at (0, 2/3);
    \coordinate (bc) at (-{sqrt(3)/3}, 5/3);
    \coordinate (cb) at ({-2*sqrt(3)/3}, 4/3);
    \coordinate (c) at (-{sqrt(3)}, 1);
    \coordinate (oc) at ({-sqrt(3)/3}, 1/3);
    \coordinate (cd) at (-{sqrt(3)}, 1/3);
    \coordinate (dc) at (-{sqrt(3)}, -1/3);
    \coordinate (d) at (-{sqrt(3)}, -1);
    \coordinate (od) at (-{sqrt(3)/3}, -1/3);
    \coordinate (de) at ({-2*sqrt(3)/3}, -4/3);
    \coordinate (ed) at (-{sqrt(3)/3}, -5/3);
    \coordinate (e) at (0, -2);
    \coordinate (oe) at (0, -2/3);
    \coordinate (ef) at ({sqrt(3)/3}, -5/3);
    \coordinate (fe) at ({2*sqrt(3)/3}, -4/3);
    \coordinate (f) at ({sqrt(3)}, -1);
    \coordinate (of) at ({sqrt(3)/3}, -1/3);
    \coordinate (fa) at ({sqrt(3)}, -1/3);
    \coordinate (af) at ({sqrt(3)}, 1/3);
    }
    \begin{scope}[scale=1.0]
    \begin{scope}[shift={({1/2*1/2}, {1/2*sqrt(3)/2})}, very thick, decoration={
        markings, 
        mark=at position 0.5 with {\arrow{>}}}
        ] 
        \defcoords
        \draw[line width=5] (oa) arc (30:90: 2/3);
        \draw[postaction={decorate}, orange] (o) -- (a);
        \draw[postaction={decorate}, orange] (o) -- (b);
        \filldraw[orange, thick] (o) circle (0.01);
        \node[anchor = south] at (ob) {$\epsilon_1$};
        \node[anchor = south west] at (oa) {$\epsilon_2$};
    \end{scope}
    \begin{scope}[shift={({-1/2*1/2}, {1/2*sqrt(3)/2})}, very thick, decoration={
        markings,
        mark=at position 0.5 with {\arrow{>}}}
        ] 
        \defcoords
        \draw[line width=5] (ob) arc (90:150: 2/3);
        \draw[postaction={decorate}, orange] (o) -- (b);
        \draw[postaction={decorate}, orange] (o) -- (c);
        \filldraw[orange, thick] (o) circle (0.01);
        \node[anchor = south east] at (oc) {$-\mu$};
        \node[anchor = south] at (ob) {$\epsilon_1$};
    \end{scope}
    \begin{scope}[shift={({-1/2*1/2}, {-1/2*sqrt(3)/2})}, very thick, decoration={
        markings,
        mark=at position 0.5 with {\arrow{>}}}
        ] 
        \defcoords
        \draw[line width=5] (od) arc (210:270: 2/3);
        \draw[postaction={decorate}, orange] (o) -- (d);
        \draw[postaction={decorate}, orange] (o) -- (e);
        \filldraw[orange, thick] (o) circle (0.01);
        \node[anchor = north] at (oe) {$\epsilon_{m-2}$};
        \node[anchor = north east] at (od) {$-\nu$};
    \end{scope}
    \begin{scope}[shift={({1/2*1/2}, {-1/2*sqrt(3)/2})}, very thick, decoration={
        markings,
        mark=at position 0.5 with {\arrow{>}}}
        ] 
        \defcoords
        \draw[line width=5] (oe) arc (270:330: 2/3);
        \draw[postaction={decorate}, orange] (o) -- (e);
        \draw[postaction={decorate}, orange] (o) -- (f);
        \filldraw[orange, thick] (o) circle (0.01);
        \node[anchor = north west] at (of) {$\epsilon_{m-3}$};
        \node[anchor = north west] at (oe) {$\epsilon_{m-2}$};
    \end{scope}
    \begin{scope}[shift={({1/2}, 0)}, very thick, decoration={
        markings,
        mark=at position 0.5 with {\arrow{>}}}
        ] 
        \defcoords
        \draw[line width=5] (of) arc (-30:30: 2/3);
        \draw[postaction={decorate}, orange] (o) -- (f);
        \draw[postaction={decorate}, orange] (o) -- (a);
        \filldraw[orange, thick] (o) circle (0.01);
        \node[anchor = south west] at (oa) {$\epsilon_2$};
        \node[anchor = west] at (1, 0.1) {$\vdots$};
        \node[anchor = north west] at (of) {$\epsilon_{m-3}$};
    \end{scope}
    \end{scope}
    \end{tikzpicture}
    }}
    \;.
    \]
    \end{enumerate}
    Combining these two cases, we see that relations (\ref{item:EdgeRelation}) are the ones we get for each edge $e \in \mathcal{E}$. 
\end{enumerate}
\end{proof}

Recall from Definition \ref{defn:ConesAndSuspensions} that a face suspension is the union of two cones over a triangle $D_3$ glued along $D_3$, each with the usual boundary marking. 
A face suspension with its marking is shown in Figure \ref{fig:fs_marking}.  

\begin{figure}[H]
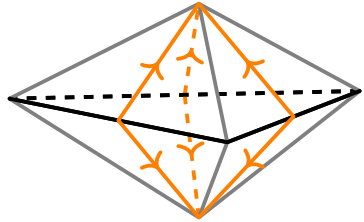

    \centering
    \includestandalone[scale=1.2]{figures/empty_fs}
    \caption{The boundary markings of a face suspension.}
    \label{fig:fs_marking}
\end{figure}

\begin{cor} \label{cor:fs_splitting_unreduced}
Let $Y = \cup_{f\in \mathcal{F}} Sf$ be a decomposition of an ideally triangulated  $3$-manifold $Y$ (without boundary except for cusps at infinity) into face suspensions. 
Then, there is a splitting map 
\begin{align*}
\Sk(Y) &\rightarrow \overline{\bigotimes}_{f\in \mathcal{F}}\Sk(Sf),\\
[L] &\mapsto \qty[\sum_{\vec{\epsilon} \;\in\; \{\substack{\mathrm{compatible}\\ \mathrm{states}}\}}\otimes_{f\in \mathcal{F}} [L_{f}^{\vec{\epsilon}}]],
\end{align*}
where $\overline{\bigotimes}_{f\in \mathcal{F}}\Sk(Sf)$ denotes the quotient of the tensor product $\bigotimes_{f\in \mathcal{F}}\Sk(Sf)$ (as $R$-modules) by the following relations:
\begin{enumerate}
    \item For each internal edge $e \in \mathcal{E}$, 
    we have the following relations among right actions of $\bigotimes_{\substack{f\in \mathcal{F} \\ e\in \mathbf{e}(f)}}\SkAlg(D_2)$ on $\bigotimes_{f\in \mathcal{F}}\Sk(Sf)$: 
    \begin{gather*}
    (-A^2)^{-\frac{\mu + \nu}{4}}
    \vcenter{\hbox{
    \includestandalone[scale=0.8]{figures/edge_splitting_relationLHS}
    }}
    \;\; = \;\;
    (-A^2)^{\frac{\mu + \nu}{4}}
    \sum_{\epsilon_1, \cdots, \epsilon_{m-2} \in \{\pm\}}
    \vcenter{\hbox{
    \includestandalone[scale=0.8]{figures/edge_splitting_relationRHS}
    }}
    \;.
    \end{gather*}
    \item For each vertex cone $Cv$ 
    \[
    \vcenter{\hbox{
    \includestandalone[scale=0.8]{figures/vertex_cone_rel_illustration}
    }}
    \;,
    \]
    we have the following relations among left actions of $\bigotimes_{\substack{T\in \mathcal{T} \\ f\in \mathbf{f}(T)}}\SkAlg(D_3)$ on $\bigotimes_{f\in \mathcal{F}} \Sk(Sf)$:
    \begin{gather*}
    (-A^2)^{\frac{\mu + \nu}{4}}
    \vcenter{\hbox{
    \includestandalone[scale=0.8]{figures/fs_splitting_relation1}
    }}
    \;\; = \;\;
    (-A^2)^{-\frac{\mu + \nu}{4}}
    \sum_{\epsilon \in \{\pm\}}
    \vcenter{\hbox{
    \includestandalone[scale=0.8]{figures/fs_splitting_relation2}
    }}
    \;.
    \end{gather*}
    In the above diagram, each sector represents one of the three face suspensions surrounding $Cv$, and the markings shown live on the edge cones adjacent to $Cv$.
\end{enumerate}
\end{cor}

\begin{proof}
    As in Corollary \ref{cor:SplittingIntoTetrahedra}, we iteratively split $Y$ into face suspensions by splitting one face suspension at a time:
    \[
    Y = Y_0
    \;\;\rightarrow\;\; 
    Sf_1 \sqcup Y_1 
    \;\;\rightarrow\;\; 
    Sf_1 \sqcup Sf_2 \sqcup Y_2 
    \;\;\rightarrow 
    \cdots 
    \rightarrow\;\; 
    Sf_1 \sqcup Sf_2 \sqcup \cdots \sqcup Sf_n.
    \]
    Every time we split a face suspension $Sf = Sf_{i+1}$ from $Y_i$, we obtain the following relations:
    \begin{enumerate}
        \item For each edge of $Sf$ that is also an edge of a tetrahedron, the argument is essentially the same as (\ref{item:tetSplittingEdgeRelProof}) in Corollary \ref{cor:SplittingIntoTetrahedra}, and we do not repeat it here. 
        \item For every vertex cone $Cv$ in $Sf$, there are three possibilities:
        \begin{enumerate}
            \item  \label{item:fs_splitting_case1} If $Cv$ is an internal edge of $Y_{i}$ and the barycenter adjacent to $Cv$ is an internal vertex of $Y_i$, we obtain the following relation among left actions of $\SkAlg(D_3) \otimes \SkAlg(D_3)$ on $\Sk(Sf) \otimes \Sk(Y_{i+1})$: 
            \begin{gather*}
                (-A^2)^{\frac{\mu + \nu}{4}}
                \vcenter{\hbox{
                \includestandalone[scale=0.8]{figures/fs_internal_barycenter_rel1}
                }}
                \;\; = \;\;
                (-A^2)^{-\frac{\mu + \nu}{4}}
                \vcenter{\hbox{
                \includestandalone[scale=0.8]{figures/fs_internal_barycenter_rel2}
                }}
                \;.
            \end{gather*}

            \item If $Cv$ is an internal edge of $Y_i$ but the barycenter adjacent to $Cv$ is not an internal vertex of $Y_i$, Theorem \ref{thm:SplittingMapIsABimodHom} identifies the left $\SkAlg(D_3)$ module structure on $\Sk(Y_i)$ with the left $\SkAlg(D_3) \;\overline{\otimes}\; \SkAlg(D_4)$ module structure on $\Sk(Sf)\otimes  \Sk(Y_{i+1})$, yielding the following relation:
            \[
            (-A^2)^{\frac{\mu + \nu}{4}}
            \vcenter{\hbox{
            \includestandalone[scale=0.8]{figures/fs_external_barycenter_rel1}
            }}
            \;\; = \;\;
            (-A^2)^{-\frac{\mu + \nu}{4}}
            \vcenter{\hbox{
            \includestandalone[scale=0.8]{figures/fs_external_barycenter_rel2}
            }}.
            \]
            In the above figure, $Sf_j$ is the other face suspension that must have previously been split from the tetrahedron shown. 
            
            \item If $Cv$ is already on the boundary of $Y_{i}$, then there must be some $j < i$ such that $Cv$ is an internal edge of $Y_{j}$ but not of $Y_{j+1}$. That is, one of the three face suspensions surrounding $Cv$, $Sf_j,$ has already been split. Then the left $\SkAlg(D_3)$-module structure on
            $\Sk(Sf) \overline{\otimes} \Sk(Sf_{k})$, where $Sf_k$ is the third face suspension associated to $Cv$, is induced from the splitting homomorphism $\SkAlg(D_4) \rightarrow \SkAlg(D_3)\otimes \SkAlg(D_3)$: 
            \[
            \vcenter{\hbox{
            \includestandalone[scale=0.8]{figures/fs_splitting_cor3}
            }}
            \;\;\mapsto\;\;
            \sum_{\epsilon\in\{\pm\}}
            \vcenter{\hbox{
            \includestandalone[scale=0.8]{figures/fs_splitting_cor4}
            }}
            \;.
            \]
            In the left hand side of the above figure, we show the two components of the boundary marking of $Y_i$ associated to edge cones adjacent to $Cv$. 
            On the right hand side, we've obtained two new boundary marking components after having split along the edge cone connecting two face suspensions.
        \end{enumerate}
    \end{enumerate}
    This concludes the proof.
\end{proof}

It is a valuable exercise to prove these two corollaries using Theorem \ref{thm:partialSplittingHomomorphism}, cutting along faces of tetrahedra and edge cones, respectively. 
Using Theorem \ref{thm:partialSplittingHomomorphism} directly recovers the symmetric versions of the internal edge and edge cone relations mentioned in Remark \ref{rmk:SplittingRelationSymmetricalForm}.

\subsection{Reduced stated skein modules}
On top of cutting the 3-manifold into pieces, we can further simplify our study of skein modules by taking a specific quotient, called the \emph{reduced} (stated) skein module. 
As we will see in Section \ref{sec:3dQuantumTrace}, it turns out that the 3d quantum trace map we need to construct factors through this quotient, so we just need a good understanding of the structure of these reduced skein modules. 

\begin{defn}
The elements of $\SkAlg(D_n)$ shown in Figure \ref{fig:bad_arcs} are called \emph{bad arcs}. 
\begin{figure}[H]
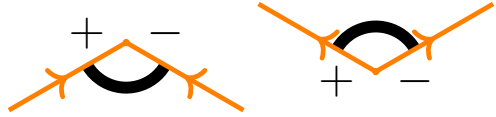

    \centering
    \includestandalone[scale=1.5]{figures/bad_arc1}
    \includestandalone[scale=1.5]{figures/bad_arc2}
    \caption{Bad arcs (seen from outside of $Y$).}
    \label{fig:bad_arcs}
\end{figure}
\end{defn}

\begin{defn}
The \emph{reduced stated skein module} $\overline{\Sk}(Y, \Gamma)$ is the quotient
\[
\overline{\Sk}(Y, \Gamma) := 
\frac{\Sk(Y, \Gamma)}{I^{\mathrm{bad}, +}\Sk(Y,\Gamma) + \Sk(Y,\Gamma)I^{\mathrm{bad}, -}},
\]
where $I^{\mathrm{bad}, +}$ is the right ideal of $\otimes_{v \in V(\Gamma)^{+}}\SkAlg(D_{\deg v})$ generated by the bad arcs near the sinks, and $I^{\mathrm{bad}, -}$ is the left ideal of $\otimes_{w \in V(\Gamma)^{-}}\SkAlg(D_{\deg w})$ generated by the bad arcs near the sources. 
That is, we set the action of the bad arcs to be $0$. 
\end{defn}

The reduced stated skein algebras of biangles and triangles have particularly simple descriptions: 
\begin{thm}[\cite{BW, CL}]
\begin{enumerate}
\item The reduced stated skein algebra of a biangle is the algebra of Laurent polynomials in one variable:
\[
\overline{\SkAlg}(D_2) =: \mathbb{B} \cong R[x,x^{-1}], 
\]
with
\[
x
\mapsto 
\vcenter{\hbox{
\begin{tikzpicture}
\draw[very thick] (0, {-sqrt(3)/2}) to[out=30, in=-90] (1/2, 0) to[out=90, in=-30] (0, {sqrt(3)/2});
\draw[very thick] (0, {-sqrt(3)/2}) to[out=150, in=-90] (-1/2, 0) to[out=90, in=-150] (0, {sqrt(3)/2});
\filldraw[fill=white] (0, {sqrt(3)/2}) circle (0.1);
\filldraw[fill=white] (0, {-sqrt(3)/2}) circle (0.1);
\draw[ultra thick] (-1/2, 0) -- (1/2, 0);
\node[anchor=east] at (-1/2, 0){$+$};
\node[anchor=west] at (1/2, 0){$+$};
\end{tikzpicture}
}}
\;,\quad
x^{-1}
\mapsto 
\vcenter{\hbox{
\begin{tikzpicture}
\draw[very thick] (0, {-sqrt(3)/2}) to[out=30, in=-90] (1/2, 0) to[out=90, in=-30] (0, {sqrt(3)/2});
\draw[very thick] (0, {-sqrt(3)/2}) to[out=150, in=-90] (-1/2, 0) to[out=90, in=-150] (0, {sqrt(3)/2});
\filldraw[fill=white] (0, {sqrt(3)/2}) circle (0.1);
\filldraw[fill=white] (0, {-sqrt(3)/2}) circle (0.1);
\draw[ultra thick] (-1/2, 0) -- (1/2, 0);
\node[anchor=east] at (-1/2, 0){$-$};
\node[anchor=west] at (1/2, 0){$-$};
\end{tikzpicture}
}}
.
\]
\item The reduced stated skein algebra of a triangle is the \emph{triangle algebra}: 
\begin{equation}\label{eq:trianglealg}
\overline{\SkAlg}(D_3) =: \mathbb{T} \cong \frac{R\langle \alpha^{\pm 1}, \beta^{\pm 1}, \gamma^{\pm 1}\rangle}{\langle \beta\alpha = A\alpha\beta, \gamma\beta = A\beta\gamma, \alpha\gamma = A\gamma\alpha \rangle},
\end{equation}
with
\begin{gather*}
\alpha 
\mapsto
\vcenter{\hbox{
\begin{tikzpicture}
\coordinate (a) at (0, 1);
\coordinate (b) at ({-sqrt(3)/2}, -1/2);
\coordinate (c) at ({sqrt(3)/2}, -1/2);
\coordinate (ab) at ({-sqrt(3)/4}, 1/4);
\coordinate (ac) at ({sqrt(3)/4}, 1/4);
\coordinate (bc) at (0, -1/2);
\draw[very thick] (a) -- (b) -- (c) -- cycle;
\filldraw[fill=white] (a) circle (0.1);
\filldraw[fill=white] (b) circle (0.1);
\filldraw[fill=white] (c) circle (0.1);
\draw[ultra thick] (ac) to[out=-150, in=-30] (ab);
\node[anchor=east] at (ab){$+$};
\node[anchor=west] at (ac){$+$};
\end{tikzpicture}
}}
\;,\quad
\beta 
\mapsto 
\vcenter{\hbox{
\begin{tikzpicture}
\coordinate (a) at (0, 1);
\coordinate (b) at ({-sqrt(3)/2}, -1/2);
\coordinate (c) at ({sqrt(3)/2}, -1/2);
\coordinate (ab) at ({-sqrt(3)/4}, 1/4);
\coordinate (ac) at ({sqrt(3)/4}, 1/4);
\coordinate (bc) at (0, -1/2);
\draw[very thick] (a) -- (b) -- (c) -- cycle;
\filldraw[fill=white] (a) circle (0.1);
\filldraw[fill=white] (b) circle (0.1);
\filldraw[fill=white] (c) circle (0.1);
\draw[ultra thick] (ab) to[out=-30, in=90] (bc);
\node[anchor=east] at (ab){$+$};
\node[anchor=north] at (bc){$+$};
\end{tikzpicture}
}}
\;,\quad
\gamma
\mapsto 
\vcenter{\hbox{
\begin{tikzpicture}
\coordinate (a) at (0, 1);
\coordinate (b) at ({-sqrt(3)/2}, -1/2);
\coordinate (c) at ({sqrt(3)/2}, -1/2);
\coordinate (ab) at ({-sqrt(3)/4}, 1/4);
\coordinate (ac) at ({sqrt(3)/4}, 1/4);
\coordinate (bc) at (0, -1/2);
\draw[very thick] (a) -- (b) -- (c) -- cycle;
\filldraw[fill=white] (a) circle (0.1);
\filldraw[fill=white] (b) circle (0.1);
\filldraw[fill=white] (c) circle (0.1);
\draw[ultra thick] (bc) to[out=90, in=-150] (ac);
\node[anchor=west] at (ac){$+$};
\node[anchor=north] at (bc){$+$};
\end{tikzpicture}
}}
\;,
\\
\alpha^{-1} 
\mapsto
\vcenter{\hbox{
\begin{tikzpicture}
\coordinate (a) at (0, 1);
\coordinate (b) at ({-sqrt(3)/2}, -1/2);
\coordinate (c) at ({sqrt(3)/2}, -1/2);
\coordinate (ab) at ({-sqrt(3)/4}, 1/4);
\coordinate (ac) at ({sqrt(3)/4}, 1/4);
\coordinate (bc) at (0, -1/2);
\draw[very thick] (a) -- (b) -- (c) -- cycle;
\filldraw[fill=white] (a) circle (0.1);
\filldraw[fill=white] (b) circle (0.1);
\filldraw[fill=white] (c) circle (0.1);
\draw[ultra thick] (ac) to[out=-150, in=-30] (ab);
\node[anchor=east] at (ab){$-$};
\node[anchor=west] at (ac){$-$};
\end{tikzpicture}
}}
\;,\quad
\beta^{-1} 
\mapsto 
\vcenter{\hbox{
\begin{tikzpicture}
\coordinate (a) at (0, 1);
\coordinate (b) at ({-sqrt(3)/2}, -1/2);
\coordinate (c) at ({sqrt(3)/2}, -1/2);
\coordinate (ab) at ({-sqrt(3)/4}, 1/4);
\coordinate (ac) at ({sqrt(3)/4}, 1/4);
\coordinate (bc) at (0, -1/2);
\draw[very thick] (a) -- (b) -- (c) -- cycle;
\filldraw[fill=white] (a) circle (0.1);
\filldraw[fill=white] (b) circle (0.1);
\filldraw[fill=white] (c) circle (0.1);
\draw[ultra thick] (ab) to[out=-30, in=90] (bc);
\node[anchor=east] at (ab){$-$};
\node[anchor=north] at (bc){$-$};
\end{tikzpicture}
}}
\;,\quad
\gamma^{-1}
\mapsto 
\vcenter{\hbox{
\begin{tikzpicture}
\coordinate (a) at (0, 1);
\coordinate (b) at ({-sqrt(3)/2}, -1/2);
\coordinate (c) at ({sqrt(3)/2}, -1/2);
\coordinate (ab) at ({-sqrt(3)/4}, 1/4);
\coordinate (ac) at ({sqrt(3)/4}, 1/4);
\coordinate (bc) at (0, -1/2);
\draw[very thick] (a) -- (b) -- (c) -- cycle;
\filldraw[fill=white] (a) circle (0.1);
\filldraw[fill=white] (b) circle (0.1);
\filldraw[fill=white] (c) circle (0.1);
\draw[ultra thick] (bc) to[out=90, in=-150] (ac);
\node[anchor=west] at (ac){$-$};
\node[anchor=north] at (bc){$-$};
\end{tikzpicture}
}}
\;.
\end{gather*}
\end{enumerate}
\end{thm}
These isomorphisms, along with the 2d splitting map, were the key elements in the construction of the 2d quantum trace map in \cite{BW} and \cite{Le}.

In the 3d setup, it is easy to see that everything we discussed in this section carries over to reduced skein modules. 
That is, we have bimodule structures:
\begin{prop}
$\overline{\Sk}(Y, \Gamma)$ is a $\otimes_{v \in V(\Gamma)^{+}}\overline{\SkAlg}(D_{\deg v})$-$\otimes_{w \in V(\Gamma)^{-}}\overline{\SkAlg}(D_{\deg w})$-bimodule. 
\end{prop}
Moreover, the splitting map (Theorem \ref{thm:SplittingMap}) descends to reduced skein modules. 
\begin{prop} \label{prop:splitting_reduced}
In the setting of Theorem \ref{thm:SplittingMap}, 
the splitting map $\sigma$ descends to a $\otimes_{v \in V(\Gamma)^{+}}\overline{\SkAlg}(D_{\deg v})$-$\otimes_{w \in V(\Gamma)^{-}}\overline{\SkAlg}(D_{\deg w})$-bimodule homomorphism
\[
\overline{\sigma} : \overline{\Sk}(Y, \Gamma) \rightarrow 
\overline{\Sk}(Y_1, \Gamma_1)
\;\overline{\otimes}\;
\overline{\Sk}(Y_2, \Gamma_2).
\]
\end{prop}
In particular, the following is an easy corollary of 
Proposition \ref{prop:splitting_reduced} and Corollary \ref{cor:fs_splitting_unreduced}. 
\begin{cor}\label{cor:fs_splitting_reduced}
    Let $Y= \cup_{f\in \mathcal{F}} Sf$ be a decomposition of an ideally triangulated $3$-manifold (without boundary except for cusps at infinity) into face suspensions. Then, there is a well-defined splitting map 
    \begin{align*}
        \overline{\sigma}:\overline{\Sk}(Y) &\rightarrow \overline{\bigotimes}_{f\in \mathcal{F}}\overline{\Sk}(Sf),\\
        [L] &\mapsto  \qty[\sum_{\vec{\epsilon} \;\in\; \{\substack{\mathrm{compatible}\\ \mathrm{states}}\}}\otimes_{f\in \mathcal{F}} [L_{f}^{\vec{\epsilon}}]],
    \end{align*}
where $\overline{\bigotimes}_{f\in \mathcal{F}}\overline{\Sk}(Sf)$ denotes the quotient of the naive tensor product $\bigotimes_{f\in \mathcal{F}}\overline{\Sk}(Sf)$ by the following relations:
\begin{enumerate}
    \item \label{item:redSkMod1} For each internal edge $e \in \mathcal{E}$ 
    we have the following relations among right actions on $\bigotimes_{f\in \mathcal{F}}\overline{\Sk}(Sf)$: 
    \begin{gather*} 
    \vcenter{\hbox{
    \includestandalone[scale=0.8]{figures/fs_reduced_splitting_relation1}
    }}
    \;\; = \;\;
    (-A^2)^{\epsilon}
    \vcenter{\hbox{
    \includestandalone[scale=0.8]{figures/fs_reduced_splitting_relation2}
    }}
    \;.
    \end{gather*}
    \item \label{item:redSkMod2} Around each vertex cone, we have the following relation among left actions on $\bigotimes_{f\in \mathcal{F}} \overline{\Sk}(Sf)$:
    \begin{gather*} 
    \vcenter{\hbox{
    \includestandalone[scale=0.8]{figures/fs_reduced_splitting_relation3}
    }}
    \;\; = \;\;
    (-A^{2})^{-\epsilon}
    \vcenter{\hbox{
    \includestandalone[scale=0.8]{figures/fs_reduced_splitting_relation4}
    }}
    \;.
    \end{gather*}
    \item \label{item:redSkMod3} Around each vertex cone, we have the following relation among left actions on $\bigotimes_{f\in \mathcal{F}} \overline{\Sk}(Sf)$:
    \begin{gather*} 
    \vcenter{\hbox{
    \includestandalone[scale=0.8]{figures/fs_reduced_splitting_relation6}}}
    \;\; + \;\;
    \vcenter{\hbox{
    \includestandalone[scale=0.8]{figures/fs_reduced_splitting_relation7}}}
    \;\; - \;\;
    \vcenter{\hbox{
    \includestandalone[scale=0.8]{figures/fs_reduced_splitting_relation5}
    }}    
    \;\; = 0 \;\;
    \;.
    \end{gather*}
\end{enumerate}
\end{cor}

Corollary \ref{cor:fs_splitting_reduced} is central to our construction of the 3d quantum trace map in Section \ref{sec:3dQuantumTrace}, 
as it allows us to define the quantum trace for each face suspension first and then glue them together.


\section{Structure theorems for 3-balls} \label{sec:struc_theorems}
In this section, we study skein modules of 3-balls whose boundary is combinatorially foliated. 
In particular, we will completely characterize the skein modules of 3-balls. 
Together with the splitting map studied in Section \ref{sec:SkeinModules}, the results of this section will constitute the main ingredient in our construction of the 3d quantum trace map in Section \ref{sec:3dQuantumTrace}.

\subsection{Stated skein modules of 3-balls}

\begin{lem}\label{lem:Sk(B)isCyclic}
Let $B$ be a 3-ball whose boundary is combinatorially foliated (in the sense of Definition \ref{defn:CombinatorialFoliation}), and let $\Gamma \subset \partial B = S^2$ be the associated boundary marking. 
Then, as a $\otimes_{v \in V(\Gamma)^{+}}\SkAlg(D_{\deg v})$-$\otimes_{w \in V(\Gamma)^{-}}\SkAlg(D_{\deg w})$-bimodule, $\Sk(B, \Gamma)$ is a cyclic module spanned by the empty skein $[\emptyset] \in \Sk(B, \Gamma)$. 
\end{lem}
\begin{proof}
First, note that adding a puncture at the center of $B^3$ doesn't change the skein module, as any isotopy can be smoothly isotoped away from that puncture. 
Once we have added a puncture, then using the obvious deformation retraction of $B^3 \setminus \{0\}$ to its boundary $\partial B^3 = S^2$, we can isotope any tangle $L$ in $B^3\setminus \{0\}$ so that it is very close to the boundary. 
In other words, we will have a diagram of the tangle $L$ on the boundary. 

Recall that the boundary $\partial B^3 = S^2$ is equipped with a combinatorial foliation, meaning, after adding some punctures, it is decomposed into elementary quadrilaterals:
\[
\vcenter{\hbox{
\begin{tikzpicture}
\filldraw[very thick, draw=darkgray, fill=lightgray] (-1, 0) -- (0, 1) -- (1, 0) -- (0, -1) -- cycle;
\draw[very thick, blue, dashed] (-1, 0) -- (1, 0);
\filldraw[draw=black, fill=white] (1, 0) circle (0.1);
\filldraw[draw=black, fill=white] (-1, 0) circle (0.1);
\draw[very thick, orange] (0, -1) -- (0, 1);
\draw[very thick, orange, ->] (0, -1) -- (0, 0);
\filldraw[orange] (0, 1) circle (0.05);
\filldraw[orange] (0, -1) circle (0.05);
\end{tikzpicture}
}}.
\]
In the picture above, the blue dashed line is a generic leaf of the associated foliation, which will play an important role in our proof. 

Pick such blue dashed lines (i.e.\ generic leaves), one for each elementary quadrilateral of the combinatorial foliation. 
We will say that our tangle $L \subset B$ is in \emph{general position} with respect to the blue dashed lines if
\begin{enumerate}
\item its diagram on $S^2$ (i.e.\ projection to the boundary) is transverse to the blue dashed lines,
\item no crossing in the diagram of $L$ meets the blue dashed lines, and
\item at each intersection between the diagram of $L$ and the blue dashed lines, $L$ (as a ribbon tangle) is flat on $S^2$. 
\end{enumerate}
Note, in the moduli space of smoothly embedded tangles, generic points are in general position, and tangles which are in non-general positions are of real codimension $\geq 1$ in the moduli space. 
Hence, we can assume that $L$ is in general position with respect to the blue dashed lines. 

Then, in each elementary quadrilateral, we can isotope $L$ by pulling the intersection points toward the orange marking along the blue dashed line. 
We can then apply a sequence of stated skein relations to break up the strands in order to get a linear combination of tangle diagrams on $S^2$ which do not intersect with the blue dashed lines at all. 
See the figure below for an illustration:
\begin{gather*}
\vcenter{\hbox{
\begin{tikzpicture}[xscale=0.5, scale=1.7]
\filldraw[lightgray] (-2.5, -0.5) -- (-2.5, 0.5) -- (2.5, 0.5) -- (2.5, -0.5) -- cycle;
\draw[line width=3] (-1.5, -0.5) -- (-1.5, 0.5);
\draw[line width=3] (-1, -0.5) -- (-1, 0.5);
\draw[line width=3] (1.25, -0.5) -- (1.25, 0.5);
\draw[very thick, blue, dashed] (-2.5, 0) -- (2.5, 0);
\draw[very thick, orange] (0, -0.5) -- (0, 0.5);
\draw[very thick, orange, ->] (0, -0.5) -- (0, 0);
\end{tikzpicture}
}}
\;\;=\;\;
\sum_{\epsilon_1, \epsilon_2, \epsilon_3 \in \{\pm 1\}}(-A^2)^{\frac{\epsilon_1 + \epsilon_2 + \epsilon_3}{2}}
\vcenter{\hbox{
\begin{tikzpicture}[xscale=0.5, scale=1.7]
\filldraw[lightgray] (-2.5, -0.5) -- (-2.5, 0.5) -- (2.5, 0.5) -- (2.5, -0.5) -- cycle;
\draw[line width=3] (1.25, -0.5) to[out=90, in=0] (0, -0.1);
\draw[line width=3] (1.25, 0.5) to[out=-90, in=0] (0, 0.1);
\draw[line width=3] (-1.5, -0.5) to[out=90, in=180] (0, -0.25);
\draw[line width=3] (-1.5, 0.5) to[out=-90, in=180] (0, 0.25);
\draw[line width=3] (-1, -0.5) to[out=90, in=180] (0, -0.4);
\draw[line width=3] (-1, 0.5) to[out=-90, in=180] (0, 0.4);
\draw[very thick, blue, dashed] (-2.5, 0) -- (2.5, 0);
\draw[very thick, orange] (0, -0.5) -- (0, 0.5);
\draw[very thick, orange, ->] (0, -0.5) -- (0, 0);
\node[anchor=west] at (0, 0.4){$\epsilon_1$};
\node[anchor=west] at (0, 0.25){$\epsilon_2$};
\node[anchor=east] at (0, 0.1){$\epsilon_3$};
\node[anchor=east] at (0, -0.1){$-\epsilon_3$};
\node[anchor=west] at (0, -0.25){$-\epsilon_2$};
\node[anchor=west] at (0, -0.4){$-\epsilon_1$};
\end{tikzpicture}
}}.
\end{gather*}
Once we have expressed $[L] \in \Sk(B,\Gamma)$ as a linear combination of tangle diagrams on $S^2$ that do not intersect with the blue dashed lines, we are done, 
as then all the components of the tangle are localized near some vertices of the marking $\Gamma$. 
That is, $[L] = \alpha_L \cdot [\emptyset]$ for some 
$\alpha_L \in \bigotimes_{v \in V(\Gamma)^{+}}\SkAlg(D_{\deg v})\otimes\bigotimes_{w \in V(\Gamma)^{-}}\SkAlg(D_{\deg w})^{\mathrm{op}}$. 
\end{proof}

\begin{cor}
In the setup of Lemma \ref{lem:Sk(B)isCyclic}, let $\mathrm{Ann}([\emptyset])$ be the left ideal of $\bigotimes_{v \in V(\Gamma)^{+}}\SkAlg(D_{\deg v})\otimes\bigotimes_{w \in V(\Gamma)^{-}}\SkAlg(D_{\deg w})^{\mathrm{op}}$ consisting of elements annihilating $[\emptyset] \in \Sk(B, \Gamma)$. 
Then, 
\[
\Sk(B, \Gamma) \cong 
\frac{\bigotimes_{v \in V(\Gamma)^{+}}\SkAlg(D_{\deg v})\otimes\bigotimes_{w \in V(\Gamma)^{-}}\SkAlg(D_{\deg w})^{\mathrm{op}}}{\mathrm{Ann}([\emptyset])}.
\]
\end{cor}
In the rest of this subsection, we will determine the ideal $\mathrm{Ann}([\emptyset])$.

\begin{thm}\label{thm:SkeinModuleOf3Ball}
In the setup of Lemma \ref{lem:Sk(B)isCyclic}, the skein module of $(B,\Gamma)$ has the following presentation: 
\[
\Sk(B, \Gamma) \cong 
\frac{\bigotimes_{v \in V(\Gamma)^{+}}\SkAlg(D_{\deg v})\otimes\bigotimes_{w \in V(\Gamma)^{-}}\SkAlg(D_{\deg w})^{\mathrm{op}}}{\mathrm{Ann}([\emptyset])},
\]
where $\mathrm{Ann}([\emptyset])$ is the left ideal generated by the following relations, one for each puncture of the combinatorial foliation of $\partial B$:\footnote{Here, we are drawing a hexagon for the sake of concreteness.In general, the face of $S^2 \setminus \Gamma$ associated to the puncture can be any $2n$-gon for some $n$, in which case we need to sum over $\epsilon_1, \cdots, \epsilon_{2n-2} \in \{\pm\}$. }  
\begin{equation}\label{eq:GeneratorsOfAnnhilator}
\vcenter{\hbox{
\begin{tikzpicture}
\draw[line width=5] (-3/16*2, {sqrt(3)*7/16*2}) to[out=-60, in=-120] (3/16*2, {sqrt(3)*7/16*2});
\node[anchor=south] at (-3/16*2, {sqrt(3)*7/16*2}){$\mu$};
\node[anchor=south] at (3/16*2, {sqrt(3)*7/16*2}){$\nu$};
\draw[blue, dashed] (0, 0) -- (1*2, 0);
\draw[blue, dashed] (0, 0) -- (-1*2, 0);
\draw[blue, dashed] (0, 0) -- (1/2*2, {sqrt(3)/2*2});
\draw[blue, dashed] (0, 0) -- (-1/2*2, {sqrt(3)/2*2});
\draw[blue, dashed] (0, 0) -- (1/2*2, {-sqrt(3)/2*2});
\draw[blue, dashed] (0, 0) -- (-1/2*2, {-sqrt(3)/2*2});
\draw[orange, very thick] (-3/4*2, {sqrt(3)/4*2}) -- (0, {sqrt(3)/2*2});
\draw[orange, very thick, ->] (-3/4*2, {sqrt(3)/4*2}) -- (-3/8*2, {sqrt(3)*3/8*2});
\draw[orange, very thick] (3/4*2, {sqrt(3)/4*2}) -- (0, {sqrt(3)/2*2});
\draw[orange, very thick, ->] (3/4*2, {sqrt(3)/4*2}) -- (3/8*2, {sqrt(3)*3/8*2});
\draw[orange, very thick] (3/4*2, {sqrt(3)/4*2}) -- (3/4*2, {-sqrt(3)/4*2});
\draw[orange, very thick, ->] (3/4*2, {sqrt(3)/4*2}) -- (3/4*2, 0);
\draw[orange, very thick] (0, {-sqrt(3)/2*2}) -- (3/4*2, {-sqrt(3)/4*2});
\draw[orange, very thick, ->] (0, {-sqrt(3)/2*2}) -- (3/8*2, {-sqrt(3)*3/8*2});
\draw[orange, very thick] (0, {-sqrt(3)/2*2}) -- (-3/4*2, {-sqrt(3)/4*2});
\draw[orange, very thick, ->] (0, {-sqrt(3)/2*2}) -- (-3/8*2, {-sqrt(3)*3/8*2});
\draw[orange, very thick] (-3/4*2, {sqrt(3)/4*2}) -- (-3/4*2, {-sqrt(3)/4*2});
\draw[orange, very thick, ->] (-3/4*2, {sqrt(3)/4*2}) -- (-3/4*2, 0);
\filldraw[fill=white] (0, 0) circle (0.1);
\filldraw[orange] (-3/4*2, {sqrt(3)/4*2}) circle (0.07);
\filldraw[orange] (3/4*2, {sqrt(3)/4*2}) circle (0.07);
\filldraw[orange] (0, {sqrt(3)/2*2}) circle (0.07);
\filldraw[orange] (0, {-sqrt(3)/2*2}) circle (0.07);
\filldraw[orange] (3/4*2, {-sqrt(3)/4*2}) circle (0.07);
\filldraw[orange] (-3/4*2, {-sqrt(3)/4*2}) circle (0.07);
\end{tikzpicture}
}}
\;\;=\;\;
\sum_{\epsilon_i \in \{\pm\}}(-A^2)^{\frac{\sum_{i} \epsilon_i}{2}}
\vcenter{\hbox{
\begin{tikzpicture}
\draw[line width=5] (-3*3/16*2, {sqrt(3)*5/16*2}) to[out=-60, in=0] (-3/4*2, {sqrt(3)/8*2});
\node[anchor=south] at (-3*3/16*2, {sqrt(3)*5/16*2}){$\mu$};
\node[anchor=east] at (-3/4*2, {sqrt(3)/8*2}){$-\epsilon_1$};
\draw[line width=5] (3/4*2, {sqrt(3)/8*2}) to[out=180, in=-120] (3*3/16*2, {sqrt(3)*5/16*2});
\node[anchor=south] at (3*3/16*2, {sqrt(3)*5/16*2}){$\nu$};
\node[anchor=west] at (3/4*2, {sqrt(3)/8*2}){$-\epsilon_4$};
\draw[line width=5] (-3*3/16*2, {-sqrt(3)*5/16*2}) to[out=60, in=0] (-3/4*2, {-sqrt(3)/8*2});
\node[anchor=east] at (-3/4*2, {-sqrt(3)/8*2}){$\epsilon_1$};
\node[anchor=north] at (-3*3/16*2, {-sqrt(3)*5/16*2}){$\epsilon_2$};
\draw[line width=5] (3/4*2, {-sqrt(3)/8*2}) to[out=180, in=120] (3*3/16*2, {-sqrt(3)*5/16*2});
\node[anchor=north] at (3*3/16*2, {-sqrt(3)*5/16*2}){$\epsilon_3$};
\node[anchor=west] at (3/4*2, {-sqrt(3)/8*2}){$\epsilon_4$};
\draw[line width=5] (-3/16*2, {-sqrt(3)*7/16*2}) to[out=60, in=120] (3/16*2, {-sqrt(3)*7/16*2});
\node[anchor=north] at (-3/16*2-0.1, {-sqrt(3)*7/16*2}){$-\epsilon_2$};
\node[anchor=north] at (3/16*2, {-sqrt(3)*7/16*2}){$-\epsilon_3$};
\draw[blue, dashed] (0, 0) -- (1*2, 0);
\draw[blue, dashed] (0, 0) -- (-1*2, 0);
\draw[blue, dashed] (0, 0) -- (1/2*2, {sqrt(3)/2*2});
\draw[blue, dashed] (0, 0) -- (-1/2*2, {sqrt(3)/2*2});
\draw[blue, dashed] (0, 0) -- (1/2*2, {-sqrt(3)/2*2});
\draw[blue, dashed] (0, 0) -- (-1/2*2, {-sqrt(3)/2*2});
\draw[orange, very thick] (-3/4*2, {sqrt(3)/4*2}) -- (0, {sqrt(3)/2*2});
\draw[orange, very thick, ->] (-3/4*2, {sqrt(3)/4*2}) -- (-3/8*2, {sqrt(3)*3/8*2});
\draw[orange, very thick] (3/4*2, {sqrt(3)/4*2}) -- (0, {sqrt(3)/2*2});
\draw[orange, very thick, ->] (3/4*2, {sqrt(3)/4*2}) -- (3/8*2, {sqrt(3)*3/8*2});
\draw[orange, very thick] (3/4*2, {sqrt(3)/4*2}) -- (3/4*2, {-sqrt(3)/4*2});
\draw[orange, very thick, ->] (3/4*2, {sqrt(3)/4*2}) -- (3/4*2, 0);
\draw[orange, very thick] (0, {-sqrt(3)/2*2}) -- (3/4*2, {-sqrt(3)/4*2});
\draw[orange, very thick, ->] (0, {-sqrt(3)/2*2}) -- (3/8*2, {-sqrt(3)*3/8*2});
\draw[orange, very thick] (0, {-sqrt(3)/2*2}) -- (-3/4*2, {-sqrt(3)/4*2});
\draw[orange, very thick, ->] (0, {-sqrt(3)/2*2}) -- (-3/8*2, {-sqrt(3)*3/8*2});
\draw[orange, very thick] (-3/4*2, {sqrt(3)/4*2}) -- (-3/4*2, {-sqrt(3)/4*2});
\draw[orange, very thick, ->] (-3/4*2, {sqrt(3)/4*2}) -- (-3/4*2, 0);
\filldraw[fill=white] (0, 0) circle (0.1);
\filldraw[orange] (-3/4*2, {sqrt(3)/4*2}) circle (0.07);
\filldraw[orange] (3/4*2, {sqrt(3)/4*2}) circle (0.07);
\filldraw[orange] (0, {sqrt(3)/2*2}) circle (0.07);
\filldraw[orange] (0, {-sqrt(3)/2*2}) circle (0.07);
\filldraw[orange] (3/4*2, {-sqrt(3)/4*2}) circle (0.07);
\filldraw[orange] (-3/4*2, {-sqrt(3)/4*2}) circle (0.07);
\end{tikzpicture}
}}
\;.
\end{equation}
\end{thm}

\begin{rmk}
The above relations can be equivalently written in a more symmetrical form as
\[
\sum_{\epsilon_i \in \{\pm\}}(-A^2)^{\frac{\sum_{i} \epsilon_i}{2}}
\vcenter{\hbox{
\begin{tikzpicture}
\draw[line width=5] (-3/16*2, {sqrt(3)*7/16*2}) to[out=-60, in=-120] (3/16*2, {sqrt(3)*7/16*2});
\node[anchor=south] at (-3/16*2, {sqrt(3)*7/16*2}){$\mu$};
\node[anchor=south] at (3/16*2, {sqrt(3)*7/16*2}){$\epsilon_5$};
\draw[line width=5] (-3*3/16*2, {sqrt(3)*5/16*2}) to[out=-60, in=0] (-3/4*2, {sqrt(3)/8*2});
\node[anchor=south] at (-3*3/16*2, {sqrt(3)*5/16*2}){$\nu$};
\node[anchor=east] at (-3/4*2, {sqrt(3)/8*2}){$-\epsilon_1$};
\draw[line width=5] (3/4*2, {sqrt(3)/8*2}) to[out=180, in=-120] (3*3/16*2, {sqrt(3)*5/16*2});
\node[anchor=south] at (3*3/16*2, {sqrt(3)*5/16*2}){$-\epsilon_5$};
\node[anchor=west] at (3/4*2, {sqrt(3)/8*2}){$-\epsilon_4$};
\draw[line width=5] (-3*3/16*2, {-sqrt(3)*5/16*2}) to[out=60, in=0] (-3/4*2, {-sqrt(3)/8*2});
\node[anchor=east] at (-3/4*2, {-sqrt(3)/8*2}){$\epsilon_1$};
\node[anchor=north] at (-3*3/16*2, {-sqrt(3)*5/16*2}){$\epsilon_2$};
\draw[line width=5] (3/4*2, {-sqrt(3)/8*2}) to[out=180, in=120] (3*3/16*2, {-sqrt(3)*5/16*2});
\node[anchor=north] at (3*3/16*2, {-sqrt(3)*5/16*2}){$\epsilon_3$};
\node[anchor=west] at (3/4*2, {-sqrt(3)/8*2}){$\epsilon_4$};
\draw[line width=5] (-3/16*2, {-sqrt(3)*7/16*2}) to[out=60, in=120] (3/16*2, {-sqrt(3)*7/16*2});
\node[anchor=north] at (-3/16*2-0.1, {-sqrt(3)*7/16*2}){$-\epsilon_2$};
\node[anchor=north] at (3/16*2, {-sqrt(3)*7/16*2}){$-\epsilon_3$};
\draw[blue, dashed] (0, 0) -- (1*2, 0);
\draw[blue, dashed] (0, 0) -- (-1*2, 0);
\draw[blue, dashed] (0, 0) -- (1/2*2, {sqrt(3)/2*2});
\draw[blue, dashed] (0, 0) -- (-1/2*2, {sqrt(3)/2*2});
\draw[blue, dashed] (0, 0) -- (1/2*2, {-sqrt(3)/2*2});
\draw[blue, dashed] (0, 0) -- (-1/2*2, {-sqrt(3)/2*2});
\draw[orange, very thick] (-3/4*2, {sqrt(3)/4*2}) -- (0, {sqrt(3)/2*2});
\draw[orange, very thick, ->] (-3/4*2, {sqrt(3)/4*2}) -- (-3/8*2, {sqrt(3)*3/8*2});
\draw[orange, very thick] (3/4*2, {sqrt(3)/4*2}) -- (0, {sqrt(3)/2*2});
\draw[orange, very thick, ->] (3/4*2, {sqrt(3)/4*2}) -- (3/8*2, {sqrt(3)*3/8*2});
\draw[orange, very thick] (3/4*2, {sqrt(3)/4*2}) -- (3/4*2, {-sqrt(3)/4*2});
\draw[orange, very thick, ->] (3/4*2, {sqrt(3)/4*2}) -- (3/4*2, 0);
\draw[orange, very thick] (0, {-sqrt(3)/2*2}) -- (3/4*2, {-sqrt(3)/4*2});
\draw[orange, very thick, ->] (0, {-sqrt(3)/2*2}) -- (3/8*2, {-sqrt(3)*3/8*2});
\draw[orange, very thick] (0, {-sqrt(3)/2*2}) -- (-3/4*2, {-sqrt(3)/4*2});
\draw[orange, very thick, ->] (0, {-sqrt(3)/2*2}) -- (-3/8*2, {-sqrt(3)*3/8*2});
\draw[orange, very thick] (-3/4*2, {sqrt(3)/4*2}) -- (-3/4*2, {-sqrt(3)/4*2});
\draw[orange, very thick, ->] (-3/4*2, {sqrt(3)/4*2}) -- (-3/4*2, 0);
\filldraw[fill=white] (0, 0) circle (0.1);
\filldraw[orange] (-3/4*2, {sqrt(3)/4*2}) circle (0.07);
\filldraw[orange] (3/4*2, {sqrt(3)/4*2}) circle (0.07);
\filldraw[orange] (0, {sqrt(3)/2*2}) circle (0.07);
\filldraw[orange] (0, {-sqrt(3)/2*2}) circle (0.07);
\filldraw[orange] (3/4*2, {-sqrt(3)/4*2}) circle (0.07);
\filldraw[orange] (-3/4*2, {-sqrt(3)/4*2}) circle (0.07);
\end{tikzpicture}
}}
\;\;=\;\;
\delta_{\mu, -\nu}\,(-A^2)^{\frac{\mu}{2}}
\;.
\]
\end{rmk}

\begin{proof}[Proof of Theorem \ref{thm:SkeinModuleOf3Ball}]
Firstly, it is clear that the left ideal $I$ generated by the elements in 
$\bigotimes_{v \in V(\Gamma)^{+}}\SkAlg(D_{\deg v})\otimes\bigotimes_{w \in V(\Gamma)^{-}}\SkAlg(D_{\deg w})^{\mathrm{op}}$
given by the difference between the LHS and RHS of the relations \eqref{eq:GeneratorsOfAnnhilator} above is contained in $\mathrm{Ann}([\emptyset])$, 
as one can apply a sequence of stated skein relations in $\Sk(B,\Gamma)$ to turn the LHS of the relation into the one on the RHS. 
That is, we have a surjective homomorphism
\[
f : \frac{\bigotimes_{v \in V(\Gamma)^{+}}\SkAlg(D_{\deg v})\otimes\bigotimes_{w \in V(\Gamma)^{-}}\SkAlg(D_{\deg w})^{\mathrm{op}}}{I}
\twoheadrightarrow 
\Sk(B,\Gamma). 
\]
To show that this is an isomorphism, 
it suffices to construct a map
\begin{equation}\label{eq:gmap}
g : \Sk(B,\Gamma) \rightarrow \frac{\bigotimes_{v \in V(\Gamma)^{+}}\SkAlg(D_{\deg v})\otimes\bigotimes_{w \in V(\Gamma)^{-}}\SkAlg(D_{\deg w})^{\mathrm{op}}}{I}
\end{equation}
such that $g\circ f$ is the identity map. 

In the proof of Lemma \ref{lem:Sk(B)isCyclic}, we have described an algorithm that, given a tangle diagram of $L$ on $S^2$ in general position, produces an element
\[
\alpha_L  \in \bigotimes_{v \in V(\Gamma)^{+}}\SkAlg(D_{\deg v})\otimes\bigotimes_{w \in V(\Gamma)^{-}}\SkAlg(D_{\deg w})^{\mathrm{op}}
\]
such that
\[
\alpha_L \cdot [\emptyset] = [L].
\]
To be more precise, a priori, this algorithm produces such an element $\alpha$ given a choice of ordering in which to break the strands up; we can either break up the strand to the left of the marking first or break up the one to the right of the marking first. 

We claim that the element $\alpha$ is independent of this choice of ordering in which we break up the strands. 
For this, we need to compare the following two different orderings in which we can break up the strands: 
\[
\vcenter{\hbox{
\begin{tikzpicture}[scale=1.3]
\filldraw[lightgray] (-2, -0.5) -- (-2, 0.5) -- (2, 0.5) -- (2, -0.5) -- cycle;
\draw[line width=3] (-1, -0.5) -- (-1, 0.5);
\draw[line width=3] (1, -0.5) -- (1, 0.5);
\draw[very thick, blue, dashed] (-2, 0) -- (2, 0);
\draw[very thick, orange] (0, -0.5) -- (0, 0.5);
\draw[very thick, orange, ->] (0, -0.5) -- (0, 0);
\end{tikzpicture}
}}
\;\;\rightarrow\;\;
\sum_{\mu, \nu\in \{\pm\}}
(-A^2)^{\frac{\mu+\nu}{2}}
\vcenter{\hbox{
\begin{tikzpicture}[scale=1.3]
\filldraw[lightgray] (-2, -0.5) -- (-2, 0.5) -- (2, 0.5) -- (2, -0.5) -- cycle;
\draw[line width=3] (-1, -0.5) to[out=90, in=180] (0, -0.3);
\draw[line width=3] (-1, 0.5) to[out=-90, in=180] (0, 0.3);
\draw[line width=3] (1, -0.5) to[out=90, in=0] (0, -0.1);
\draw[line width=3] (1, 0.5) to[out=-90, in=0] (0, 0.1);
\draw[very thick, blue, dashed] (-2, 0) -- (2, 0);
\draw[very thick, orange] (0, -0.5) -- (0, 0.5);
\draw[very thick, orange, ->] (0, -0.5) -- (0, 0);
\node[anchor=west] at (0, 0.3){$\mu$};
\node[anchor=east] at (0, 0.1){$\nu$};
\node[anchor=east] at (0, -0.1){$-\nu$};
\node[anchor=west] at (0, -0.3){$-\mu$};
\end{tikzpicture}
}}
\]
or
\[
\vcenter{\hbox{
\begin{tikzpicture}[scale=1.3]
\filldraw[lightgray] (-2, -0.5) -- (-2, 0.5) -- (2, 0.5) -- (2, -0.5) -- cycle;
\draw[line width=3] (-1, -0.5) -- (-1, 0.5);
\draw[line width=3] (1, -0.5) -- (1, 0.5);
\draw[very thick, blue, dashed] (-2, 0) -- (2, 0);
\draw[very thick, orange] (0, -0.5) -- (0, 0.5);
\draw[very thick, orange, ->] (0, -0.5) -- (0, 0);
\end{tikzpicture}
}}
\;\;\rightarrow\;\;
\sum_{\mu, \nu\in \{\pm\}}
(-A^2)^{\frac{\mu+\nu}{2}}
\vcenter{\hbox{
\begin{tikzpicture}[scale=1.3]
\filldraw[lightgray] (-2, -0.5) -- (-2, 0.5) -- (2, 0.5) -- (2, -0.5) -- cycle;
\draw[line width=3] (-1, -0.5) to[out=90, in=180] (0, -0.1);
\draw[line width=3] (-1, 0.5) to[out=-90, in=180] (0, 0.1);
\draw[line width=3] (1, -0.5) to[out=90, in=0] (0, -0.3);
\draw[line width=3] (1, 0.5) to[out=-90, in=0] (0, 0.3);
\draw[very thick, blue, dashed] (-2, 0) -- (2, 0);
\draw[very thick, orange] (0, -0.5) -- (0, 0.5);
\draw[very thick, orange, ->] (0, -0.5) -- (0, 0);
\node[anchor=west] at (0, 0.1){$\mu$};
\node[anchor=east] at (0, 0.3){$\nu$};
\node[anchor=east] at (0, -0.3){$-\nu$};
\node[anchor=west] at (0, -0.1){$-\mu$};
\end{tikzpicture}
}}.
\]
That is, if $v$ and $w$ are the sink and the source vertex of the orange marking in the figure, we need to check that the two right-hand side figures represent the same element in 
$\SkAlg(D_{\deg v}) \otimes \SkAlg(D_{\deg w})^{\mathrm{op}}$, 
which is straightforward using Lemma \ref{lem:StatedSkeinRelLemma}:  
\begin{align*}
&\sum_{\mu, \nu\in \{\pm\}}
(-A^2)^{\frac{\mu+\nu}{2}}
\vcenter{\hbox{
\begin{tikzpicture}[xscale=0.7, scale=1.3]
\filldraw[lightgray] (-1.5, -0.5) -- (-1.5, 0.5) -- (1.5, 0.5) -- (1.5, -0.5) -- cycle;
\draw[line width=3] (-1, -0.5) to[out=90, in=180] (0, -0.3);
\draw[line width=3] (-1, 0.5) to[out=-90, in=180] (0, 0.3);
\draw[line width=3] (1, -0.5) to[out=90, in=0] (0, -0.1);
\draw[line width=3] (1, 0.5) to[out=-90, in=0] (0, 0.1);
\draw[very thick, blue, dashed] (-1.5, 0) -- (1.5, 0);
\draw[very thick, orange] (0, -0.5) -- (0, 0.5);
\draw[very thick, orange, ->] (0, -0.5) -- (0, 0);
\node[anchor=west] at (0, 0.3){$\mu$};
\node[anchor=east] at (0, 0.1){$\nu$};
\node[anchor=east] at (0, -0.1){$-\nu$};
\node[anchor=west] at (0, -0.3){$-\mu$};
\end{tikzpicture}
}}
-
\sum_{\mu, \nu\in \{\pm\}}
(-A^2)^{\frac{\mu+\nu}{2}}
\vcenter{\hbox{
\begin{tikzpicture}[xscale=0.7, scale=1.3]
\filldraw[lightgray] (-1.5, -0.5) -- (-1.5, 0.5) -- (1.5, 0.5) -- (1.5, -0.5) -- cycle;
\draw[line width=3] (-1, -0.5) to[out=90, in=180] (0, -0.1);
\draw[line width=3] (-1, 0.5) to[out=-90, in=180] (0, 0.1);
\draw[line width=3] (1, -0.5) to[out=90, in=0] (0, -0.3);
\draw[line width=3] (1, 0.5) to[out=-90, in=0] (0, 0.3);
\draw[very thick, blue, dashed] (-1.5, 0) -- (1.5, 0);
\draw[very thick, orange] (0, -0.5) -- (0, 0.5);
\draw[very thick, orange, ->] (0, -0.5) -- (0, 0);
\node[anchor=west] at (0, 0.1){$\mu$};
\node[anchor=east] at (0, 0.3){$\nu$};
\node[anchor=east] at (0, -0.3){$-\nu$};
\node[anchor=west] at (0, -0.1){$-\mu$};
\end{tikzpicture}
}}
\\
&= \sum_{\mu\in \{\pm\}}
(-A^2)^{\mu}
\vcenter{\hbox{
\begin{tikzpicture}[xscale=0.7, scale=1.3]
\filldraw[lightgray] (-1.5, -0.5) -- (-1.5, 0.5) -- (1.5, 0.5) -- (1.5, -0.5) -- cycle;
\draw[line width=3] (-1, -0.5) to[out=90, in=180] (0, -0.3);
\draw[line width=3] (-1, 0.5) to[out=-90, in=180] (0, 0.3);
\draw[line width=3] (1, -0.5) to[out=90, in=0] (0, -0.1);
\draw[line width=3] (1, 0.5) to[out=-90, in=0] (0, 0.1);
\draw[very thick, blue, dashed] (-1.5, 0) -- (1.5, 0);
\draw[very thick, orange] (0, -0.5) -- (0, 0.5);
\draw[very thick, orange, ->] (0, -0.5) -- (0, 0);
\node[anchor=west] at (0, 0.3){$\mu$};
\node[anchor=east] at (0, 0.1){$\mu$};
\node[anchor=east] at (0, -0.1){$-\mu$};
\node[anchor=west] at (0, -0.3){$-\mu$};
\end{tikzpicture}
}}
+
\vcenter{\hbox{
\begin{tikzpicture}[xscale=0.7, scale=1.3]
\filldraw[lightgray] (-1.5, -0.5) -- (-1.5, 0.5) -- (1.5, 0.5) -- (1.5, -0.5) -- cycle;
\draw[line width=3] (-1, -0.5) to[out=90, in=180] (0, -0.3);
\draw[line width=3] (-1, 0.5) to[out=-90, in=180] (0, 0.3);
\draw[line width=3] (1, -0.5) to[out=90, in=0] (0, -0.1);
\draw[line width=3] (1, 0.5) to[out=-90, in=0] (0, 0.1);
\draw[very thick, blue, dashed] (-1.5, 0) -- (1.5, 0);
\draw[very thick, orange] (0, -0.5) -- (0, 0.5);
\draw[very thick, orange, ->] (0, -0.5) -- (0, 0);
\node[anchor=west] at (0, 0.3){$+$};
\node[anchor=east] at (0, 0.1){$-$};
\node[anchor=east] at (0, -0.1){$+$};
\node[anchor=west] at (0, -0.3){$-$};
\end{tikzpicture}
}}
+
\vcenter{\hbox{
\begin{tikzpicture}[xscale=0.7, scale=1.3]
\filldraw[lightgray] (-1.5, -0.5) -- (-1.5, 0.5) -- (1.5, 0.5) -- (1.5, -0.5) -- cycle;
\draw[line width=3] (-1, -0.5) to[out=90, in=180] (0, -0.3);
\draw[line width=3] (-1, 0.5) to[out=-90, in=180] (0, 0.3);
\draw[line width=3] (1, -0.5) to[out=90, in=0] (0, -0.1);
\draw[line width=3] (1, 0.5) to[out=-90, in=0] (0, 0.1);
\draw[very thick, blue, dashed] (-1.5, 0) -- (1.5, 0);
\draw[very thick, orange] (0, -0.5) -- (0, 0.5);
\draw[very thick, orange, ->] (0, -0.5) -- (0, 0);
\node[anchor=west] at (0, 0.3){$-$};
\node[anchor=east] at (0, 0.1){$+$};
\node[anchor=east] at (0, -0.1){$-$};
\node[anchor=west] at (0, -0.3){$+$};
\end{tikzpicture}
}}\\
&\quad -
\sum_{\mu\in \{\pm\}}
(-A^2)^{\mu}
\vcenter{\hbox{
\begin{tikzpicture}[xscale=0.7, scale=1.3]
\filldraw[lightgray] (-1.5, -0.5) -- (-1.5, 0.5) -- (1.5, 0.5) -- (1.5, -0.5) -- cycle;
\draw[line width=3] (-1, -0.5) to[out=90, in=180] (0, -0.1);
\draw[line width=3] (-1, 0.5) to[out=-90, in=180] (0, 0.1);
\draw[line width=3] (1, -0.5) to[out=90, in=0] (0, -0.3);
\draw[line width=3] (1, 0.5) to[out=-90, in=0] (0, 0.3);
\draw[very thick, blue, dashed] (-1.5, 0) -- (1.5, 0);
\draw[very thick, orange] (0, -0.5) -- (0, 0.5);
\draw[very thick, orange, ->] (0, -0.5) -- (0, 0);
\node[anchor=west] at (0, 0.1){$\mu$};
\node[anchor=east] at (0, 0.3){$\mu$};
\node[anchor=east] at (0, -0.3){$-\mu$};
\node[anchor=west] at (0, -0.1){$-\mu$};
\end{tikzpicture}
}}
-
\vcenter{\hbox{
\begin{tikzpicture}[xscale=0.7, scale=1.3]
\filldraw[lightgray] (-1.5, -0.5) -- (-1.5, 0.5) -- (1.5, 0.5) -- (1.5, -0.5) -- cycle;
\draw[line width=3] (-1, -0.5) to[out=90, in=180] (0, -0.1);
\draw[line width=3] (-1, 0.5) to[out=-90, in=180] (0, 0.1);
\draw[line width=3] (1, -0.5) to[out=90, in=0] (0, -0.3);
\draw[line width=3] (1, 0.5) to[out=-90, in=0] (0, 0.3);
\draw[very thick, blue, dashed] (-1.5, 0) -- (1.5, 0);
\draw[very thick, orange] (0, -0.5) -- (0, 0.5);
\draw[very thick, orange, ->] (0, -0.5) -- (0, 0);
\node[anchor=west] at (0, 0.1){$+$};
\node[anchor=east] at (0, 0.3){$-$};
\node[anchor=east] at (0, -0.3){$+$};
\node[anchor=west] at (0, -0.1){$-$};
\end{tikzpicture}
}}
-
\vcenter{\hbox{
\begin{tikzpicture}[xscale=0.7, scale=1.3]
\filldraw[lightgray] (-1.5, -0.5) -- (-1.5, 0.5) -- (1.5, 0.5) -- (1.5, -0.5) -- cycle;
\draw[line width=3] (-1, -0.5) to[out=90, in=180] (0, -0.1);
\draw[line width=3] (-1, 0.5) to[out=-90, in=180] (0, 0.1);
\draw[line width=3] (1, -0.5) to[out=90, in=0] (0, -0.3);
\draw[line width=3] (1, 0.5) to[out=-90, in=0] (0, 0.3);
\draw[very thick, blue, dashed] (-1.5, 0) -- (1.5, 0);
\draw[very thick, orange] (0, -0.5) -- (0, 0.5);
\draw[very thick, orange, ->] (0, -0.5) -- (0, 0);
\node[anchor=west] at (0, 0.1){$-$};
\node[anchor=east] at (0, 0.3){$+$};
\node[anchor=east] at (0, -0.3){$-$};
\node[anchor=west] at (0, -0.1){$+$};
\end{tikzpicture}
}}\\
&= 
A^{-1}
\vcenter{\hbox{
\begin{tikzpicture}[xscale=0.7, scale=1.3]
\filldraw[lightgray] (-1.5, -0.5) -- (-1.5, 0.5) -- (1.5, 0.5) -- (1.5, -0.5) -- cycle;
\draw[line width=3] (-1, -0.5) to[out=90, in=180] (0, -0.1);
\draw[line width=3] (-1, 0.5) to[out=-90, in=180] (0, 0.3);
\draw[line width=3] (1, -0.5) to[out=90, in=0] (0, -0.3);
\draw[line width=3] (1, 0.5) to[out=-90, in=0] (0, 0.1);
\draw[very thick, blue, dashed] (-1.5, 0) -- (1.5, 0);
\draw[very thick, orange] (0, -0.5) -- (0, 0.5);
\draw[very thick, orange, ->] (0, -0.5) -- (0, 0);
\node[anchor=west] at (0, 0.3){$+$};
\node[anchor=east] at (0, 0.1){$-$};
\node[anchor=west] at (0, -0.1){$-$};
\node[anchor=east] at (0, -0.3){$+$};
\end{tikzpicture}
}}
+ 
\qty(
A^{-1}
\vcenter{\hbox{
\begin{tikzpicture}[xscale=0.7, scale=1.3]
\filldraw[lightgray] (-1.5, -0.5) -- (-1.5, 0.5) -- (1.5, 0.5) -- (1.5, -0.5) -- cycle;
\draw[line width=3] (-1, -0.5) to[out=90, in=180] (0, -0.1);
\draw[line width=3] (-1, 0.5) to[out=-90, in=180] (0, 0.3);
\draw[line width=3] (1, -0.5) to[out=90, in=0] (0, -0.3);
\draw[line width=3] (1, 0.5) to[out=-90, in=0] (0, 0.1);
\draw[very thick, blue, dashed] (-1.5, 0) -- (1.5, 0);
\draw[very thick, orange] (0, -0.5) -- (0, 0.5);
\draw[very thick, orange, ->] (0, -0.5) -- (0, 0);
\node[anchor=west] at (0, 0.3){$-$};
\node[anchor=east] at (0, 0.1){$+$};
\node[anchor=west] at (0, -0.1){$+$};
\node[anchor=east] at (0, -0.3){$-$};
\end{tikzpicture}
}}
+A^{-1}(A^2-A^{-2})
\vcenter{\hbox{
\begin{tikzpicture}[xscale=0.7, scale=1.3]
\filldraw[lightgray] (-1.5, -0.5) -- (-1.5, 0.5) -- (1.5, 0.5) -- (1.5, -0.5) -- cycle;
\draw[line width=3] (-1, -0.5) to[out=90, in=180] (0, -0.1);
\draw[line width=3] (-1, 0.5) to[out=-90, in=180] (0, 0.3);
\draw[line width=3] (1, -0.5) to[out=90, in=0] (0, -0.3);
\draw[line width=3] (1, 0.5) to[out=-90, in=0] (0, 0.1);
\draw[very thick, blue, dashed] (-1.5, 0) -- (1.5, 0);
\draw[very thick, orange] (0, -0.5) -- (0, 0.5);
\draw[very thick, orange, ->] (0, -0.5) -- (0, 0);
\node[anchor=west] at (0, 0.3){$-$};
\node[anchor=east] at (0, 0.1){$+$};
\node[anchor=west] at (0, -0.1){$-$};
\node[anchor=east] at (0, -0.3){$+$};
\end{tikzpicture}
}}
)\\
&\quad - 
A^{-1}\vcenter{\hbox{
\begin{tikzpicture}[xscale=0.7, scale=1.3]
\filldraw[lightgray] (-1.5, -0.5) -- (-1.5, 0.5) -- (1.5, 0.5) -- (1.5, -0.5) -- cycle;
\draw[line width=3] (-1, -0.5) to[out=90, in=180] (0, -0.1);
\draw[line width=3] (-1, 0.5) to[out=-90, in=180] (0, 0.3);
\draw[line width=3] (1, -0.5) to[out=90, in=0] (0, -0.3);
\draw[line width=3] (1, 0.5) to[out=-90, in=0] (0, 0.1);
\draw[very thick, blue, dashed] (-1.5, 0) -- (1.5, 0);
\draw[very thick, orange] (0, -0.5) -- (0, 0.5);
\draw[very thick, orange, ->] (0, -0.5) -- (0, 0);
\node[anchor=west] at (0, 0.3){$-$};
\node[anchor=east] at (0, 0.1){$+$};
\node[anchor=west] at (0, -0.1){$+$};
\node[anchor=east] at (0, -0.3){$-$};
\end{tikzpicture}
}}
-
\qty(
A^{-1}
\vcenter{\hbox{
\begin{tikzpicture}[xscale=0.7, scale=1.3]
\filldraw[lightgray] (-1.5, -0.5) -- (-1.5, 0.5) -- (1.5, 0.5) -- (1.5, -0.5) -- cycle;
\draw[line width=3] (-1, -0.5) to[out=90, in=180] (0, -0.1);
\draw[line width=3] (-1, 0.5) to[out=-90, in=180] (0, 0.3);
\draw[line width=3] (1, -0.5) to[out=90, in=0] (0, -0.3);
\draw[line width=3] (1, 0.5) to[out=-90, in=0] (0, 0.1);
\draw[very thick, blue, dashed] (-1.5, 0) -- (1.5, 0);
\draw[very thick, orange] (0, -0.5) -- (0, 0.5);
\draw[very thick, orange, ->] (0, -0.5) -- (0, 0);
\node[anchor=west] at (0, 0.3){$+$};
\node[anchor=east] at (0, 0.1){$-$};
\node[anchor=west] at (0, -0.1){$-$};
\node[anchor=east] at (0, -0.3){$+$};
\end{tikzpicture}
}}
+A^{-1}(A^2-A^{-2})
\vcenter{\hbox{
\begin{tikzpicture}[xscale=0.7, scale=1.3]
\filldraw[lightgray] (-1.5, -0.5) -- (-1.5, 0.5) -- (1.5, 0.5) -- (1.5, -0.5) -- cycle;
\draw[line width=3] (-1, -0.5) to[out=90, in=180] (0, -0.1);
\draw[line width=3] (-1, 0.5) to[out=-90, in=180] (0, 0.3);
\draw[line width=3] (1, -0.5) to[out=90, in=0] (0, -0.3);
\draw[line width=3] (1, 0.5) to[out=-90, in=0] (0, 0.1);
\draw[very thick, blue, dashed] (-1.5, 0) -- (1.5, 0);
\draw[very thick, orange] (0, -0.5) -- (0, 0.5);
\draw[very thick, orange, ->] (0, -0.5) -- (0, 0);
\node[anchor=west] at (0, 0.3){$-$};
\node[anchor=east] at (0, 0.1){$+$};
\node[anchor=west] at (0, -0.1){$-$};
\node[anchor=east] at (0, -0.3){$+$};
\end{tikzpicture}
}}
)\\
&= 0.
\end{align*}
Note, in the calculation above, all the isotopies or skein relations we are using happen away from the blue dashed lines. 
That is, they take place in the skein algebra $\SkAlg(D_{\deg v}) \otimes \SkAlg(D_{\deg w})^{\mathrm{op}}$. 
Therefore, 
the element $\alpha_L \in \bigotimes_{v \in V(\Gamma)^{+}}\SkAlg(D_{\deg v})\otimes\bigotimes_{w \in V(\Gamma)^{-}}\SkAlg(D_{\deg w})^{\mathrm{op}}$
we produce using this algorithm depends only on the diagram of $L$ on $S^2$ and not on the choice of ordering of breaking up the strands. 

Next, we need to study how $\alpha_L$ changes under isotopy of $L$. 
Any isotopy of $L$ is a finite composition of isotopies of the following types:
\begin{enumerate}[label= (\Roman*)]
\item\label{item:typeIboundaryisotopy} An isotopy in the class of tangles in general position with respect to the blue dashed lines. 
\item\label{item:typeIIboundaryisotopy} Half twist of the ribbon tangle near a blue dashed line. 
\[
\vcenter{\hbox{
\begin{tikzpicture}
\filldraw[lightgray] (-1, -1) -- (-1, 1) -- (1, 1) -- (1, -1) -- cycle;
\draw[line width=5] (0, -1) -- (0, 1);
\draw[very thick, blue, dashed] (-1, 0) -- (1, 0);
\end{tikzpicture}
}}
\;\leftrightarrow\;
\vcenter{\hbox{
\begin{tikzpicture}
\filldraw[lightgray] (-1, -1) -- (-1, 1) -- (1, 1) -- (1, -1) -- cycle;
\begin{scope}[shift={(0, 0.5)}, yscale=-0.5]
    \fill[bottom color=black, top color=white] (-0.085, 1) to[out=-90, in=100] (0, 0) to[out=80, in=-90] (0.085, 1)--cycle;
    \fill[top color=black, bottom color=white] (-0.085, -1) to[out=90, in=-100] (0, 0) to[out=-80, in=90] (0.085, -1)--cycle;
    \draw[line width=1] (-0.075, -1) to[out=90, in=-90] (0.075, 1);
\end{scope}
\begin{scope}[shift={(0, -0.5)}, yscale=0.5]
    \fill[top color=black, bottom color=white] (-0.085, 1) to[out=-90, in=100] (0, 0) to[out=80, in=-90] (0.085, 1)--cycle;
    \fill[bottom color=black, top color=white] (-0.085, -1) to[out=90, in=-100] (0, 0) to[out=-80, in=90] (0.085, -1)--cycle;
    \draw[line width=1] (-0.075, -1) to[out=90, in=-90] (0.075, 1);
\end{scope}
\draw[very thick, blue, dashed] (-1, 0) -- (1, 0);
\end{tikzpicture}
}}
\]
\item\label{item:typeIIIboundaryisotopy} Birth or annihilation of a pair of intersection points between the diagram of $L$ and a blue dashed line. 
\[
\vcenter{\hbox{
\begin{tikzpicture}
\filldraw[lightgray] (-1, -1) -- (-1, 1) -- (1, 1) -- (1, -1) -- cycle;
\draw[line width=5] (-0.4, -1) -- (-0.4, -0.8) to[out=90, in=180] (0, -0.4) to[out=-0, in=90] (0.4, -0.8) -- (0.4, -1);
\draw[very thick, blue, dashed] (-1, 0) -- (1, 0);
\end{tikzpicture}
}}
\;\leftrightarrow\;
\vcenter{\hbox{
\begin{tikzpicture}
\filldraw[lightgray] (-1, -1) -- (-1, 1) -- (1, 1) -- (1, -1) -- cycle;
\draw[line width=5] (-0.4, -1) -- (-0.4, 0) to[out=90, in=180] (0, 0.4) to[out=-0, in=90] (0.4, 0) -- (0.4, -1);
\draw[very thick, blue, dashed] (-1, 0) -- (1, 0);
\end{tikzpicture}
}}
\]
\item\label{item:typeIVboundaryisotopy} Pulling an end point of the tangle across a blue dashed line. 
\[
\vcenter{\hbox{
\begin{tikzpicture}
\filldraw[lightgray] (-1, -1) -- (-1, 1) -- (1, 1) -- (1, -1) -- cycle;
\draw[line width=5] (-0.4, -1) -- (-0.4, -0.8) to[out=90, in=180] (0, -0.4);
\draw[very thick, blue, dashed] (-1, 0) -- (1, 0);
\draw[very thick, orange] (0, -1) -- (0, 1);
\draw[very thick, orange, ->] (0, -1) -- (0, 0);
\node[anchor=west] at (0, -0.4){$\mu$};
\end{tikzpicture}
}}
\;\leftrightarrow\;
\vcenter{\hbox{
\begin{tikzpicture}
\filldraw[lightgray] (-1, -1) -- (-1, 1) -- (1, 1) -- (1, -1) -- cycle;
\draw[line width=5] (-0.4, -1) -- (-0.4, 0) to[out=90, in=180] (0, 0.4); 
\draw[very thick, blue, dashed] (-1, 0) -- (1, 0);
\draw[very thick, orange] (0, -1) -- (0, 1);
\draw[very thick, orange, ->] (0, -1) -- (0, 0);
\node[anchor=west] at (0, 0.4){$\mu$};
\end{tikzpicture}
}}
\]
\item\label{item:typeVboundaryisotopy} Moving a crossing across a blue dashed line. 
\[
\vcenter{\hbox{
\begin{tikzpicture}
\filldraw[lightgray] (-1, -1) -- (-1, 1) -- (1, 1) -- (1, -1) -- cycle;
\draw[line width=5] (0.4, -1) to[out=90, in=-90] (-0.4, 0) -- (-0.4, 1);
\draw[lightgray, line width=9] (-0.4, -1) to[out=90, in=-90] (0.4, 0) -- (0.4, 1);
\draw[line width=5] (-0.4, -1) to[out=90, in=-90] (0.4, 0) -- (0.4, 1);
\draw[very thick, blue, dashed] (-1, 0) -- (1, 0);
\end{tikzpicture}
}}
\;\leftrightarrow\;
\vcenter{\hbox{
\begin{tikzpicture}
\filldraw[lightgray] (-1, -1) -- (-1, 1) -- (1, 1) -- (1, -1) -- cycle;
\draw[line width=5] (0.4, -1) -- (0.4, 0) to[out=90, in=-90] (-0.4, 1);
\draw[lightgray, line width=9] (-0.4, -1) -- (-0.4, 0) to[out=90, in=-90] (0.4, 1);
\draw[line width=5] (-0.4, -1) -- (-0.4, 0) to[out=90, in=-90] (0.4, 1);
\draw[very thick, blue, dashed] (-1, 0) -- (1, 0);
\end{tikzpicture}
}}
\]
\item\label{item:typeVIboundaryisotopy} An isotopy pulling $L$ across a puncture. 
\[
\vcenter{\hbox{
\begin{tikzpicture}[scale=0.8, yscale=-1]
\filldraw[lightgray] (0, 0) circle (2);
\draw[line width=5] (0.517638, -1.93185) to[out=105, in=0] (0, -1) to[out=180, in=75] (-0.517638, -1.93185);
\draw[blue, very thick, dashed] (0, 0) -- (1*2, 0);
\draw[blue, very thick, dashed] (0, 0) -- (-1*2, 0);
\draw[blue, very thick, dashed] (0, 0) -- (1/2*2, {sqrt(3)/2*2});
\draw[blue, very thick, dashed] (0, 0) -- (-1/2*2, {sqrt(3)/2*2});
\draw[blue, very thick, dashed] (0, 0) -- (1/2*2, {-sqrt(3)/2*2});
\draw[blue, very thick, dashed] (0, 0) -- (-1/2*2, {-sqrt(3)/2*2});
\filldraw[fill=white] (0, 0) circle (0.1);
\end{tikzpicture}
}}
\;\leftrightarrow\;
\vcenter{\hbox{
\begin{tikzpicture}[scale=0.8, yscale=-1]
\filldraw[lightgray] (0, 0) circle (2);
\draw[line width=5] (0.517638, -1.93185) to[out=105, in=-150] ({1/2}, {-sqrt(3)/2});
\draw[line width=5] ({-1/2}, {-sqrt(3)/2}) to[out=-30, in=75] (-0.517638, -1.93185);
\draw[line width=5] ({1/2}, {-sqrt(3)/2}) arc (-60:240:1);
\draw[blue, very thick, dashed] (0, 0) -- (1*2, 0);
\draw[blue, very thick, dashed] (0, 0) -- (-1*2, 0);
\draw[blue, very thick, dashed] (0, 0) -- (1/2*2, {sqrt(3)/2*2});
\draw[blue, very thick, dashed] (0, 0) -- (-1/2*2, {sqrt(3)/2*2});
\draw[blue, very thick, dashed] (0, 0) -- (1/2*2, {-sqrt(3)/2*2});
\draw[blue, very thick, dashed] (0, 0) -- (-1/2*2, {-sqrt(3)/2*2});
\filldraw[fill=white] (0, 0) circle (0.1);
\end{tikzpicture}
}}
\]
\end{enumerate}
It is clear that $\alpha_L$ does not change under the isotopies of type \ref{item:typeIboundaryisotopy}. 
In fact, it does not change under isotopies of type \ref{item:typeIIboundaryisotopy}, \ref{item:typeIIIboundaryisotopy}, \ref{item:typeIVboundaryisotopy}, and \ref{item:typeVboundaryisotopy} either, 
as we show below. 
In the cases the orange marking is not visible in the above pictures, we assume that the orange marking is to the right and oriented upward; the proof is analogous in other cases. 
\begin{itemize}
\item For isotopies of type \ref{item:typeIIboundaryisotopy}, 
\[
\alpha_{\mathrm{RHS}} =
\sum_{\mu \in \{\pm\}}(-A^2)^{\frac{\mu}{2}}\vcenter{\hbox{
\begin{tikzpicture}
\filldraw[lightgray] (-1, -1) -- (-1, 1) -- (1, 1) -- (1, -1) -- cycle;
\begin{scope}[shift={(0, 0.7)}, yscale=-0.3]
    \fill[bottom color=black, top color=white] (-0.086, 1) to[out=-90, in=100] (0, 0) to[out=80, in=-90] (0.086, 1)--cycle;
    \fill[top color=black, bottom color=white] (-0.086, -1) to[out=90, in=-100] (0, 0) to[out=-80, in=90] (0.086, -1)--cycle;
    \draw[line width=1] (-0.072, -1) to[out=90, in=-90] (0.072, 1);
\end{scope}
\begin{scope}[shift={(0, -0.7)}, yscale=0.3]
    \fill[top color=black, bottom color=white] (-0.086, 1) to[out=-90, in=100] (0, 0) to[out=80, in=-90] (0.086, 1)--cycle;
    \fill[bottom color=black, top color=white] (-0.086, -1) to[out=90, in=-100] (0, 0) to[out=-80, in=90] (0.086, -1)--cycle;
    \draw[line width=1] (-0.072, -1) to[out=90, in=-90] (0.072, 1);
\end{scope}
\draw[line width=5] (0, 0.4) to[out=-90, in=180] (0.8, 0.2);
\draw[line width=5] (0, -0.4) to[out=90, in=180] (0.8, -0.2);
\draw[very thick, blue, dashed] (-1, 0) -- (1, 0);
\draw[very thick, orange] (0.8, -1) -- (0.8, 1);
\draw[very thick, orange, ->] (0.8, -1) -- (0.8, 0);
\node[anchor=west] at (0.8, 0.2){$\mu$};
\node[anchor=west] at (0.8, -0.2){$-\mu$};
\end{tikzpicture}
}}
=
\sum_{\mu \in \{\pm\}}(-A^2)^{\frac{\mu}{2}}\vcenter{\hbox{
\begin{tikzpicture}
\filldraw[lightgray] (-1, -1) -- (-1, 1) -- (1, 1) -- (1, -1) -- cycle;
\draw[line width=5] (0, 1) -- (0, 0.4) to[out=-90, in=180] (0.8, 0.2);
\draw[line width=5] (0, -1) -- (0, -0.4) to[out=90, in=180] (0.8, -0.2);
\draw[very thick, blue, dashed] (-1, 0) -- (1, 0);
\draw[very thick, orange] (0.8, -1) -- (0.8, 1);
\draw[very thick, orange, ->] (0.8, -1) -- (0.8, 0);
\node[anchor=west] at (0.8, 0.2){$\mu$};
\node[anchor=west] at (0.8, -0.2){$-\mu$};
\end{tikzpicture}
}}
= \alpha_{\mathrm{LHS}}.
\]
\item For isotopies of type \ref{item:typeIIIboundaryisotopy}, 
\begin{align*}
\alpha_{\mathrm{RHS}}
&=
\sum_{\mu,\nu\in \{\pm\}}
(-A^2)^{\frac{\mu+\nu}{2}}
\vcenter{\hbox{
\begin{tikzpicture}
\filldraw[lightgray] (-1, -1) -- (-1, 1) -- (1, 1) -- (1, -1) -- cycle;
\draw[line width=5] (-0.4, -1) to[out=90, in=180] (0.8, -0.2);
\draw[line width=5] (0.4, -1) to[out=90, in=180] (0.8, -0.5);
\draw[line width=5] (-0.4, 0.6) to[out=90, in=180] (0, 0.9) to[out=-0, in=90] (0.4, 0.6);
\draw[line width=5] (-0.4, 0.6) to[out=-90, in=180] (0.8, 0.2);
\draw[line width=5] (0.4, 0.6) to[out=-90, in=180] (0.8, 0.5);
\draw[very thick, blue, dashed] (-1, 0) -- (1, 0);
\draw[very thick, orange] (0.8, -1) -- (0.8, 1);
\draw[very thick, orange, ->] (0.8, -1) -- (0.8, 0);
\node[anchor=west] at (0.8, 0.2){$\nu$};
\node[anchor=west] at (0.8, -0.2){$-\nu$};
\node[anchor=west] at (0.8, 0.5){$\mu$};
\node[anchor=west] at (0.8, -0.5){$-\mu$};
\end{tikzpicture}
}}
=
\sum_{\mu\in \{\pm\}}
(-A^2)^{\frac{\mu}{2}}
\vcenter{\hbox{
\begin{tikzpicture}
\filldraw[lightgray] (-1, -1) -- (-1, 1) -- (1, 1) -- (1, -1) -- cycle;
\draw[line width=5] (-0.4, -1) to[out=90, in=180] (0.8, -0.2);
\draw[line width=5] (0.4, -1) to[out=90, in=180] (0.8, -0.5);
\draw[very thick, blue, dashed] (-1, 0) -- (1, 0);
\draw[very thick, orange] (0.8, -1) -- (0.8, 1);
\draw[very thick, orange, ->] (0.8, -1) -- (0.8, 0);
\node[anchor=west] at (0.8, -0.2){$\mu$};
\node[anchor=west] at (0.8, -0.5){$-\mu$};
\end{tikzpicture}
}}\\
&=
\vcenter{\hbox{
\begin{tikzpicture}
\filldraw[lightgray] (-1, -1) -- (-1, 1) -- (1, 1) -- (1, -1) -- cycle;
\draw[line width=5] (-0.4, -1) -- (-0.4, -0.8) to[out=90, in=180] (0, -0.4) to[out=-0, in=90] (0.4, -0.8) -- (0.4, -1);
\draw[very thick, blue, dashed] (-1, 0) -- (1, 0);
\draw[very thick, orange] (0.8, -1) -- (0.8, 1);
\draw[very thick, orange, ->] (0.8, -1) -- (0.8, 0);
\end{tikzpicture}
}} 
= \alpha_{\mathrm{LHS}}.
\end{align*}
\item For isotopies of type \ref{item:typeIVboundaryisotopy}, 
\[
\alpha_{\mathrm{RHS}}
=
\sum_{\nu\in \{\pm\}}
(-A^2)^{\frac{\nu}{2}}
\vcenter{\hbox{
\begin{tikzpicture}
\filldraw[lightgray] (-1, -1) -- (-1, 1) -- (1, 1) -- (1, -1) -- cycle;
\draw[line width=5] (-0.4, -1) to[out=90, in=180] (0, -0.2);
\draw[line width=5] (0, 0.6) to[out=180, in=90] (-0.4, 0.4) to[out=-90, in=180] (0, 0.2);
\draw[very thick, blue, dashed] (-1, 0) -- (1, 0);
\draw[very thick, orange] (0, -1) -- (0, 1);
\draw[very thick, orange, ->] (0, -1) -- (0, 0);
\node[anchor=west] at (0, 0.6){$\mu$};
\node[anchor=west] at (0, 0.2){$\nu$};
\node[anchor=west] at (0, -0.2){$-\nu$};
\end{tikzpicture}
}}
=
\vcenter{\hbox{
\begin{tikzpicture}
\filldraw[lightgray] (-1, -1) -- (-1, 1) -- (1, 1) -- (1, -1) -- cycle;
\draw[line width=5] (-0.4, -1) to[out=90, in=180] (0, -0.2);
\draw[very thick, blue, dashed] (-1, 0) -- (1, 0);
\draw[very thick, orange] (0, -1) -- (0, 1);
\draw[very thick, orange, ->] (0, -1) -- (0, 0);
\node[anchor=west] at (0, -0.2){$\mu$};
\end{tikzpicture}
}}
= \alpha_{\mathrm{LHS}}.
\]
\item For isotopies of type \ref{item:typeVboundaryisotopy}, 
\begin{align*}
&\alpha_{\mathrm{LHS}} - \alpha_{\mathrm{RHS}}\\
&=
\sum_{\mu, \nu \in \{\pm\}}
(-A^2)^{\frac{\mu + \nu}{2}}
\vcenter{\hbox{
\begin{tikzpicture}
\filldraw[lightgray] (-1, -1) -- (-1, 1) -- (1, 1) -- (1, -1) -- cycle;
\draw[line width=5] (-0.4, 1) to[out=-90, in=180] (0.8, 0.2);
\draw[line width=5] (0.4, 1) to[out=-90, in=180] (0.8, 0.6);
\draw[line width=5] (0.4, -1) to[out=90, in=-90] (0, -0.6) to[out=90, in=180] (0.8, -0.2);
\draw[lightgray, line width=9] (-0.4, -1) to[out=90, in=180] (0.4, -0.6) -- (0.8, -0.6);
\draw[line width=5] (-0.4, -1) to[out=90, in=180] (0.8, -0.6);
\draw[very thick, blue, dashed] (-1, 0) -- (1, 0);
\draw[very thick, orange] (0.8, -1) -- (0.8, 1);
\draw[very thick, orange, ->] (0.8, -1) -- (0.8, 0);
\node[anchor=west] at (0.8, 0.6){$\mu$};
\node[anchor=west] at (0.8, 0.2){$\nu$};
\node[anchor=west] at (0.8, -0.2){$-\nu$};
\node[anchor=west] at (0.8, -0.6){$-\mu$};
\end{tikzpicture}
}}
- 
\sum_{\mu, \nu \in \{\pm\}}
(-A^2)^{\frac{\mu + \nu}{2}}
\vcenter{\hbox{
\begin{tikzpicture}
\filldraw[lightgray] (-1, -1) -- (-1, 1) -- (1, 1) -- (1, -1) -- cycle;
\draw[line width=5] (-0.4, -1) to[out=90, in=180] (0.8, -0.2);
\draw[line width=5] (0.4, -1) to[out=90, in=180] (0.8, -0.6);
\draw[line width=5] (-0.4, 1) to[out=-90, in=180] (0.8, 0.6);
\draw[lightgray, line width=9] (0.4, 1) to[out=-90, in=90] (0, 0.6) to[out=-90, in=180] (0.8, 0.2);
\draw[line width=5] (0.4, 1) to[out=-90, in=90] (0, 0.6) to[out=-90, in=180] (0.8, 0.2);
\draw[very thick, blue, dashed] (-1, 0) -- (1, 0);
\draw[very thick, orange] (0.8, -1) -- (0.8, 1);
\draw[very thick, orange, ->] (0.8, -1) -- (0.8, 0);
\node[anchor=west] at (0.8, 0.6){$\mu$};
\node[anchor=west] at (0.8, 0.2){$\nu$};
\node[anchor=west] at (0.8, -0.2){$-\nu$};
\node[anchor=west] at (0.8, -0.6){$-\mu$};
\end{tikzpicture}
}} \\
&= 
\sum_{\mu, \nu \in \{\pm\}}
(-A^2)^{\frac{\mu + \nu}{2}}
\qty(
A
\vcenter{\hbox{
\begin{tikzpicture}
\filldraw[lightgray] (-1, -1) -- (-1, 1) -- (1, 1) -- (1, -1) -- cycle;
\draw[line width=5] (-0.4, 1) to[out=-90, in=180] (0.8, 0.2);
\draw[line width=5] (0.4, 1) to[out=-90, in=180] (0.8, 0.6);
\draw[line width=5] (0.4, -1) to[out=90, in=180] (0.8, -0.6);
\draw[line width=5] (-0.4, -1) to[out=90, in=180] (0.8, -0.2);
\draw[very thick, blue, dashed] (-1, 0) -- (1, 0);
\draw[very thick, orange] (0.8, -1) -- (0.8, 1);
\draw[very thick, orange, ->] (0.8, -1) -- (0.8, 0);
\node[anchor=west] at (0.8, 0.6){$\mu$};
\node[anchor=west] at (0.8, 0.2){$\nu$};
\node[anchor=west] at (0.8, -0.2){$-\nu$};
\node[anchor=west] at (0.8, -0.6){$-\mu$};
\end{tikzpicture}
}}
+
A^{-1}
\vcenter{\hbox{
\begin{tikzpicture}
\filldraw[lightgray] (-1, -1) -- (-1, 1) -- (1, 1) -- (1, -1) -- cycle;
\draw[line width=5] (-0.4, 1) to[out=-90, in=180] (0.8, 0.2);
\draw[line width=5] (0.4, 1) to[out=-90, in=180] (0.8, 0.6);
\draw[line width=5] (-0.4, -1) to[out=90, in=180] (0, -0.6) to[out=0, in=90] (0.4, -1);
\draw[line width=5] (0.8, -0.2) to[out=180, in=90] (0.6, -0.4) to[out=-90, in=180] (0.8, -0.6);
\draw[very thick, blue, dashed] (-1, 0) -- (1, 0);
\draw[very thick, orange] (0.8, -1) -- (0.8, 1);
\draw[very thick, orange, ->] (0.8, -1) -- (0.8, 0);
\node[anchor=west] at (0.8, 0.6){$\mu$};
\node[anchor=west] at (0.8, 0.2){$\nu$};
\node[anchor=west] at (0.8, -0.2){$-\nu$};
\node[anchor=west] at (0.8, -0.6){$-\mu$};
\end{tikzpicture}
}}
)
\\
&\quad
-
\sum_{\mu, \nu \in \{\pm\}}
(-A^2)^{\frac{\mu + \nu}{2}}
\qty(
A
\vcenter{\hbox{
\begin{tikzpicture}
\filldraw[lightgray] (-1, -1) -- (-1, 1) -- (1, 1) -- (1, -1) -- cycle;
\draw[line width=5] (-0.4, 1) to[out=-90, in=180] (0.8, 0.2);
\draw[line width=5] (0.4, 1) to[out=-90, in=180] (0.8, 0.6);
\draw[line width=5] (0.4, -1) to[out=90, in=180] (0.8, -0.6);
\draw[line width=5] (-0.4, -1) to[out=90, in=180] (0.8, -0.2);
\draw[very thick, blue, dashed] (-1, 0) -- (1, 0);
\draw[very thick, orange] (0.8, -1) -- (0.8, 1);
\draw[very thick, orange, ->] (0.8, -1) -- (0.8, 0);
\node[anchor=west] at (0.8, 0.6){$\mu$};
\node[anchor=west] at (0.8, 0.2){$\nu$};
\node[anchor=west] at (0.8, -0.2){$-\nu$};
\node[anchor=west] at (0.8, -0.6){$-\mu$};
\end{tikzpicture}
}}
+
A^{-1}
\vcenter{\hbox{
\begin{tikzpicture}
\filldraw[lightgray] (-1, -1) -- (-1, 1) -- (1, 1) -- (1, -1) -- cycle;
\draw[line width=5] (-0.4, -1) to[out=90, in=180] (0.8, -0.2);
\draw[line width=5] (0.4, -1) to[out=90, in=180] (0.8, -0.6);
\draw[line width=5] (-0.4, 1) to[out=-90, in=180] (0, 0.6) to[out=0, in=-90] (0.4, 1);
\draw[line width=5] (0.8, 0.2) to[out=180, in=-90] (0.6, 0.4) to[out=90, in=180] (0.8, 0.6);
\draw[very thick, blue, dashed] (-1, 0) -- (1, 0);
\draw[very thick, orange] (0.8, -1) -- (0.8, 1);
\draw[very thick, orange, ->] (0.8, -1) -- (0.8, 0);
\node[anchor=west] at (0.8, 0.6){$\mu$};
\node[anchor=west] at (0.8, 0.2){$\nu$};
\node[anchor=west] at (0.8, -0.2){$-\nu$};
\node[anchor=west] at (0.8, -0.6){$-\mu$};
\end{tikzpicture}
}}
)\\
&= 
A^{-1}
\sum_{\mu\in \{\pm\}}
(-A^2)^{\frac{\mu}{2}}
\vcenter{\hbox{
\begin{tikzpicture}
\filldraw[lightgray] (-1, -1) -- (-1, 1) -- (1, 1) -- (1, -1) -- cycle;
\draw[line width=5] (-0.4, 1) to[out=-90, in=180] (0.8, 0.2);
\draw[line width=5] (0.4, 1) to[out=-90, in=180] (0.8, 0.6);
\draw[line width=5] (-0.4, -1) to[out=90, in=180] (0, -0.6) to[out=0, in=90] (0.4, -1);
\draw[very thick, blue, dashed] (-1, 0) -- (1, 0);
\draw[very thick, orange] (0.8, -1) -- (0.8, 1);
\draw[very thick, orange, ->] (0.8, -1) -- (0.8, 0);
\node[anchor=west] at (0.8, 0.6){$\mu$};
\node[anchor=west] at (0.8, 0.2){$-\mu$};
\end{tikzpicture}
}}
-
A^{-1}
\sum_{\mu\in \{\pm\}}
(-A^2)^{\frac{\mu}{2}}
\vcenter{\hbox{
\begin{tikzpicture}
\filldraw[lightgray] (-1, -1) -- (-1, 1) -- (1, 1) -- (1, -1) -- cycle;
\draw[line width=5] (-0.4, -1) to[out=90, in=180] (0.8, -0.2);
\draw[line width=5] (0.4, -1) to[out=90, in=180] (0.8, -0.6);
\draw[line width=5] (-0.4, 1) to[out=-90, in=180] (0, 0.6) to[out=0, in=-90] (0.4, 1);
\draw[very thick, blue, dashed] (-1, 0) -- (1, 0);
\draw[very thick, orange] (0.8, -1) -- (0.8, 1);
\draw[very thick, orange, ->] (0.8, -1) -- (0.8, 0);
\node[anchor=west] at (0.8, -0.2){$\mu$};
\node[anchor=west] at (0.8, -0.6){$-\mu$};
\end{tikzpicture}
}}
\\
&= 
A^{-1}
\vcenter{\hbox{
\begin{tikzpicture}
\filldraw[lightgray] (-1, -1) -- (-1, 1) -- (1, 1) -- (1, -1) -- cycle;
\draw[line width=5] (-0.4, 1) to[out=-90, in=180] (0, 0.6) to[out=0, in=-90] (0.4, 1);
\draw[line width=5] (-0.4, -1) to[out=90, in=180] (0, -0.6) to[out=0, in=90] (0.4, -1);
\draw[very thick, blue, dashed] (-1, 0) -- (1, 0);
\draw[very thick, orange] (0.8, -1) -- (0.8, 1);
\draw[very thick, orange, ->] (0.8, -1) -- (0.8, 0);
\end{tikzpicture}
}}
- A^{-1}
\vcenter{\hbox{
\begin{tikzpicture}
\filldraw[lightgray] (-1, -1) -- (-1, 1) -- (1, 1) -- (1, -1) -- cycle;
\draw[line width=5] (-0.4, 1) to[out=-90, in=180] (0, 0.6) to[out=0, in=-90] (0.4, 1);
\draw[line width=5] (-0.4, -1) to[out=90, in=180] (0, -0.6) to[out=0, in=90] (0.4, -1);
\draw[very thick, blue, dashed] (-1, 0) -- (1, 0);
\draw[very thick, orange] (0.8, -1) -- (0.8, 1);
\draw[very thick, orange, ->] (0.8, -1) -- (0.8, 0);
\end{tikzpicture}
}}
 = 0
.
\end{align*}
\end{itemize}

The only remaining thing to check is how $\alpha_L$ changes under isotopies of type \ref{item:typeVIboundaryisotopy}. 
The $\alpha_L$ for the RHS of the isotopy of type \ref{item:typeVIboundaryisotopy} would look like
\begin{gather*}
\alpha_{\mathrm{RHS}}
\;\;=\;\; 
\sum_{\mu, \nu \in \{\pm\}} (-A^2)^{\frac{\mu+\nu}{2}} \sum_{\epsilon_i \in  \{\pm\}}(-A^2)^{\frac{\sum_{i}\epsilon_i}{2}}
\vcenter{\hbox{
\begin{tikzpicture}[scale=0.7]
\filldraw[lightgray] (3, 0) -- (3/2, {3/2*sqrt(3)}) -- (-3/2, {3/2*sqrt(3)}) -- (-3, 0) -- (-3/2, {-3/2*sqrt(3)}) -- (3/2, {-3/2*sqrt(3)}) -- cycle;
\filldraw[lightgray] (-4, -0.6) rectangle (4, 0.6);
\begin{scope}[rotate=60]
    \filldraw[lightgray] (-4, -0.6) rectangle (4, 0.6);
\end{scope}
\begin{scope}[rotate=120]
    \filldraw[lightgray] (-4, -0.6) rectangle (4, 0.6);
\end{scope}
\draw[line width=3] (0.3, {3/2*sqrt(3)}) to[out=-90, in=-120] (0.93, 2.4) -- ({2-sqrt(3)*0.2}, {2*sqrt(3)+0.2});
\begin{scope}[xscale=-1]
    \draw[line width=3] (0.3, {3/2*sqrt(3)}) to[out=-90, in=-120] (0.93, 2.4) -- ({2-sqrt(3)*0.2}, {2*sqrt(3)+0.2});
\end{scope}
\begin{scope}[rotate=60]
    \draw[line width=3] ({2-sqrt(3)/10}, {2*sqrt(3)+0.1}) .. controls (0, 0.2) and (0, 0.2) .. ({-2+sqrt(3)/10}, {2*sqrt(3)+0.1});
\end{scope}
\begin{scope}[rotate=120]
    \draw[line width=3] ({2-sqrt(3)/10}, {2*sqrt(3)+0.1}) .. controls (0, 0.2) and (0, 0.2) .. ({-2+sqrt(3)/10}, {2*sqrt(3)+0.1});
\end{scope}
\begin{scope}[rotate=180]
    \draw[line width=3] ({2-sqrt(3)/10}, {2*sqrt(3)+0.1}) .. controls (0, 0.2) and (0, 0.2) .. ({-2+sqrt(3)/10}, {2*sqrt(3)+0.1});
\end{scope}
\begin{scope}[rotate=-60]
    \draw[line width=3] ({2-sqrt(3)/10}, {2*sqrt(3)+0.1}) .. controls (0, 0.2) and (0, 0.2) .. ({-2+sqrt(3)/10}, {2*sqrt(3)+0.1});
\end{scope}
\begin{scope}[rotate=-120]
    \draw[line width=3] ({2-sqrt(3)/10}, {2*sqrt(3)+0.1}) .. controls (0, 0.2) and (0, 0.2) .. ({-2+sqrt(3)/10}, {2*sqrt(3)+0.1});
\end{scope}
\draw[blue, very thick, dashed] (0, 0) -- (1*4, 0);
\draw[blue, very thick, dashed] (0, 0) -- (-1*4, 0);
\draw[blue, very thick, dashed] (0, 0) -- (1/2*4, {sqrt(3)/2*4});
\draw[blue, very thick, dashed] (0, 0) -- (-1/2*4, {sqrt(3)/2*4});
\draw[blue, very thick, dashed] (0, 0) -- (1/2*4, {-sqrt(3)/2*4});
\draw[blue, very thick, dashed] (0, 0) -- (-1/2*4, {-sqrt(3)/2*4});
\draw[orange, ultra thick] (4, {-4/3*sqrt(3)}) -- (4, {4/3*sqrt(3)}) -- (0, {8/3*sqrt(3)}) -- (-4, {4/3*sqrt(3)}) -- (-4, {-4/3*sqrt(3)}) -- (0, {-8/3*sqrt(3)}) -- cycle;
\draw[orange, ultra thick, ->] (4, {4/3*sqrt(3)}) -- (4, 0);
\draw[orange, ultra thick, ->] (4, {4/3*sqrt(3)}) -- (1/2*4, {sqrt(3)/2*4});
\draw[orange, ultra thick, ->] (-4, {4/3*sqrt(3)}) -- (-1/2*4, {sqrt(3)/2*4});
\draw[orange, ultra thick, ->] (-4, {4/3*sqrt(3)}) -- (-4, 0);
\draw[orange, ultra thick, ->] (0, {-8/3*sqrt(3)}) -- (-1/2*4, {-sqrt(3)/2*4});
\draw[orange, ultra thick, ->] (0, {-8/3*sqrt(3)}) -- (1/2*4, {-sqrt(3)/2*4});
\filldraw[fill=white] (0, 0) circle (0.2);
\filldraw[orange] (4, {-4/3*sqrt(3)}) circle (0.1);
\filldraw[orange] (4, {4/3*sqrt(3)}) circle (0.1);
\filldraw[orange] (0, {8/3*sqrt(3)}) circle (0.1);
\filldraw[orange] (-4, {4/3*sqrt(3)}) circle (0.1);
\filldraw[orange] (-4, {-4/3*sqrt(3)}) circle (0.1);
\filldraw[orange] (0, {-8/3*sqrt(3)}) circle (0.1);
\node at (4.5, 0.3){$-\epsilon_4$};
\node at (4.5, -0.3){$\epsilon_4$};
\begin{scope}[rotate=-60]
    \node at (4.5, 0.3){$\epsilon_3$};
    \node at (4.5, -0.3){$-\epsilon_3$};
\end{scope}
\begin{scope}[rotate=-120]
    \node at (4.5, 0.3){$-\epsilon_2$};
    \node at (4.5, -0.3){$\epsilon_2$};
\end{scope}
\begin{scope}[rotate=-180]
    \node at (4.5, 0.3){$\epsilon_1$};
    \node at (4.6, -0.3){$-\epsilon_1$};
\end{scope}
\begin{scope}[rotate=60]
    \node at (4.5, 0.3){$\nu$};
    \node at (4.5, -0.3){$-\nu$};
\end{scope}
\begin{scope}[rotate=120]
    \node at (4.5, 0.4){$-\mu$};
    \node at (4.5, -0.4){$\mu$};
\end{scope}
\end{tikzpicture}
}}.
\end{gather*}
Here, we are breaking up the strand that went over the puncture last, using the fact we have proved earlier that $\alpha_L$ is independent of the choice of ordering of breaking up the strands in the algorithm. 
Moreover, we haven't filled in the background color near the vertices to indicate that there may be some other parts of the tangle in those regions. 

Once we mod out the skein algebra
$\bigotimes_{v \in V(\Gamma)^{+}}\SkAlg(D_{\deg v})\otimes\bigotimes_{w \in V(\Gamma)^{-}}\SkAlg(D_{\deg w})^{\mathrm{op}}$
by the left ideal $I$ generated by relations \eqref{eq:GeneratorsOfAnnhilator}, 
the above picture is equivalent to
\begin{align*}
&\sum_{\mu, \nu \in \{\pm\}} (-A^2)^{\frac{\mu+\nu}{2}}
\vcenter{\hbox{
\begin{tikzpicture}[scale=0.7]
\filldraw[lightgray] (3, 0) -- (3/2, {3/2*sqrt(3)}) -- (-3/2, {3/2*sqrt(3)}) -- (-3, 0) -- (-3/2, {-3/2*sqrt(3)}) -- (3/2, {-3/2*sqrt(3)}) -- cycle;
\filldraw[lightgray] (-4, -0.6) rectangle (4, 0.6);
\begin{scope}[rotate=60]
    \filldraw[lightgray] (-4, -0.6) rectangle (4, 0.6);
\end{scope}
\begin{scope}[rotate=120]
    \filldraw[lightgray] (-4, -0.6) rectangle (4, 0.6);
\end{scope}
\draw[line width=3] (0.3, {3/2*sqrt(3)}) to[out=-90, in=-120] (0.93, 2.4) -- ({2-sqrt(3)*0.2}, {2*sqrt(3)+0.2});
\begin{scope}[xscale=-1]
    \draw[line width=3] (0.3, {3/2*sqrt(3)}) to[out=-90, in=-120] (0.93, 2.4) -- ({2-sqrt(3)*0.2}, {2*sqrt(3)+0.2});
\end{scope}
\begin{scope}[rotate=0]
    \draw[line width=3] ({2-sqrt(3)/10}, {2*sqrt(3)+0.1}) .. controls (0, 0.2) and (0, 0.2) .. ({-2+sqrt(3)/10}, {2*sqrt(3)+0.1});
\end{scope}
\draw[blue, very thick, dashed] (0, 0) -- (1*4, 0);
\draw[blue, very thick, dashed] (0, 0) -- (-1*4, 0);
\draw[blue, very thick, dashed] (0, 0) -- (1/2*4, {sqrt(3)/2*4});
\draw[blue, very thick, dashed] (0, 0) -- (-1/2*4, {sqrt(3)/2*4});
\draw[blue, very thick, dashed] (0, 0) -- (1/2*4, {-sqrt(3)/2*4});
\draw[blue, very thick, dashed] (0, 0) -- (-1/2*4, {-sqrt(3)/2*4});
\draw[orange, ultra thick] (4, {-4/3*sqrt(3)}) -- (4, {4/3*sqrt(3)}) -- (0, {8/3*sqrt(3)}) -- (-4, {4/3*sqrt(3)}) -- (-4, {-4/3*sqrt(3)}) -- (0, {-8/3*sqrt(3)}) -- cycle;
\draw[orange, ultra thick, ->] (4, {4/3*sqrt(3)}) -- (4, 0);
\draw[orange, ultra thick, ->] (4, {4/3*sqrt(3)}) -- (1/2*4, {sqrt(3)/2*4});
\draw[orange, ultra thick, ->] (-4, {4/3*sqrt(3)}) -- (-1/2*4, {sqrt(3)/2*4});
\draw[orange, ultra thick, ->] (-4, {4/3*sqrt(3)}) -- (-4, 0);
\draw[orange, ultra thick, ->] (0, {-8/3*sqrt(3)}) -- (-1/2*4, {-sqrt(3)/2*4});
\draw[orange, ultra thick, ->] (0, {-8/3*sqrt(3)}) -- (1/2*4, {-sqrt(3)/2*4});
\filldraw[fill=white] (0, 0) circle (0.2);
\filldraw[orange] (4, {-4/3*sqrt(3)}) circle (0.1);
\filldraw[orange] (4, {4/3*sqrt(3)}) circle (0.1);
\filldraw[orange] (0, {8/3*sqrt(3)}) circle (0.1);
\filldraw[orange] (-4, {4/3*sqrt(3)}) circle (0.1);
\filldraw[orange] (-4, {-4/3*sqrt(3)}) circle (0.1);
\filldraw[orange] (0, {-8/3*sqrt(3)}) circle (0.1);
\begin{scope}[rotate=60]
    \node at (4.5, 0.6){$\nu$};
    \node at (4.5, 0){$-\nu$};
\end{scope}
\begin{scope}[rotate=120]
    \node at (4.5, 0.1){$-\mu$};
    \node at (4.5, -0.6){$\mu$};
\end{scope}
\end{tikzpicture}
}}
\;\;=\;\;
\vcenter{\hbox{
\begin{tikzpicture}[scale=0.7]
\filldraw[lightgray] (3, 0) -- (3/2, {3/2*sqrt(3)}) -- (-3/2, {3/2*sqrt(3)}) -- (-3, 0) -- (-3/2, {-3/2*sqrt(3)}) -- (3/2, {-3/2*sqrt(3)}) -- cycle;
\filldraw[lightgray] (-4, -0.6) rectangle (4, 0.6);
\begin{scope}[rotate=60]
    \filldraw[lightgray] (-4, -0.6) rectangle (4, 0.6);
\end{scope}
\begin{scope}[rotate=120]
    \filldraw[lightgray] (-4, -0.6) rectangle (4, 0.6);
\end{scope}
\draw[line width=3] (0.3, {3/2*sqrt(3)}) to[out=-90, in=0] (0, 2.2) to[out=180, in=-90](-0.3, {3/2*sqrt(3)});
\draw[blue, very thick, dashed] (0, 0) -- (1*4, 0);
\draw[blue, very thick, dashed] (0, 0) -- (-1*4, 0);
\draw[blue, very thick, dashed] (0, 0) -- (1/2*4, {sqrt(3)/2*4});
\draw[blue, very thick, dashed] (0, 0) -- (-1/2*4, {sqrt(3)/2*4});
\draw[blue, very thick, dashed] (0, 0) -- (1/2*4, {-sqrt(3)/2*4});
\draw[blue, very thick, dashed] (0, 0) -- (-1/2*4, {-sqrt(3)/2*4});
\draw[orange, ultra thick] (4, {-4/3*sqrt(3)}) -- (4, {4/3*sqrt(3)}) -- (0, {8/3*sqrt(3)}) -- (-4, {4/3*sqrt(3)}) -- (-4, {-4/3*sqrt(3)}) -- (0, {-8/3*sqrt(3)}) -- cycle;
\draw[orange, ultra thick, ->] (4, {4/3*sqrt(3)}) -- (4, 0);
\draw[orange, ultra thick, ->] (4, {4/3*sqrt(3)}) -- (1/2*4, {sqrt(3)/2*4});
\draw[orange, ultra thick, ->] (-4, {4/3*sqrt(3)}) -- (-1/2*4, {sqrt(3)/2*4});
\draw[orange, ultra thick, ->] (-4, {4/3*sqrt(3)}) -- (-4, 0);
\draw[orange, ultra thick, ->] (0, {-8/3*sqrt(3)}) -- (-1/2*4, {-sqrt(3)/2*4});
\draw[orange, ultra thick, ->] (0, {-8/3*sqrt(3)}) -- (1/2*4, {-sqrt(3)/2*4});
\filldraw[fill=white] (0, 0) circle (0.2);
\filldraw[orange] (4, {-4/3*sqrt(3)}) circle (0.1);
\filldraw[orange] (4, {4/3*sqrt(3)}) circle (0.1);
\filldraw[orange] (0, {8/3*sqrt(3)}) circle (0.1);
\filldraw[orange] (-4, {4/3*sqrt(3)}) circle (0.1);
\filldraw[orange] (-4, {-4/3*sqrt(3)}) circle (0.1);
\filldraw[orange] (0, {-8/3*sqrt(3)}) circle (0.1);
\end{tikzpicture}
}}\\
&\;\;=\;\;
\alpha_{\mathrm{LHS}},
\end{align*}
as desired. 

To this end, define the map $g$ in \eqref{eq:gmap} to be 
\[
g: [L] \mapsto [\alpha_L],
\]
which we have just shown to be well-defined. 
By construction, $g\circ f$ is the identity map. 
This completes our proof. 
\end{proof}

\begin{rmk}
Note that the only property of a 3-ball we used in the proof of Theorem \ref{thm:SkeinModuleOf3Ball} is that, once we puncture the center of the 3-ball, it is topologically $S^2 \times I$ so that we can project any tangle to the combinatorially foliated boundary. 
Therefore, the theorem extends straightforwardly to any $(\Sigma \times I, \Gamma \times \{0\})$, where $(\Sigma, \Gamma)$ is a combinatorially foliated surface with the associated marking. 
We stated the theorem only for 3-balls, as that is all we need for the purpose of constructing the 3d quantum trace map. 
\end{rmk}

\subsection{Reduced stated skein modules of 3-balls}
From the proof of Theorem \ref{thm:SkeinModuleOf3Ball}, we immediately get the following structure theorem for the reduced stated skein modules of 3-balls: 
\begin{cor}\label{cor:RedSkeinModuleOf3Balls}
In the setup of Lemma \ref{lem:Sk(B)isCyclic}, the reduced skein module of $(B,\Gamma)$ has the following presentation: 
\[
\overline{\Sk}(B, \Gamma) \cong 
\frac{\bigotimes_{v \in V(\Gamma)^{+}}\overline{\SkAlg}(D_{\deg v})\otimes\bigotimes_{w \in V(\Gamma)^{-}}\overline{\SkAlg}(D_{\deg w})^{\mathrm{op}}}{\mathrm{Ann}([\emptyset])},
\]
where $\mathrm{Ann}([\emptyset])$ is the left ideal generated by the following relations, one for each puncture of the combinatorial foliation of $\partial B$:
\[
\vcenter{\hbox{
\begin{tikzpicture}
\draw[line width=5] (-3/16*2, {sqrt(3)*7/16*2}) to[out=-60, in=-120] (3/16*2, {sqrt(3)*7/16*2});
\node[anchor=south] at (-3/16*2, {sqrt(3)*7/16*2}){$\mu$};
\node[anchor=south] at (3/16*2, {sqrt(3)*7/16*2}){$\nu$};
\draw[orange, very thick] (-3/4*2, {sqrt(3)/4*2}) -- (0, {sqrt(3)/2*2});
\draw[orange, very thick, ->] (-3/4*2, {sqrt(3)/4*2}) -- (-3/8*2, {sqrt(3)*3/8*2});
\draw[orange, very thick] (3/4*2, {sqrt(3)/4*2}) -- (0, {sqrt(3)/2*2});
\draw[orange, very thick, ->] (3/4*2, {sqrt(3)/4*2}) -- (3/8*2, {sqrt(3)*3/8*2});
\draw[orange, very thick] (3/4*2, {sqrt(3)/4*2}) -- (3/4*2, {-sqrt(3)/4*2});
\draw[orange, very thick, ->] (3/4*2, {sqrt(3)/4*2}) -- (3/4*2, 0);
\draw[orange, very thick] (0, {-sqrt(3)/2*2}) -- (3/4*2, {-sqrt(3)/4*2});
\draw[orange, very thick, ->] (0, {-sqrt(3)/2*2}) -- (3/8*2, {-sqrt(3)*3/8*2});
\draw[orange, very thick] (0, {-sqrt(3)/2*2}) -- (-3/4*2, {-sqrt(3)/4*2});
\draw[orange, very thick, ->] (0, {-sqrt(3)/2*2}) -- (-3/8*2, {-sqrt(3)*3/8*2});
\draw[orange, very thick] (-3/4*2, {sqrt(3)/4*2}) -- (-3/4*2, {-sqrt(3)/4*2});
\draw[orange, very thick, ->] (-3/4*2, {sqrt(3)/4*2}) -- (-3/4*2, 0);
\draw[blue, dashed] (0, 0) -- (1*2, 0);
\draw[blue, dashed] (0, 0) -- (-1*2, 0);
\draw[blue, dashed] (0, 0) -- (1/2*2, {sqrt(3)/2*2});
\draw[blue, dashed] (0, 0) -- (-1/2*2, {sqrt(3)/2*2});
\draw[blue, dashed] (0, 0) -- (1/2*2, {-sqrt(3)/2*2});
\draw[blue, dashed] (0, 0) -- (-1/2*2, {-sqrt(3)/2*2});
\filldraw[fill=white] (0, 0) circle (0.1);
\filldraw[orange] (-3/4*2, {sqrt(3)/4*2}) circle (0.07);
\filldraw[orange] (3/4*2, {sqrt(3)/4*2}) circle (0.07);
\filldraw[orange] (0, {sqrt(3)/2*2}) circle (0.07);
\filldraw[orange] (0, {-sqrt(3)/2*2}) circle (0.07);
\filldraw[orange] (3/4*2, {-sqrt(3)/4*2}) circle (0.07);
\filldraw[orange] (-3/4*2, {-sqrt(3)/4*2}) circle (0.07);
\end{tikzpicture}
}}
\;\;=\;\;
\sum_{\epsilon_i \in \{\pm\}}(-A^2)^{\frac{\sum_{i} \epsilon_i}{2}}
\vcenter{\hbox{
\begin{tikzpicture}
\draw[line width=5] (-3*3/16*2, {sqrt(3)*5/16*2}) to[out=-60, in=0] (-3/4*2, {sqrt(3)/8*2});
\node[anchor=south] at (-3*3/16*2, {sqrt(3)*5/16*2}){$\mu$};
\node[anchor=east] at (-3/4*2, {sqrt(3)/8*2}){$-\epsilon_1$};
\draw[line width=5] (3/4*2, {sqrt(3)/8*2}) to[out=180, in=-120] (3*3/16*2, {sqrt(3)*5/16*2});
\node[anchor=south] at (3*3/16*2, {sqrt(3)*5/16*2}){$\nu$};
\node[anchor=west] at (3/4*2, {sqrt(3)/8*2}){$-\epsilon_4$};
\draw[line width=5] (-3*3/16*2, {-sqrt(3)*5/16*2}) to[out=60, in=0] (-3/4*2, {-sqrt(3)/8*2});
\node[anchor=east] at (-3/4*2, {-sqrt(3)/8*2}){$\epsilon_1$};
\node[anchor=north] at (-3*3/16*2, {-sqrt(3)*5/16*2}){$\epsilon_2$};
\draw[line width=5] (3/4*2, {-sqrt(3)/8*2}) to[out=180, in=120] (3*3/16*2, {-sqrt(3)*5/16*2});
\node[anchor=north] at (3*3/16*2, {-sqrt(3)*5/16*2}){$\epsilon_3$};
\node[anchor=west] at (3/4*2, {-sqrt(3)/8*2}){$\epsilon_4$};
\draw[line width=5] (-3/16*2, {-sqrt(3)*7/16*2}) to[out=60, in=120] (3/16*2, {-sqrt(3)*7/16*2});
\node[anchor=north] at (-3/16*2-0.1, {-sqrt(3)*7/16*2}){$-\epsilon_2$};
\node[anchor=north] at (3/16*2, {-sqrt(3)*7/16*2}){$-\epsilon_3$};
\draw[orange, very thick] (-3/4*2, {sqrt(3)/4*2}) -- (0, {sqrt(3)/2*2});
\draw[orange, very thick, ->] (-3/4*2, {sqrt(3)/4*2}) -- (-3/8*2, {sqrt(3)*3/8*2});
\draw[orange, very thick] (3/4*2, {sqrt(3)/4*2}) -- (0, {sqrt(3)/2*2});
\draw[orange, very thick, ->] (3/4*2, {sqrt(3)/4*2}) -- (3/8*2, {sqrt(3)*3/8*2});
\draw[orange, very thick] (3/4*2, {sqrt(3)/4*2}) -- (3/4*2, {-sqrt(3)/4*2});
\draw[orange, very thick, ->] (3/4*2, {sqrt(3)/4*2}) -- (3/4*2, 0);
\draw[orange, very thick] (0, {-sqrt(3)/2*2}) -- (3/4*2, {-sqrt(3)/4*2});
\draw[orange, very thick, ->] (0, {-sqrt(3)/2*2}) -- (3/8*2, {-sqrt(3)*3/8*2});
\draw[orange, very thick] (0, {-sqrt(3)/2*2}) -- (-3/4*2, {-sqrt(3)/4*2});
\draw[orange, very thick, ->] (0, {-sqrt(3)/2*2}) -- (-3/8*2, {-sqrt(3)*3/8*2});
\draw[orange, very thick] (-3/4*2, {sqrt(3)/4*2}) -- (-3/4*2, {-sqrt(3)/4*2});
\draw[orange, very thick, ->] (-3/4*2, {sqrt(3)/4*2}) -- (-3/4*2, 0);
\draw[blue, dashed] (0, 0) -- (1*2, 0);
\draw[blue, dashed] (0, 0) -- (-1*2, 0);
\draw[blue, dashed] (0, 0) -- (1/2*2, {sqrt(3)/2*2});
\draw[blue, dashed] (0, 0) -- (-1/2*2, {sqrt(3)/2*2});
\draw[blue, dashed] (0, 0) -- (1/2*2, {-sqrt(3)/2*2});
\draw[blue, dashed] (0, 0) -- (-1/2*2, {-sqrt(3)/2*2});
\filldraw[fill=white] (0, 0) circle (0.1);
\filldraw[orange] (-3/4*2, {sqrt(3)/4*2}) circle (0.07);
\filldraw[orange] (3/4*2, {sqrt(3)/4*2}) circle (0.07);
\filldraw[orange] (0, {sqrt(3)/2*2}) circle (0.07);
\filldraw[orange] (0, {-sqrt(3)/2*2}) circle (0.07);
\filldraw[orange] (3/4*2, {-sqrt(3)/4*2}) circle (0.07);
\filldraw[orange] (-3/4*2, {-sqrt(3)/4*2}) circle (0.07);
\end{tikzpicture}
}}
\;.
\]
\end{cor}
Note, because bad arcs evaluate to $0$ in reduced skein modules, the above relation can be simplified to the following relations, where the face of $S^2 \setminus \Gamma$ associated to the puncture is a $2n$-gon: 
\begin{enumerate}
\item \[
\vcenter{\hbox{
\begin{tikzpicture}
\draw[line width=5] (-3/16*2, {sqrt(3)*7/16*2}) to[out=-60, in=-120] (3/16*2, {sqrt(3)*7/16*2});
\node[anchor=south] at (-3/16*2, {sqrt(3)*7/16*2}){$\mu$};
\node[anchor=south] at (3/16*2, {sqrt(3)*7/16*2}){$\mu$};
\draw[orange, very thick] (-3/4*2, {sqrt(3)/4*2}) -- (0, {sqrt(3)/2*2});
\draw[orange, very thick, ->] (-3/4*2, {sqrt(3)/4*2}) -- (-3/8*2, {sqrt(3)*3/8*2});
\draw[orange, very thick] (3/4*2, {sqrt(3)/4*2}) -- (0, {sqrt(3)/2*2});
\draw[orange, very thick, ->] (3/4*2, {sqrt(3)/4*2}) -- (3/8*2, {sqrt(3)*3/8*2});
\draw[orange, very thick] (3/4*2, {sqrt(3)/4*2}) -- (3/4*2, {-sqrt(3)/4*2});
\draw[orange, very thick, ->] (3/4*2, {sqrt(3)/4*2}) -- (3/4*2, 0);
\draw[orange, very thick] (0, {-sqrt(3)/2*2}) -- (3/4*2, {-sqrt(3)/4*2});
\draw[orange, very thick, ->] (0, {-sqrt(3)/2*2}) -- (3/8*2, {-sqrt(3)*3/8*2});
\draw[orange, very thick] (0, {-sqrt(3)/2*2}) -- (-3/4*2, {-sqrt(3)/4*2});
\draw[orange, very thick, ->] (0, {-sqrt(3)/2*2}) -- (-3/8*2, {-sqrt(3)*3/8*2});
\draw[orange, very thick] (-3/4*2, {sqrt(3)/4*2}) -- (-3/4*2, {-sqrt(3)/4*2});
\draw[orange, very thick, ->] (-3/4*2, {sqrt(3)/4*2}) -- (-3/4*2, 0);
\draw[blue, dashed] (0, 0) -- (1*2, 0);
\draw[blue, dashed] (0, 0) -- (-1*2, 0);
\draw[blue, dashed] (0, 0) -- (1/2*2, {sqrt(3)/2*2});
\draw[blue, dashed] (0, 0) -- (-1/2*2, {sqrt(3)/2*2});
\draw[blue, dashed] (0, 0) -- (1/2*2, {-sqrt(3)/2*2});
\draw[blue, dashed] (0, 0) -- (-1/2*2, {-sqrt(3)/2*2});
\filldraw[fill=white] (0, 0) circle (0.1);
\filldraw[orange] (-3/4*2, {sqrt(3)/4*2}) circle (0.07);
\filldraw[orange] (3/4*2, {sqrt(3)/4*2}) circle (0.07);
\filldraw[orange] (0, {sqrt(3)/2*2}) circle (0.07);
\filldraw[orange] (0, {-sqrt(3)/2*2}) circle (0.07);
\filldraw[orange] (3/4*2, {-sqrt(3)/4*2}) circle (0.07);
\filldraw[orange] (-3/4*2, {-sqrt(3)/4*2}) circle (0.07);
\end{tikzpicture}
}}
\;\;=\;\;
(-A^2)^{-\mu(n-1)}
\vcenter{\hbox{
\begin{tikzpicture}
\draw[line width=5] (-3*3/16*2, {sqrt(3)*5/16*2}) to[out=-60, in=0] (-3/4*2, {sqrt(3)/8*2});
\node[anchor=south] at (-3*3/16*2, {sqrt(3)*5/16*2}){$\mu$};
\node[anchor=east] at (-3/4*2, {sqrt(3)/8*2}){$\mu$};
\draw[line width=5] (3/4*2, {sqrt(3)/8*2}) to[out=180, in=-120] (3*3/16*2, {sqrt(3)*5/16*2});
\node[anchor=south] at (3*3/16*2, {sqrt(3)*5/16*2}){$\mu$};
\node[anchor=west] at (3/4*2, {sqrt(3)/8*2}){$\mu$};
\draw[line width=5] (-3*3/16*2, {-sqrt(3)*5/16*2}) to[out=60, in=0] (-3/4*2, {-sqrt(3)/8*2});
\node[anchor=east] at (-3/4*2, {-sqrt(3)/8*2}){$-\mu$};
\node[anchor=north] at (-3*3/16*2-0.1, {-sqrt(3)*5/16*2}){$-\mu$};
\draw[line width=5] (3/4*2, {-sqrt(3)/8*2}) to[out=180, in=120] (3*3/16*2, {-sqrt(3)*5/16*2});
\node[anchor=north] at (3*3/16*2, {-sqrt(3)*5/16*2}){$-\mu$};
\node[anchor=west] at (3/4*2, {-sqrt(3)/8*2}){$-\mu$};
\draw[line width=5] (-3/16*2, {-sqrt(3)*7/16*2}) to[out=60, in=120] (3/16*2, {-sqrt(3)*7/16*2});
\node[anchor=north] at (-3/16*2, {-sqrt(3)*7/16*2}){$\mu$};
\node[anchor=north] at (3/16*2, {-sqrt(3)*7/16*2}){$\mu$};
\draw[orange, very thick] (-3/4*2, {sqrt(3)/4*2}) -- (0, {sqrt(3)/2*2});
\draw[orange, very thick, ->] (-3/4*2, {sqrt(3)/4*2}) -- (-3/8*2, {sqrt(3)*3/8*2});
\draw[orange, very thick] (3/4*2, {sqrt(3)/4*2}) -- (0, {sqrt(3)/2*2});
\draw[orange, very thick, ->] (3/4*2, {sqrt(3)/4*2}) -- (3/8*2, {sqrt(3)*3/8*2});
\draw[orange, very thick] (3/4*2, {sqrt(3)/4*2}) -- (3/4*2, {-sqrt(3)/4*2});
\draw[orange, very thick, ->] (3/4*2, {sqrt(3)/4*2}) -- (3/4*2, 0);
\draw[orange, very thick] (0, {-sqrt(3)/2*2}) -- (3/4*2, {-sqrt(3)/4*2});
\draw[orange, very thick, ->] (0, {-sqrt(3)/2*2}) -- (3/8*2, {-sqrt(3)*3/8*2});
\draw[orange, very thick] (0, {-sqrt(3)/2*2}) -- (-3/4*2, {-sqrt(3)/4*2});
\draw[orange, very thick, ->] (0, {-sqrt(3)/2*2}) -- (-3/8*2, {-sqrt(3)*3/8*2});
\draw[orange, very thick] (-3/4*2, {sqrt(3)/4*2}) -- (-3/4*2, {-sqrt(3)/4*2});
\draw[orange, very thick, ->] (-3/4*2, {sqrt(3)/4*2}) -- (-3/4*2, 0);
\draw[blue, dashed] (0, 0) -- (1*2, 0);
\draw[blue, dashed] (0, 0) -- (-1*2, 0);
\draw[blue, dashed] (0, 0) -- (1/2*2, {sqrt(3)/2*2});
\draw[blue, dashed] (0, 0) -- (-1/2*2, {sqrt(3)/2*2});
\draw[blue, dashed] (0, 0) -- (1/2*2, {-sqrt(3)/2*2});
\draw[blue, dashed] (0, 0) -- (-1/2*2, {-sqrt(3)/2*2});
\filldraw[fill=white] (0, 0) circle (0.1);
\filldraw[orange] (-3/4*2, {sqrt(3)/4*2}) circle (0.07);
\filldraw[orange] (3/4*2, {sqrt(3)/4*2}) circle (0.07);
\filldraw[orange] (0, {sqrt(3)/2*2}) circle (0.07);
\filldraw[orange] (0, {-sqrt(3)/2*2}) circle (0.07);
\filldraw[orange] (3/4*2, {-sqrt(3)/4*2}) circle (0.07);
\filldraw[orange] (-3/4*2, {-sqrt(3)/4*2}) circle (0.07);
\end{tikzpicture}
}}
\;,
\]
\item \begin{align*}
\vcenter{\hbox{
\begin{tikzpicture}[scale=0.8]
\draw[line width=5, gray] (-3/16*2, {sqrt(3)*7/16*2}) to[out=-60, in=-120] (3/16*2, {sqrt(3)*7/16*2});
\node[anchor=south] at (-3/16*2, {sqrt(3)*7/16*2}){$-$};
\node[anchor=south] at (3/16*2, {sqrt(3)*7/16*2}){$+$};
\draw[orange, very thick] (-3/4*2, {sqrt(3)/4*2}) -- (0, {sqrt(3)/2*2});
\draw[orange, very thick, ->] (-3/4*2, {sqrt(3)/4*2}) -- (-3/8*2, {sqrt(3)*3/8*2});
\draw[orange, very thick] (3/4*2, {sqrt(3)/4*2}) -- (0, {sqrt(3)/2*2});
\draw[orange, very thick, ->] (3/4*2, {sqrt(3)/4*2}) -- (3/8*2, {sqrt(3)*3/8*2});
\draw[orange, very thick] (3/4*2, {sqrt(3)/4*2}) -- (3/4*2, {-sqrt(3)/4*2});
\draw[orange, very thick, ->] (3/4*2, {sqrt(3)/4*2}) -- (3/4*2, 0);
\draw[orange, very thick] (0, {-sqrt(3)/2*2}) -- (3/4*2, {-sqrt(3)/4*2});
\draw[orange, very thick, ->] (0, {-sqrt(3)/2*2}) -- (3/8*2, {-sqrt(3)*3/8*2});
\draw[orange, very thick] (0, {-sqrt(3)/2*2}) -- (-3/4*2, {-sqrt(3)/4*2});
\draw[orange, very thick, ->] (0, {-sqrt(3)/2*2}) -- (-3/8*2, {-sqrt(3)*3/8*2});
\draw[orange, very thick] (-3/4*2, {sqrt(3)/4*2}) -- (-3/4*2, {-sqrt(3)/4*2});
\draw[orange, very thick, ->] (-3/4*2, {sqrt(3)/4*2}) -- (-3/4*2, 0);
\draw[blue, dashed] (0, 0) -- (1*2, 0);
\draw[blue, dashed] (0, 0) -- (-1*2, 0);
\draw[blue, dashed] (0, 0) -- (1/2*2, {sqrt(3)/2*2});
\draw[blue, dashed] (0, 0) -- (-1/2*2, {sqrt(3)/2*2});
\draw[blue, dashed] (0, 0) -- (1/2*2, {-sqrt(3)/2*2});
\draw[blue, dashed] (0, 0) -- (-1/2*2, {-sqrt(3)/2*2});
\filldraw[fill=white] (0, 0) circle (0.1);
\filldraw[orange] (-3/4*2, {sqrt(3)/4*2}) circle (0.07);
\filldraw[orange] (3/4*2, {sqrt(3)/4*2}) circle (0.07);
\filldraw[orange] (0, {sqrt(3)/2*2}) circle (0.07);
\filldraw[orange] (0, {-sqrt(3)/2*2}) circle (0.07);
\filldraw[orange] (3/4*2, {-sqrt(3)/4*2}) circle (0.07);
\filldraw[orange] (-3/4*2, {-sqrt(3)/4*2}) circle (0.07);
\end{tikzpicture}
}}
&\;\;=\;\;
(-A^2)^{-n+1}
\vcenter{\hbox{
\begin{tikzpicture}[scale=0.8]
\draw[line width=5, gray] (-3*3/16*2, {sqrt(3)*5/16*2}) to[out=-60, in=0] (-3/4*2, {sqrt(3)/8*2});
\node[anchor=south] at (-3*3/16*2, {sqrt(3)*5/16*2}){$-$};
\node[anchor=east] at (-3/4*2, {sqrt(3)/8*2}){$+$};
\draw[line width=5] (3/4*2, {sqrt(3)/8*2}) to[out=180, in=-120] (3*3/16*2, {sqrt(3)*5/16*2});
\node[anchor=south] at (3*3/16*2, {sqrt(3)*5/16*2}){$+$};
\draw[line width=5] (-3*3/16*2, {-sqrt(3)*5/16*2}) to[out=60, in=0] (-3/4*2, {-sqrt(3)/8*2});
\draw[line width=5] (3/4*2, {-sqrt(3)/8*2}) to[out=180, in=120] (3*3/16*2, {-sqrt(3)*5/16*2});
\draw[line width=5] (-3/16*2, {-sqrt(3)*7/16*2}) to[out=60, in=120] (3/16*2, {-sqrt(3)*7/16*2});
\draw[orange, very thick] (-3/4*2, {sqrt(3)/4*2}) -- (0, {sqrt(3)/2*2});
\draw[orange, very thick, ->] (-3/4*2, {sqrt(3)/4*2}) -- (-3/8*2, {sqrt(3)*3/8*2});
\draw[orange, very thick] (3/4*2, {sqrt(3)/4*2}) -- (0, {sqrt(3)/2*2});
\draw[orange, very thick, ->] (3/4*2, {sqrt(3)/4*2}) -- (3/8*2, {sqrt(3)*3/8*2});
\draw[orange, very thick] (3/4*2, {sqrt(3)/4*2}) -- (3/4*2, {-sqrt(3)/4*2});
\draw[orange, very thick, ->] (3/4*2, {sqrt(3)/4*2}) -- (3/4*2, 0);
\draw[orange, very thick] (0, {-sqrt(3)/2*2}) -- (3/4*2, {-sqrt(3)/4*2});
\draw[orange, very thick, ->] (0, {-sqrt(3)/2*2}) -- (3/8*2, {-sqrt(3)*3/8*2});
\draw[orange, very thick] (0, {-sqrt(3)/2*2}) -- (-3/4*2, {-sqrt(3)/4*2});
\draw[orange, very thick, ->] (0, {-sqrt(3)/2*2}) -- (-3/8*2, {-sqrt(3)*3/8*2});
\draw[orange, very thick] (-3/4*2, {sqrt(3)/4*2}) -- (-3/4*2, {-sqrt(3)/4*2});
\draw[orange, very thick, ->] (-3/4*2, {sqrt(3)/4*2}) -- (-3/4*2, 0);
\draw[blue, dashed] (0, 0) -- (1*2, 0);
\draw[blue, dashed] (0, 0) -- (-1*2, 0);
\draw[blue, dashed] (0, 0) -- (1/2*2, {sqrt(3)/2*2});
\draw[blue, dashed] (0, 0) -- (-1/2*2, {sqrt(3)/2*2});
\draw[blue, dashed] (0, 0) -- (1/2*2, {-sqrt(3)/2*2});
\draw[blue, dashed] (0, 0) -- (-1/2*2, {-sqrt(3)/2*2});
\filldraw[fill=white] (0, 0) circle (0.1);
\filldraw[orange] (-3/4*2, {sqrt(3)/4*2}) circle (0.07);
\filldraw[orange] (3/4*2, {sqrt(3)/4*2}) circle (0.07);
\filldraw[orange] (0, {sqrt(3)/2*2}) circle (0.07);
\filldraw[orange] (0, {-sqrt(3)/2*2}) circle (0.07);
\filldraw[orange] (3/4*2, {-sqrt(3)/4*2}) circle (0.07);
\filldraw[orange] (-3/4*2, {-sqrt(3)/4*2}) circle (0.07);
\end{tikzpicture}
}}
\;+\;
(-A^2)^{-n+2}
\vcenter{\hbox{
\begin{tikzpicture}[scale=0.8]
\draw[line width=5] (-3*3/16*2, {sqrt(3)*5/16*2}) to[out=-60, in=0] (-3/4*2, {sqrt(3)/8*2});
\node[anchor=south] at (-3*3/16*2, {sqrt(3)*5/16*2}){$-$};
\draw[line width=5] (3/4*2, {sqrt(3)/8*2}) to[out=180, in=-120] (3*3/16*2, {sqrt(3)*5/16*2});
\node[anchor=south] at (3*3/16*2, {sqrt(3)*5/16*2}){$+$};
\draw[line width=5, gray] (-3*3/16*2, {-sqrt(3)*5/16*2}) to[out=60, in=0] (-3/4*2, {-sqrt(3)/8*2});
\node[anchor=east] at (-3/4*2, {-sqrt(3)/8*2}){$+$};
\node[anchor=north] at (-3*3/16*2, {-sqrt(3)*5/16*2}){$-$};
\draw[line width=5] (3/4*2, {-sqrt(3)/8*2}) to[out=180, in=120] (3*3/16*2, {-sqrt(3)*5/16*2});
\draw[line width=5] (-3/16*2, {-sqrt(3)*7/16*2}) to[out=60, in=120] (3/16*2, {-sqrt(3)*7/16*2});
\draw[orange, very thick] (-3/4*2, {sqrt(3)/4*2}) -- (0, {sqrt(3)/2*2});
\draw[orange, very thick, ->] (-3/4*2, {sqrt(3)/4*2}) -- (-3/8*2, {sqrt(3)*3/8*2});
\draw[orange, very thick] (3/4*2, {sqrt(3)/4*2}) -- (0, {sqrt(3)/2*2});
\draw[orange, very thick, ->] (3/4*2, {sqrt(3)/4*2}) -- (3/8*2, {sqrt(3)*3/8*2});
\draw[orange, very thick] (3/4*2, {sqrt(3)/4*2}) -- (3/4*2, {-sqrt(3)/4*2});
\draw[orange, very thick, ->] (3/4*2, {sqrt(3)/4*2}) -- (3/4*2, 0);
\draw[orange, very thick] (0, {-sqrt(3)/2*2}) -- (3/4*2, {-sqrt(3)/4*2});
\draw[orange, very thick, ->] (0, {-sqrt(3)/2*2}) -- (3/8*2, {-sqrt(3)*3/8*2});
\draw[orange, very thick] (0, {-sqrt(3)/2*2}) -- (-3/4*2, {-sqrt(3)/4*2});
\draw[orange, very thick, ->] (0, {-sqrt(3)/2*2}) -- (-3/8*2, {-sqrt(3)*3/8*2});
\draw[orange, very thick] (-3/4*2, {sqrt(3)/4*2}) -- (-3/4*2, {-sqrt(3)/4*2});
\draw[orange, very thick, ->] (-3/4*2, {sqrt(3)/4*2}) -- (-3/4*2, 0);
\draw[blue, dashed] (0, 0) -- (1*2, 0);
\draw[blue, dashed] (0, 0) -- (-1*2, 0);
\draw[blue, dashed] (0, 0) -- (1/2*2, {sqrt(3)/2*2});
\draw[blue, dashed] (0, 0) -- (-1/2*2, {sqrt(3)/2*2});
\draw[blue, dashed] (0, 0) -- (1/2*2, {-sqrt(3)/2*2});
\draw[blue, dashed] (0, 0) -- (-1/2*2, {-sqrt(3)/2*2});
\filldraw[fill=white] (0, 0) circle (0.1);
\filldraw[orange] (-3/4*2, {sqrt(3)/4*2}) circle (0.07);
\filldraw[orange] (3/4*2, {sqrt(3)/4*2}) circle (0.07);
\filldraw[orange] (0, {sqrt(3)/2*2}) circle (0.07);
\filldraw[orange] (0, {-sqrt(3)/2*2}) circle (0.07);
\filldraw[orange] (3/4*2, {-sqrt(3)/4*2}) circle (0.07);
\filldraw[orange] (-3/4*2, {-sqrt(3)/4*2}) circle (0.07);
\end{tikzpicture}
}}
\\
&\quad\quad\quad+\;\;\cdots
\;\;+\;\;
(-A^2)^{n-1}
\vcenter{\hbox{
\begin{tikzpicture}[scale=0.8]
\draw[line width=5] (-3*3/16*2, {sqrt(3)*5/16*2}) to[out=-60, in=0] (-3/4*2, {sqrt(3)/8*2});
\node[anchor=south] at (-3*3/16*2, {sqrt(3)*5/16*2}){$-$};
\draw[line width=5, gray] (3/4*2, {sqrt(3)/8*2}) to[out=180, in=-120] (3*3/16*2, {sqrt(3)*5/16*2});
\node[anchor=south] at (3*3/16*2, {sqrt(3)*5/16*2}){$+$};
\node[anchor=west] at (3/4*2, {sqrt(3)/8*2}){$-$};
\draw[line width=5] (-3*3/16*2, {-sqrt(3)*5/16*2}) to[out=60, in=0] (-3/4*2, {-sqrt(3)/8*2});
\draw[line width=5] (3/4*2, {-sqrt(3)/8*2}) to[out=180, in=120] (3*3/16*2, {-sqrt(3)*5/16*2});
\draw[line width=5] (-3/16*2, {-sqrt(3)*7/16*2}) to[out=60, in=120] (3/16*2, {-sqrt(3)*7/16*2});
\draw[orange, very thick] (-3/4*2, {sqrt(3)/4*2}) -- (0, {sqrt(3)/2*2});
\draw[orange, very thick, ->] (-3/4*2, {sqrt(3)/4*2}) -- (-3/8*2, {sqrt(3)*3/8*2});
\draw[orange, very thick] (3/4*2, {sqrt(3)/4*2}) -- (0, {sqrt(3)/2*2});
\draw[orange, very thick, ->] (3/4*2, {sqrt(3)/4*2}) -- (3/8*2, {sqrt(3)*3/8*2});
\draw[orange, very thick] (3/4*2, {sqrt(3)/4*2}) -- (3/4*2, {-sqrt(3)/4*2});
\draw[orange, very thick, ->] (3/4*2, {sqrt(3)/4*2}) -- (3/4*2, 0);
\draw[orange, very thick] (0, {-sqrt(3)/2*2}) -- (3/4*2, {-sqrt(3)/4*2});
\draw[orange, very thick, ->] (0, {-sqrt(3)/2*2}) -- (3/8*2, {-sqrt(3)*3/8*2});
\draw[orange, very thick] (0, {-sqrt(3)/2*2}) -- (-3/4*2, {-sqrt(3)/4*2});
\draw[orange, very thick, ->] (0, {-sqrt(3)/2*2}) -- (-3/8*2, {-sqrt(3)*3/8*2});
\draw[orange, very thick] (-3/4*2, {sqrt(3)/4*2}) -- (-3/4*2, {-sqrt(3)/4*2});
\draw[orange, very thick, ->] (-3/4*2, {sqrt(3)/4*2}) -- (-3/4*2, 0);
\draw[blue, dashed] (0, 0) -- (1*2, 0);
\draw[blue, dashed] (0, 0) -- (-1*2, 0);
\draw[blue, dashed] (0, 0) -- (1/2*2, {sqrt(3)/2*2});
\draw[blue, dashed] (0, 0) -- (-1/2*2, {sqrt(3)/2*2});
\draw[blue, dashed] (0, 0) -- (1/2*2, {-sqrt(3)/2*2});
\draw[blue, dashed] (0, 0) -- (-1/2*2, {-sqrt(3)/2*2});
\filldraw[fill=white] (0, 0) circle (0.1);
\filldraw[orange] (-3/4*2, {sqrt(3)/4*2}) circle (0.07);
\filldraw[orange] (3/4*2, {sqrt(3)/4*2}) circle (0.07);
\filldraw[orange] (0, {sqrt(3)/2*2}) circle (0.07);
\filldraw[orange] (0, {-sqrt(3)/2*2}) circle (0.07);
\filldraw[orange] (3/4*2, {-sqrt(3)/4*2}) circle (0.07);
\filldraw[orange] (-3/4*2, {-sqrt(3)/4*2}) circle (0.07);
\end{tikzpicture}
}}
\;.
\end{align*}
\end{enumerate}
In the second relation above, which is a $2n$-term relation, we highlighted the elementary tangles with opposite states on its end points with gray color. 
We haven't labeled all the states, but there is a unique way of labeling the remaining states in such a way that each term could possibly be non-zero in the reduced skein algebra. 

\begin{defn}
An \emph{$n$-gonal pillow} $nP$ is the $3$-ball $B^3$
along with an embedded oriented graph $\Gamma \subset \partial B^3 = S^2$ on the boundary, which is dual to
the cell decomposition of $S^2$ into two 2-cells, both of which are $n$-gons whose boundaries are identified with each other. 
\end{defn}

\begin{eg}
A triangular pillow $3P$ looks like Figure \ref{fig:triangular_pillow}. 
\begin{figure}[H]
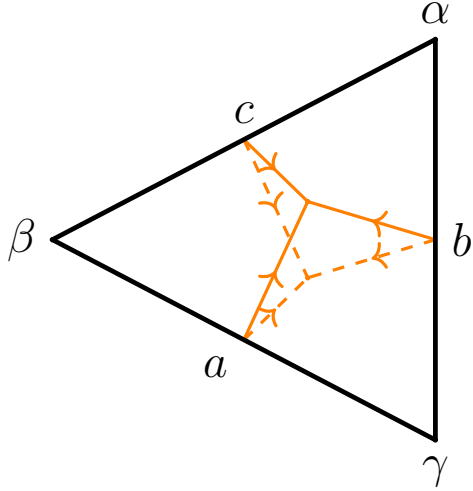

    \centering
    \includestandalone[scale=1.5]{figures/triangular_pillow}
    \caption{A triangular pillow (or a face suspension).}
    \label{fig:triangular_pillow}
\end{figure}
Note, each face suspension $Sf$ is topologically equivalent to a triangular pillow. 
\end{eg}

Applying Corollary \ref{cor:RedSkeinModuleOf3Balls} to a triangular pillow, we immediately get:
\begin{cor}\label{cor:red_skein_module_fs}
The reduced stated skein module of a triangular pillow is given by 
\[
\overline{\Sk}(3P) \cong \frac{\mathbb{T}^{\otimes 2} \otimes \mathbb{B}^{\otimes 3}}{
\langle
(-A^2)\alpha_1\alpha_2 = x_b x_c, \;
(-A^2)\beta_1\beta_2 = x_c x_a, \;
(-A^2)\gamma_1\gamma_2 = x_a x_b
\rangle}
\]
as a left $\mathbb{T}^{\otimes 2} \otimes (\mathbb{B}^{\mathrm{op}})^{\otimes 3} = \mathbb{T}^{\otimes 2} \otimes \mathbb{B}^{\otimes 3}$-module, 
where
\[
\mathbb{T}^{\otimes 2} =
\frac{R\langle \alpha_1^{\pm 1}, \beta_1^{\pm 1}, \gamma_1^{\pm 1}\rangle}{\langle \beta_1\alpha_1 = A\alpha_1\beta_1, \gamma_1\beta_1 = A\beta_1\gamma_1, \alpha_1\gamma_1 = A\gamma_1\alpha_1 \rangle}
\otimes
\frac{R\langle \alpha_2^{\pm 1}, \beta_2^{\pm 1}, \gamma_2^{\pm 1}\rangle}{\langle \alpha_2\beta_2 = A\beta_2\alpha_2, \beta_2\gamma_2 = A\gamma_2\beta_2, \gamma_2\alpha_2 = A\alpha_2\gamma_2 \rangle}
\]
and
\[
\mathbb{B}^{\otimes 3} = R[x_a^{\pm 1}, x_b^{\pm 1}, x_c^{\pm 1}]. 
\]
\end{cor}

\begin{rmk}
One interesting feature of an $n$-gonal pillow is that it can be split into two $n$-gonal pillows, along an $n$-gon. 
That is, we have splitting maps
\[
\sigma : \Sk(nP) \rightarrow \Sk(nP) \;\overline{\otimes}\; \Sk(nP)
\]
and
\[
\overline{\sigma} : \overline{\Sk}(nP) \rightarrow \overline{\Sk}(nP) \;\overline{\otimes}\; \overline{\Sk}(nP).
\]
The reduced tensor product $\Sk(nP) \;\overline{\otimes}\; \Sk(nP)$ (and similarly its reduced version, $\overline{\Sk}(nP) \;\overline{\otimes}\; \overline{\Sk}(nP)$) above is easy to describe: 
it is the usual tensor product of left and right modules over the same algebra, $\SkAlg(D_n)$, after turning one of the left $\SkAlg(D_n)$-module structures into a right $\SkAlg(D_n)$-module structure using the self-duality map in Proposition \ref{prop:SelfDuality}. 
The right $\SkAlg(D_2)$-module structures on $\Sk(nP) \;\overline{\otimes}\; \Sk(nP)$ is induced from the usual coalgebra structure on $\SkAlg(D_2)$, which is the splitting map for splitting a bigon into two bigons. 
In this way, we get a coalgebra-like structure on $\Sk(nP)$. 

If $Y$ is a 3-manifold whose boundary marking $\Gamma$ is induced from an ideal triangulation of its boundary, 
then for any face of the boundary ideal triangulation, we can split a triangular pillow off of $Y$. 
That is, $\Sk(Y,\Gamma)$ is naturally equipped with a comodule-like structure over $\Sk(3P)$:
\[
\sigma : \Sk(Y,\Gamma) \rightarrow \Sk(Y,\Gamma) \;\overline{\otimes}\; \Sk(3P).
\]
This is a 3-dimensional analog of the coalgebra structure on $\SkAlg(D_2)$ and the $\SkAlg(D_2)$-comodule structure on $\SkAlg(\Sigma)$ for any bordered punctured surface $\Sigma$ with a choice of a boundary interval, as described in \cite{CL}. 
\end{rmk}


\section{Quantum trace map for triangulated 3-manifolds}\label{sec:3dQuantumTrace}
In this section, we construct the quantum trace map for ideally triangulated 3-manifolds. 
Our strategy is to cut the triangulated 3-manifold $Y$ into face suspensions $Y = \cup_{f\in \mathcal{F}} Sf$, construct the quantum trace map for each face suspension $Sf$, and then show that they glue well, by using the splitting theorem we proved earlier. 

\subsection{Quantum trace map for a single face suspension}
First, let's recall the definition of the Weyl-ordered product in a quantum torus. 
\begin{defn}
Let $\Lambda$ be a lattice with a basis $e_1, \cdots, e_n$, equipped with an anti-symmetric bilinear form 
$\langle \, , \rangle : \Lambda \times \Lambda \rightarrow \mathbb{Z}$. 
The associated \emph{quantum torus} $\mathrm{QT}_{\Lambda}$ is the unital associative algebra over a ring containing an invertible element $A$, generated by formal variables $x_1, \cdots, x_n$
with commutation relations
\[
x_i x_j = A^{\langle e_i, e_j\rangle} x_j x_i. 
\]

The \emph{Weyl-ordered product} $[x_{i_1} \cdots x_{i_k}]$ of $x_{i_1}, \cdots, x_{i_k}$ is defined to be
\[
[x_{i_1} \cdots  x_{i_k}] := A^{-\frac{\sum_{1\leq j_1 < j_2 \leq k} \langle e_{i_{j_1}}, e_{i_{j_2}}\rangle}{2}} x_{i_1} \cdots  x_{i_k}. 
\]
Note, the Weyl-ordered product is independent of the ordering of the variables. 
\end{defn}

Recall from \eqref{eq:trianglealg} that the triangle algebra $\mathbb{T}$, which is an example of a quantum torus, is the reduced stated skein algebra of a triangle. 
It is generated by $A$-commuting variables $\alpha, \beta, \gamma$ associated to the three vertices of the triangle. 
The triangle algebra admits a natural extension generated by $A$-commuting variables $a, b, c$ associated to the edges of the triangle. 

\begin{defn}
The \emph{extended triangle algebra} $\widetilde{\mathbb{T}}$ is
an extension of the triangle algebra $\mathbb{T}$
defined as
\[
\widetilde{\mathbb{T}} := \frac{R\langle a^{\pm 1}, b^{\pm 1}, c^{\pm 1}\rangle}{\langle ba = Aab, cb = Abc, ac = Aca \rangle}.
\]
It is equipped with the following embedding of the triangle algebra $\mathbb{T}$: 
\begin{align*}
\mathbb{T} &\hookrightarrow \widetilde{\mathbb{T}}\\
\alpha &\mapsto A^{-1} [bc],\\
\beta &\mapsto A^{-1} [ca],\\
\gamma &\mapsto A^{-1} [ab].
\end{align*}
\end{defn}

\begin{defn}
Let $Sf$ be a face suspension. 
The \emph{face suspension module} $\mathbb{S}f$ of $Sf$ is 
\[
\mathbb{S}f := \widetilde{\mathbb{T}}^{\otimes 2} = 
\frac{R\langle a_1^{\pm 1}, b_1^{\pm 1}, c_1^{\pm 1}\rangle \otimes R\langle a_2^{\pm 1}, b_2^{\pm 1}, c_2^{\pm 1}\rangle}{\langle b_1a_1 = Aa_1b_1, c_1b_1 = Ab_1c_1, a_1c_1 = Ac_1a_1,\; a_2b_2 = Ab_2a_2, b_2c_2 = Ac_2b_2, c_2a_2 = Aa_2c_2\rangle}
\]
as a regular $\widetilde{\mathbb{T}}^{\otimes 2}$-$\widetilde{\mathbb{T}}^{\otimes 2}$ bimodule. 
\end{defn}
We will think of the $6$ generators of $\mathbb{S}f$ as being naturally associated to the $6$ edge cones on the boundary of $Sf$: 
$a_1, b_1, c_1$ (resp.\ $a_2, b_2, c_2$) are associated to the top (resp.\ bottom) $3$ edge cones in Figure \ref{fig:triangular_pillow}. 

The following lemma equips $\mathbb{S}f$ with a structure of a $\mathbb{T}^{\otimes 2}$-$\mathbb{B}^{\otimes 3}$-bimodule: 
\begin{lem}\label{lem:triangle_biangle_embedding}
We have the following embeddings of algebras
\begin{align*}
\mathbb{T}^{\otimes 2} &\hookrightarrow \widetilde{\mathbb{T}}^{\otimes 2}\\
\alpha_1 &\mapsto A^{-1} [b_1c_1],\\
\beta_1 &\mapsto A^{-1} [c_1a_1],\\
\gamma_1 &\mapsto A^{-1} [a_1b_1],\\
\alpha_2 &\mapsto A^{-1} [b_2c_2],\\
\beta_2 &\mapsto A^{-1} [c_2a_2],\\
\gamma_2 &\mapsto A^{-1} [a_2b_2],
\end{align*}
and
\begin{align*}
\mathbb{B}^{\otimes 3} &\hookrightarrow \widetilde{\mathbb{T}}^{\otimes 2}\\
x_a &\mapsto (-1)^{-\frac{1}{2}}a_1 \otimes a_2,\\
x_b &\mapsto (-1)^{-\frac{1}{2}}b_1 \otimes b_2,\\
x_c &\mapsto (-1)^{-\frac{1}{2}}c_1 \otimes c_2.
\end{align*}
\end{lem}

We define the quantum trace $\Tr_{Sf}$ for a single face suspension as follows:
\begin{thm} \label{thm:fsQT}
There is a well-defined $\mathbb{T}^{\otimes 2}$-$\mathbb{B}^{\otimes 3}$-bimodule homomorphism 
\[
\Tr_{Sf} : \overline{\Sk}(Sf) \rightarrow \mathbb{S}f
\]
mapping the empty skein $[\emptyset]$ to $1$. 
\end{thm}
\begin{proof}
The embeddings of Lemma \ref{lem:triangle_biangle_embedding} induce a $\mathbb{T}^{\otimes 2}$-$\mathbb{B}^{\otimes 3}$-bimodule map
\[
\mathbb{T}^{\otimes 2} \otimes \mathbb{B}^{\otimes 3} \rightarrow \widetilde{\mathbb{T}}^{\otimes 2}.
\]
The relations defining $\overline{\Sk}(Sf)$ in Corollary \ref{cor:red_skein_module_fs} are preserved under this map; for instance, the first relation is preserved because
\begin{align*}
(-A^2)\alpha_1 \alpha_2 [\emptyset] \;&\mapsto\; -[b_1 c_1] \otimes [b_2 c_2], \\
[\emptyset]x_b x_c \;&\mapsto\; -b_1c_1 \otimes b_2c_2 = -[b_1 c_1] \otimes [b_2 c_2],
\end{align*}
and likewise the other two relations are preserved. 
Therefore, there is a well-defined $\mathbb{T}^{\otimes 2}$-$\mathbb{B}^{\otimes 3}$-bimodule homomorphism
\[
\overline{\Sk}(Sf) \rightarrow \widetilde{\mathbb{T}}^{\otimes 2}.
\]
\end{proof}

\begin{thm}\label{thm:InjectivityOfTrSf}
The map $\Tr_{Sf}$ is injective. 
\end{thm}
\begin{proof}
It is easy to see that the image of $\Tr_{Sf}$ is exactly the submodule of $\widetilde{\mathbb{T}}^{\otimes 2}$ spanned by monomials of even total degree. 
We claim that $\Tr_{Sf}$ is an isomorphism from $\overline{\Sk}(Sf)$ to this submodule. 
It suffices to show that in each grading $\vec{n} = (n_{a_1}, n_{b_1}, n_{c_1}, n_{a_2}, n_{b_2}, n_{c_2}) \in \mathbb{Z}^{E(\Gamma)}$ (Remark \ref{rmk:SkeinModuleIsGraded}) with $\sum \vec{n}$ even, 
$\overline{\Sk}(Sf)$ is 1-dimensional. 
For this, note that, for any monomial in $\alpha_1, \beta_1, \gamma_1, \alpha_2, \beta_2, \gamma_2, x_a, x_b, x_c$, 
we can apply the relations defining $\overline{\Sk}(Sf)$ (Corollary \ref{cor:red_skein_module_fs}) to express it as a monomial in $\alpha_1, \beta_1, \gamma_1, x_a, x_b, x_c$. 
It is easy to see that such a monomial is uniquely determined by its grading $\vec{n} = (n_{a_1}, n_{b_1}, n_{c_1}, n_{a_2}, n_{b_2}, n_{c_2})$ in $\mathbb{Z}^{E(\Gamma)}$; the monomial must be
\[
\alpha_1^{\frac{-n_{a_1}+n_{b_1}+n_{c_1}+n_{a_2}-n_{b_2}-n_{c_2}}{2}} \beta_1^{\frac{n_{a_1}-n_{b_1}+n_{c_1}-n_{a_2}+n_{b_2}-n_{c_2}}{2}} \gamma_1^{\frac{n_{a_1}+n_{b_1}-n_{c_1}-n_{a_2}-n_{b_2}+n_{c_2}}{2}} x_a^{n_{a_2}} x_b^{n_{b_2}} x_c^{n_{c_2}}. 
\]
It follows that $\overline{\Sk}(Sf)$ is 1-dimensional in each graded piece. 
\end{proof}


\subsection{Square-root quantum gluing modules and the main theorem}
For an ideally triangulated 3-manifold $Y$ (without boundary except for cusps at infinity) decomposed into face suspensions $Y = \cup_{f\in \mathcal{F}} Sf$,
we define $\overline{\bigotimes}_{f\in \mathcal{F}} \mathbb{S}f$ to be the maximal quotient of $\bigotimes_{f\in \mathcal{F}} \mathbb{S}f$ such that the tensor product 
\[
\otimes_{f\in \mathcal{F}}\Tr_{Sf} : 
\bigotimes_{f\in \mathcal{F}} \overline{\Sk}(Sf) \rightarrow 
\bigotimes_{f\in \mathcal{F}}\mathbb{S}f
\]
of the homomorphisms $\Tr_{Sf}$ descends to the quotient
\[
\overline{\otimes}_{f\in \mathcal{F}} \Tr_{Sf} : 
\overline{\bigotimes}_{f\in \mathcal{F}} \overline{\Sk}(Sf) \rightarrow 
\overline{\bigotimes}_{f\in \mathcal{F}}\mathbb{S}f.
\]
That is, it will fit into the following commutative diagram 
\[
\begin{tikzcd}
& \bigotimes_{f\in \mathcal{F}} \overline{\Sk}(Sf) \arrow[d] \arrow[r, "\otimes \Tr_{Sf}"] & \bigotimes_{f\in \mathcal{F}}\mathbb{S}f \arrow[d]\\
\Sk(Y) = \overline{\Sk}(Y) \arrow[r, "\overline{\sigma}"] \arrow[rr, bend right=20, dashed, "\Tr_{\mathcal{T}} \;:=\; \overline{\otimes} \Tr_{Sf} \,\circ\; \overline{\sigma}"'] & \overline{\bigotimes}_{f\in \mathcal{F}} \overline{\Sk}(Sf) \arrow[r, dashed, "\overline{\otimes} \Tr_{Sf}"] & \overline{\bigotimes}_{f\in \mathcal{F}}\mathbb{S}f
\end{tikzcd},
\]
where the vertical arrows are quotients. 

Let $U_{+}$ and $U_{-}$ be respectively the $R$-submodules of $\otimes_{f\in \mathcal{F}}\mathbb{T}^{\otimes 2}$ and $\otimes_{f\in \mathcal{F}}{\mathbb{B}}^{\otimes 3}$ spanned by the relations in Corollary \ref{cor:fs_splitting_reduced} so that
\[
\overline{\bigotimes}_{f\in \mathcal{F}} \overline{\Sk}(Sf)
= 
\frac{\bigotimes_{f\in \mathcal{F}} \overline{\Sk}(Sf)}{U_{+}\qty(\bigotimes_{f\in \mathcal{F}} \overline{\Sk}(Sf)) +
\qty(\bigotimes_{f\in \mathcal{F}} \overline{\Sk}(Sf))U_{-}}. 
\]
Then, by our definition above, the maximal quotient is the quotient
\[
\overline{\bigotimes}_{f\in \mathcal{F}} \mathbb{S}f 
:= 
\frac{\bigotimes_{f\in \mathcal{F}} \mathbb{S}f}{V_{+} \qty(\bigotimes_{f\in \mathcal{F}} \mathbb{S}f) + 
\qty(\bigotimes_{f\in \mathcal{F}} \mathbb{S}f) V_{-}}
\]
of $\otimes_{f\in \mathcal{F}} \mathbb{S}f \cong \otimes_{f\in \mathcal{F}}\widetilde{\mathbb{T}}^{\otimes 2}$, 
where $V_+$ (resp.\ $V_-$) is the image of $U_+$ (resp.\ $U_-$) under $\otimes_{f\in \mathcal{F}} \Tr_{Sf}$. 

\begin{prop}\label{prop:TensorProductOfQuantumTraceIsInjective}
The map $\overline{\otimes}_{f\in \mathcal{F}} \Tr_{Sf}$ is injective. 
\end{prop}
\begin{proof}
We saw in the proof of Theorem \ref{thm:InjectivityOfTrSf} that $\Tr_{Sf} : \overline{\Sk}(Sf) \rightarrow \mathbb{S}f$ is an embedding onto the even part with respect to the total $\mathbb{Z}$-grading on $\mathbb{S}f$. 
Likewise, 
\[
\otimes_{f\in \mathcal{F}} \Tr_{Sf} : 
\bigotimes_{f\in \mathcal{F}} \overline{\Sk}(Sf) \rightarrow
\bigotimes_{f\in \mathcal{F}} \mathbb{S}f
\]
is an embedding onto the even part with respect to the $\mathbb{Z}^{\mathcal{F}}$-grading on $\bigotimes_{f\in \mathcal{F}} \mathbb{S}f$. 
It follows that
\begin{align*}
&V_{+} \qty(\bigotimes_{f\in \mathcal{F}} \mathbb{S}f) + 
\qty(\bigotimes_{f\in \mathcal{F}} \mathbb{S}f) V_{-} 
\;\cap\;
\mathrm{Im}(\otimes_{f\in \mathcal{F}} \Tr_{Sf}) \\
&=
\qty(\otimes_{f\in \mathcal{F}} \Tr_{Sf})
\qty(
U_{+}\qty(\bigotimes_{f\in \mathcal{F}} \overline{\Sk}(Sf)) +
\qty(\bigotimes_{f\in \mathcal{F}} \overline{\Sk}(Sf))U_{-}
).
\end{align*}
This means that, if $\alpha \in \bigotimes_{f\in \mathcal{F}} \overline{\Sk}(Sf)$ represents a kernel $\overline{\alpha} \in \overline{\bigotimes}_{f\in \mathcal{F}} \overline{\Sk}(Sf)$ of the map $\overline{\otimes}_{f\in \mathcal{F}} \Tr_{Sf}$ so that
\[
(\otimes_{f\in \mathcal{F}} \Tr_{Sf})(\alpha) 
\;\in\; 
V_{+} \qty(\bigotimes_{f\in \mathcal{F}} \mathbb{S}f) +
\qty(\bigotimes_{f\in \mathcal{F}} \mathbb{S}f) V_{-}, 
\]
then 
\[
(\otimes_{f\in \mathcal{F}} \Tr_{Sf})(\alpha) 
\;\in\; 
\qty(\otimes_{f\in \mathcal{F}} \Tr_{Sf})
\qty(
U_{+}\qty(\bigotimes_{f\in \mathcal{F}} \overline{\Sk}(Sf)) +
\qty(\bigotimes_{f\in \mathcal{F}} \overline{\Sk}(Sf))U_{-}
),
\]
meaning $\overline{\alpha} = 0$. 
Therefore, $\overline{\otimes}_{f\in \mathcal{F}} \Tr_{Sf}$ is injective. 
\end{proof}

We describe the subspaces $V_+$ and $V_-$ explicitly below. 
By convention, we suppress the tensor products defining $\bigotimes_{f\in \mathcal{F}} \mathbb{S}f$. 
That is, we refer to the element $(a_1\otimes b_1) \otimes (a_2 \otimes b_2) \otimes \cdots \otimes (a_k \otimes b_k) \in \bigotimes_{f\in \mathcal{F}} \mathbb{S}f \cong \bigotimes_{f\in \mathcal{F}} \widetilde{\mathbb{T}}^{\otimes 2}$ as $a_1 b_2 a_2 b_2 \cdots a_k b_k$. 
We use the notation $x_{f,e}$ to refer to the generator of $\mathbb{S}f$ corresponding to the edge cone $Ce$. 

\begin{prop}
\begin{enumerate}
    \item \label{item:left_ideal_gen} The subspace $V_-$ of $\otimes_{f\in \mathcal{F}}\widetilde{\mathbb{T}}^{\otimes 2}$ is spanned by the following elements, for every edge $e\in \mathcal{E}$: 
    \[
    [x_{f_1,e_1}x_{f_1,e_2} x_{f_2,e_2}x_{f_2,e_3} \cdots x_{f_k,e_k}x_{f_k,e_1}]^{\epsilon} 
    \;+\;  
    \qty((-1)^{\frac{k}{2}} A^2)^{\epsilon},\quad \epsilon \in \{\pm 1\}, 
    \]
    where we suppose $e$ is surrounded by the face suspensions $Sf_1, Sf_2, \cdots, Sf_k$ and edge cones $Ce_1, Ce_2, \cdots, Ce_k$. 
    \item The subspace $V_+$ of $\otimes_{f\in \mathcal{F}}\widetilde{\mathbb{T}}^{\otimes 2}$ is spanned by the following elements, for every vertex cone $Cv$:
    \begin{enumerate}
        \item \label{item:right_ideal_gen1}
        \[
        [x_{f_1,e_1}x_{f_1,e_2}x_{f_2,e_2}x_{f_2,e_3}x_{f_3,e_3}x_{f_3,e_1}]^{\epsilon} 
        \;+\;
        A^{\epsilon},\quad \epsilon \in \{\pm 1\}, 
        \]
        \item \label{item:right_ideal_gen2}
        \[
        A^{-1}[x_{f_1,e_1}x_{f_1,e_2}x_{f_2,e_2}x^{-1}_{f_2,e_3}] 
        \;+\;
        A[x_{f_1,e_1}x^{-1}_{f_1,e_2}x^{-1}_{f_2,e_2}x_{f_2,e_3}^{-1}]
        \;-\;
        [x_{f_3,e_3}x^{-1}_{f_3,e_1}],
        \]
        where we suppose $Cv$ is surrounded by the face suspensions $Sf_1, Sf_2, Sf_3$ and edge cones $Ce_1, Ce_2, Ce_3$ in clockwise order when viewed from $v$, as in the figure below:
        \[
        \vcenter{\hbox{
        \includestandalone[scale=0.7]{figures/labelings_around_vertex_cone}
        }}.
        \]
    \end{enumerate}
\end{enumerate}
\end{prop}

\begin{proof}
    We compute the images of relations (\ref{item:redSkMod1}), (\ref{item:redSkMod2}), (\ref{item:redSkMod3}), from Corollary \ref{cor:fs_splitting_reduced}.

    \begin{enumerate}
        \item Suppose an internal edge is surrounded by the face suspensions $Sf_{1}, Sf_{2}, \cdots, Sf_ {k}$ and edge cones $Ce_1, Ce_2, \cdots, Ce_k$. Then, 
        \begin{multline*}
            \left (\bigotimes_{i=1}^k \Tr_{Sf} \right) \left(
            \includestandalone[scale=.8,valign=c]{figures/fs_reduced_splitting_relation1}
            \;-\;
            (-A^2)^{\epsilon} \includestandalone[scale=.8,valign=c]{figures/fs_reduced_splitting_relation2}\right) \\
            \;=\; 
            (-1)^{-\frac{k \epsilon}{2}}[x_{f_1,e_1}^{\epsilon}x_{f_1,e_2}^{\epsilon} x_{f_2,e_2}^{\epsilon}x_{f_2,e_3}^{\epsilon} \cdots x_{f_k,e_k}^{\epsilon}x_{f_k,e_1}^{\epsilon}] - (-A^2)^{\epsilon}.
        \end{multline*}
        This gives the relation in (\ref{item:left_ideal_gen}) above. 
            
        \item Suppose a vertex cone $Cv$ is surrounded by the face suspensions $Sf_{1}, Sf_{2}, Sf_{3}$ and edge cones $Ce_1, Ce_2, Ce_3$. Then, 
        \begin{multline*}
            \left (\bigotimes_{i=1}^3 \Tr_{Sf} \right) \left(
            \includestandalone[scale=.8,valign=c]{figures/fs_reduced_splitting_relation3}
            \;-\;
            (-A^2)^{-\epsilon} \includestandalone[scale=.8,valign=c]{figures/fs_reduced_splitting_relation4}\right) \\
            \;=\; 
            A^{-3\epsilon}[x_{f_1,e_1}^{\epsilon}x_{f_1,e_2}^{\epsilon}x_{f_2,e_2}^{\epsilon}x_{f_2,e_3}^{\epsilon} x_{f_3,e_3}^{\epsilon}x_{f_3,e_1}^{\epsilon}]
            -(-A^2)^{-\epsilon}.
        \end{multline*}
        This yields the relation in (\ref{item:right_ideal_gen1}). 
        
        \item Suppose a vertex cone $Cv$ is surrounded by the face suspensions $Sf_{1}, Sf_{2}, Sf_{3}$ and edge cones $Ce_1, Ce_2, Ce_3$ in clockwise order when viewed from $v$. Then, 
        \begin{multline*}
            \left (\bigotimes_{i=1}^3 \Tr_{Sf} \right) \left(
            \includestandalone[scale=.8,valign=c]{figures/fs_reduced_splitting_relation6}
            \;+\;
            \includestandalone[scale=.8,valign=c]{figures/fs_reduced_splitting_relation7}
            \;-\;
            \includestandalone[scale=.8,valign=c]{figures/fs_reduced_splitting_relation5}
            \right) \\
            \;=\; 
            A^{-1}[x_{f_1,e_1}x_{f_1,e_2}x_{f_2,e_2}x^{-1}_{f_2,e_3}] 
            + 
            A[x_{f_1,e_1}x^{-1}_{f_1,e_2} x_{f_2,e_2}^{-1}x_{f_2,e_3}^{-1}]
            - [x_{f_3,e_3}x^{-1}_{f_3,e_1}].
        \end{multline*} 
        This gives the relation in (\ref{item:right_ideal_gen2}).
    \end{enumerate}
\end{proof}

\begin{rmk}
While we have included relations for both signs $\epsilon \in \{\pm 1\}$ in $V_+$ (resp.\ $V_-$), it is enough to consider only one of the relations, as one implies the other in the right (resp.\ left) ideal generated by $V_+$ (resp.\ $V_-$). 
Therefore, from now on, for simplicity, we will drop the relations for $\epsilon = -1$ and, by abuse of notation, still call the corresponding subspaces $V_+$ and $V_-$. 
They still generate the same relations and therefore give the same quotient. 
\end{rmk}

\begin{defn}
    For each bare edge $e\in \mathbf{e}(T)$, 
    the \emph{square root quantized shape parameter} $\hat{x}_{e}$ corresponding to the edge cone $Ce$ is 
    \[
    \hat{x}_{e} = \overline{x_{f_i,e}\otimes x_{f_j,e}},
    \]
    where $x_{f_i,e}$ and $x_{f_j,e}$ are the two generators of $\mathbb{S}f_i$ and $\mathbb{S}f_j$, respectively, that correspond to the edge cone $Ce$, 
    and $\overline{x_{f_i,e}\otimes x_{f_j,e}}$ denotes the image of the element $x_{f_i,e}\otimes x_{f_j,e} \in \bigotimes_{f\in \mathcal{F}} \mathbb{S}f$ in $\overline{\bigotimes}_{f\in \mathcal{F}} \mathbb{S}f$ under the quotient. 
\end{defn}
Note, if the edge cones $Ce_1$ and $Ce_2$ are both adjacent to the same edge cone $Cv$, and $Ce_2$ is clockwise with respect to $Ce_1$ when viewed from $v$, we have 
\[
\hat{x}_{e_1}\hat{x}_{e_2} = A\,\hat{x}_{e_2}\hat{x}_{e_1}.
\]

It is a straightforward exercise to rewrite the generators of $V_+$ and $V_-$ in terms of square root quantized shape parameters. 
\begin{lem} \label{lem:idealRelsUsingSquareRootShapeParams}
\begin{enumerate}
    \item The subspace $V_-$ of $\otimes_{f\in \mathcal{F}}\widetilde{\mathbb{T}}^{\otimes 2}$ is spanned by the following elements, one for each $e\in \mathcal{E}$:
    \[
    [\hat{x}_{e_1}\hat{x}_{e_2}\cdots \hat{x}_{e_k}]
    \;+\;
    (-1)^{\frac{k}{2}} A^2,  
     \]
    where $\hat{x}_{e_1}, \hat{x}_{e_2},\cdots, \hat{x}_{e_{k-1}}$, and $\hat{x}_{e_k}$ are the $k$ square root quantized shape parameters associated to the $k$ edge cones abutting $e$. 
    
    \item The subspace $V_+$ of $\otimes_{f\in \mathcal{F}}\widetilde{\mathbb{T}}^{\otimes 2}$ is spanned by the following elements, two for each vertex cone $Cv$:
    \begin{align*}
        &[\hat{x}_{e_1}\hat{x}_{e_2}\hat{x}_{e_3}] 
        \;+\; 
        A, \\
        &\hat{x}^{-2}_{e_1} \;+\; \hat{x}^{2}_{e_2} 
        \;+\; 
        1,\\ 
    \end{align*}
    where $\hat{x}_{e_1}, \hat{x}_{e_2}, \hat{x}_{e_3}$ are the square root quantized shape parameters associated to the three edge cones $Ce_1, Ce_2, Ce_3$ abutting $Cv$ in clockwise order when viewed from $v$. 
\end{enumerate}
\end{lem}

\begin{rmk}
    The elements 
    \[
    [\hat{x}_{e_1}\hat{x}_{e_2}\hat{x}_{e_3}]
    \;+\; 
    A ,
    \]
    where $e_1$, $e_2,$ and $e_3$ are three distinct edge cones abutting the same vertex cone, 
    are central. 
    Using the $A$-commutation relations among square root quantum shape parameters, it is an easy check that 
    \[  [\hat{x}_{e_1}\hat{x}_{e_2}\hat{x}_{e_3}]\hat{x}_{e_4} = \hat{x}_{e_4} [\hat{x}_{e_1}\hat{x}_{e_2}\hat{x}_{e_3}],
    \]
    where $e_4$ is any edge cone.
\end{rmk}

\begin{defn} \label{defn:sqgm}
    The \emph{square root quantum gluing module} $\SQGM_{\mathcal{T}}(Y)$ is the $R$-submodule of $\overline{\bigotimes}_{f\in \mathcal{F}} \mathbb{S}f$ spanned by monomials in $\hat{x}_{e}$, for all $T\in \mathcal{T}$ and $e \in \textbf{e}(T)$.
\end{defn}

To prove our main theorem, we need one last lemma.

\begin{lem}
The image of the composition 
\[
\Tr_{\mathcal{T}} := \overline{\otimes} \Tr_{Sf} \,\circ\, \overline{\sigma}
: \overline{\Sk}(Y) 
\overset{\overline{\sigma}}{\rightarrow} 
\overline{\bigotimes}_{f\in \mathcal{F}} \overline{\Sk}(Sf) 
\overset{\overline{\otimes} \Tr_{Sf}}{\rightarrow} 
\overline{\bigotimes}_{f\in \mathcal{F}}\mathbb{S}f
\]
lies in $\SQGM_{\mathcal{T}}(Y)$. 
\end{lem}

\begin{proof}
    Starting with a link $[L] \in \overline{\Sk}(Y)$, we apply the splitting map $\overline{\sigma}$ (Proposition \ref{prop:splitting_reduced}) to obtain 
    \[
    \overline{\sigma}([L]) = 
    \qty[\sum_{\vec{\epsilon} \in \{\substack{\mathrm{compatible}\\ \mathrm{states}}\}}\otimes_{f\in \mathcal{F}} [L_{f}^{\vec{\epsilon}}]] \in  \overline{\bigotimes}_{f\in \mathcal{F}}\overline{\Sk}(Sf).
    \]
    It is enough to check each term of this sum lands in $\SQGM_{\mathcal{T}}(Y)$ under $\overline{\otimes} \Tr_{Sf}$, so fix some compatible state $\vec{s}$. 

    Recall from Remark $\ref{rmk:SkeinModuleIsGraded}$ that $\overline{\Sk}(Sf)$ is naturally $\mathbb{Z}^6$-graded, where the grading is given by the sum of states on each edge of $\Gamma$. 
    Suppose two face suspensions $Sf_1$ and $Sf_2$ are glued along an edge cone $Ce$, and that $\Gamma_1$ and $\Gamma_2$ are the boundary markings of $Sf_1$ and $Sf_2$, respectively.
    Since $\vec{s}$ is a compatible state, the sum of the states of $[L_{f_1}^{\vec{s}}]$ on $\Gamma_1\cap Ce$ and the sum of the states of $[L_{f_2}^{\vec{s}}]$ on $\Gamma_2\cap Ce$ are equal; 
    see Remark \ref{rmk:NewGradingAfterSplitting}. 

    $\Tr_{Sf}$ respects the grading on $\overline{\Sk}(Sf)$. 
    Thus, $\Tr_{Sf_1}([L_{f_1}^{\vec{s}}])$ and $\Tr_{Sf_2}([L_{f_2}^{\vec{s}}])$ will yield monomials $P_1$ and $P_2$, respectively, such that $x_{f_1,e}$ appears in $P_1$ with the same degree that $x_{f_2,e}$ appears in $P_2$. 
    This concludes the proof.       
\end{proof}

Our main theorem now follows immediately. 

\begin{thm}\label{thm:quantumTrace}
There is a $R$-module homomorphism $\Tr_{\mathcal{T}} : \Sk(Y) \rightarrow \SQGM_{\mathcal{T}}(Y)$ defined as the composition 
\[
\Sk(Y) = \overline{\Sk}(Y) \overset{\overline{\sigma}}{\rightarrow} \overline{\bigotimes}_{f\in \mathcal{F}} \overline{\Sk}(Sf) \overset{\overline{\bigotimes}_{f\in \mathcal{F}} \Tr_{Sf}}{\rightarrow} \overline{\bigotimes}_{f\in \mathcal{F}}\mathbb{S}f. 
\]
\end{thm}

\subsection{Quantum gluing module}\label{subsec:qgm}

In this subsection, we explain in what sense the square root quantum gluing module $\SQGM_{\mathcal{T}}(Y)$ can be thought of as an extension of the quantum gluing module. 

The following definition is from \cite{Dim} and \cite{AGLR}.

\begin{defn}
    Let
    \[
    \mathrm{QT}_{\mathcal{T}}(Y) := \bigotimes_{T\in \mathcal{T}} \frac{R\langle \hat{Z}_T^{\pm 1}, \hat{Z}_T'^{\pm 1}, \hat{Z}_T''^{\pm 1}\rangle}{\langle \hat{Z}_T \hat{Z}'_T = A^{4} \hat{Z}'_T \hat{Z}_T, 
    \hat{Z}'_T \hat{Z}''_T = A^{4} \hat{Z}''_T \hat{Z}'_T, 
    \hat{Z}''_T \hat{Z}_T = A^{4} \hat{Z}_T \hat{Z}''_T\rangle}
    \]
    be a quantum torus in $3n$ generators, where $n$ is the number of tetrahedra in $\mathcal{T}$. 

    The generators $\hat{Z}_i, \hat{Z}_i', \hat{Z}_i''$ are called \emph{quantized shape parameters}, as
    they are naturally associated to bare edges of tetrahedra in the following way:
    if the bare edge $e$ of a tetrahedron $T$ is labeled by the classical shape parameter $Z_T, Z'_T$, or $Z''_T$, then we associate $\hat{Z}_T,\hat{Z}'_T$, or $\hat{Z}''_T$, respectively, to $e$.
   
    The \emph{quantum gluing module} $\QGM_{\mathcal{T}}(Y)$ is defined as 
    \[
    \QGM_{\mathcal{T}}(Y) := \frac{\mathrm{QT}_{\mathcal{T}}(Y)}{W_+ \mathrm{QT}_{\mathcal{T}}(Y) +
    \mathrm{QT}_{\mathcal{T}}(Y) W_-},
    \]
    where $W_+$ and $W_-$ are the $R$-submodules of $\mathrm{QT}_{\mathcal{T}}(Y)$ defined as follows:
    \begin{enumerate}
        \item $W_-$ is spanned by the following elements, one for each edge $e\in \mathcal{E}$:
        \[
        [\hat{Y}_{1} \hat{Y}_{2} \cdots \hat{Y}_{k}] 
        \;-\; 
        A^4 ,
        \]
        where the $\hat{Y}_{i}\in \{\hat{Z}_T, \hat{Z}'_T, \hat{Z}''_T \;\vert\; T \in \mathcal{T} \}$ are the $k$ quantized shape parameters associated to the $k$ bare edges glued around $e$. 
        \item $W_+$ is spanned by the following elements, two for each tetrahedron $T\in \mathcal{T}$:
        \begin{enumerate}
            \item 
            \[
            [\hat{Z}_T \hat{Z}'_T \hat{Z}''_T] \;+\; A^2 ,
            \]
            \item 
            \[
            \hat{Z}''^{-1}_T + \hat{Z}_T \;-\; 1. 
            \]
        \end{enumerate}
    \end{enumerate}
\end{defn}

In the rest of this section, we show that there is a natural homomorphism
\[
\iota : \QGM_{\mathcal{T}}(Y) \rightarrow \SQGM_{\mathcal{T}}(Y)
\]
which we conjecture to be injective. 

\begin{lem} \label{lem:squaredParams}
    Suppose $\hat{x}_{e_i}$ and $\hat{x}_{e_j}$ are the square root quantized shape parameters corresponding to opposite edges $e_i$ and $e_j$ in a tetrahedron $T$. 
    Then, we must have 
    \[
    \hat{x}_{e_i}^2 = \hat{x}_{e_j}^2. 
    \]
\end{lem}
\begin{proof}
    Using the labeling of edges as shown in Figure \ref{fig:edgeLabels}, 
    we obtain the following relations:
    \begin{align}                   [\hat{x}_{e_1}\hat{x}_{e_2}\hat{x}_{e_3}] &= -A \label{eqn:lem1},\\ 
    [\hat{x}_{e_1}\hat{x}_{e_5}\hat{x}_{e_6}] &= -A \label{eqn:lem2},\\
    [\hat{x}_{e_4}\hat{x}_{e_2}\hat{x}_{e_6}] &= -A \label{eqn:lem3},\\
    [\hat{x}_{e_4}\hat{x}_{e_5}\hat{x}_{e_3}] &= -A. \label{eqn:lem4}
    \end{align}
    \begin{figure}[H]
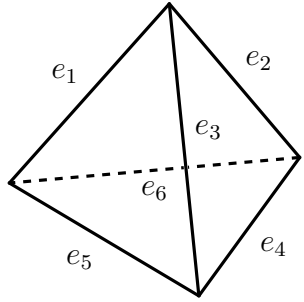

    \centering
    \raisebox{6mm}{\includestandalone[scale=1]{figures/edge_labels}}
    \caption{The labeling of the edges of a tetrahedron we assume for Lemma \ref{lem:squaredParams}.}
    \label{fig:edgeLabels}
    \end{figure}
    From (\ref{eqn:lem1}) and (\ref{eqn:lem2}), we see 
    \begin{equation*}
        \hat{x}_{e_2}\hat{x}_{e_3} = \hat{x}_{e_5}\hat{x}_{e_6},
    \end{equation*}
    and from (\ref{eqn:lem3}) and (\ref{eqn:lem4}), we obtain 
    \begin{equation*}
        \hat{x}_{e_2}\hat{x}_{e_6} = \hat{x}_{e_5}\hat{x}_{e_3}.
    \end{equation*}
    Combining these last two equations, we find 
    \[\hat{x}^2_{e_3} = \hat{x}^2_{e_6}\]
    and 
    \[\hat{x}^2_{e_2} = \hat{x}^2_{e_5}.\]
    We can get the last desired relation similarly.   
\end{proof}

\begin{defn} \label{defn:qgm}
    For each tetrahedron $T \in \mathcal{T}$, define 
    \begin{align*}
        \hat{X}_{T} &\coloneqq -\hat{x}^2_{Z_T},\\
        \hat{X}'_{T} &\coloneqq -\hat{x}^2_{Z_T'},\\
        \hat{X}''_{T} &\coloneqq -\hat{x}^2_{Z_T''},
    \end{align*}
    where each $\hat{x}_{Z_T}, \hat{x}_{Z'_T}, \hat{x}_{Z''_T}$ is a square root quantized shape parameter associated to an edge with classical shape parameter $Z_T, Z'_T$, and $Z''_T$, respectively.
\end{defn} 
Notice that, for any $T_i, T_j \in \mathcal{T}$, these elements satisfy the relations
\begin{align*}
    \hat{X}_{T_i}\hat{X}_{T_j}' &= A^{4\delta_{i,j}}\hat{X}'_{T_j}\hat{X}_{T_i},\\
    \hat{X}'_{T_i}\hat{X}_{T_j}'' &= A^{4\delta_{i,j}}\hat{X}''_{T_j}\hat{X}'_{T_i},\\
    \hat{X}''_{T_i}\hat{X}_{T_j} &=  A^{4\delta_{i,j}}\hat{X}_{T_j}\hat{X}''_{T_i},
\end{align*}
which are exactly the relations for the generators of the quantum torus $\mathrm{QT}_{\mathcal{T}}(Y)$. 

\begin{rmk}
Notice the minus sign in the definition of our quantized shape parameters. 
This is expected from the classical story in Section \ref{subsec:3dClassicalTrace}; our edge cone crossing matrices used square roots of negative shape parameters. 
\end{rmk}

The following lemma is a simple computation, using the presentation of $V_-$ and $V_+$ obtained in Lemma \ref{lem:idealRelsUsingSquareRootShapeParams}.

\begin{lem} \label{lem:containsSquaredRels}
    \begin{enumerate}
    \item The subspace $V_-$ contains the following elements, for every edge $e\in \mathcal{E}$: 
    \[
    [\hat{Y}_{1}\hat{Y}_{2} \cdots \hat{Y}_{k}] - A^4 , 
    \]
    where the $\hat{Y}_{j}\in \{\hat{X}_{T}, \hat{X}'_{T}, \hat{X}''_{T} \;\vert\; T \in \mathcal{T}\}$ are the $k$ quantized shape parameters associated to the bare edges of tetrahedra glued around $e$.
    \item The subspace $V_+$ contains the following elements, for each $T\in \mathcal{T}$: 
    \begin{align*}
         &[\hat{X}_{T} \hat{X}'_{T} \hat{X}''_{T}] + A^2 ,\\
         &\hat{X}_{T}''^{-1} + \hat{X}_{T} - 1.
    \end{align*}
\end{enumerate}
\end{lem}

From Lemmas \ref{lem:squaredParams} and \ref{lem:containsSquaredRels}, it immediately follows that:
\begin{thm}
There is a natural $R$-module homomorphism
$\iota: \mathrm{QGM}_{\mathcal{T}}(Y) \rightarrow \SQGM_{\mathcal{T}}(Y)$
defined on monomials by
\[
\prod_{T \in \mathcal{T}} \hat{Z}^{m_{T}}_{T} \hat{Z}'^{m_{T}'}_{T} \hat{Z}''^{m_{T}''}_{T} \mapsto 
\prod_{T \in \mathcal{T}} \hat{X}^{m_{T}}_{T} \hat{X}'^{m_{T}'}_{T} \hat{X}''^{m_{T}''}_{T}.
\]
\end{thm}

We conjecture that the map $\iota: \mathrm{QGM}_{\mathcal{T}}(Y) \rightarrow \SQGM_{\mathcal{T}}(Y)$ is an embedding
; see Conjecture \ref{conj:qgmIsSubMod}. 
We have verified the claim classically for the figure-$8$ knot complement by computing elimination ideals in Mathematica.


\section{Explicit examples} \label{sec:examples}

In this section, we detail the process of how one can actually compute our map in practice, and then use this procedure to compute the quantum trace of some links in the figure-8 knot complement $Y = S^3 \setminus 4_1$.

\subsection{Computing the 3d quantum trace in practice}

Given a link $[L]$ in an ideally triangulated $3$-manifold, one follows the following procedure:
\begin{enumerate}
    \item On each edge cone $Ce$, choose an embedded interval $I_{e}$ oriented from the midpoint of $e$ to the barycenter of the tetrahedron containing $Ce$. 
    Isotope $L$ so that at each intersection $p$ of $L$ with an edge cone $Ce$, $p \in I_{e}$, and the framing of $L$ at $p$ is tangent to $I_{e}$.
    \item Use the splitting map $\overline{\sigma}$ to cut $\overline{\Sk}(Y)$ into the reduced tensor product of reduced stated skein modules of face suspensions. 
    \item In each face suspension, apply the face suspension quantum trace from Theorem \ref{thm:fsQT}. 
\end{enumerate}

We show examples of bad arcs in a face suspension in Figure \ref{fig:fsBadArcs}.

For convenience, we repeat the definition of $\Tr_{Sf}$ from Theorem \ref{thm:fsQT} with a bit more detail. 

\begin{figure}
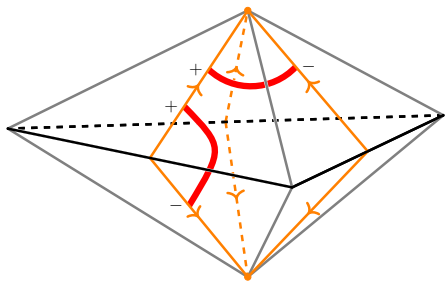

    \centering
    \includestandalone[scale=.8]{figures/fs_bad_arcs}
    \caption{Examples of bad arcs in a face suspension.}
    \label{fig:fsBadArcs}
\end{figure}

Recall that $\overline{\Sk}(Sf)$ is a $\mathbb{T}^{\otimes 2}$-$\mathbb{B}^{\otimes 3}$-bimodule, spanned by the empty skein. 
If $\alpha \in \mathbb{T}$ consists of a single ribbon tangle with the canonical framing connecting the edge cones corresponding to $x_1, x_2 \in \mathbb{S}f$ with states $\mu$ and $\nu$, respectively, then define 
\[
\mathrm{ev}(\alpha) := A^{-\frac{\mu+\nu}{2}}[x^{\mu}_1x^{\nu}_2].
\] 
If $\beta \in \mathbb{B}$ consists of a single ribbon tangle with the canonical framing connecting the edge cones corresponding to $x_1, x_2 \in \mathbb{S}f$ both with states $\mu$, then define 
\[
\mathrm{ev}( \beta ) := (-1)^{-\frac{\mu}{2}}x^{\mu}_1x^{\mu}_2.
\]
These are nothing but the embeddings defined in Lemma \ref{lem:triangle_biangle_embedding}. 

Given $[L] \in \overline{\Sk}(Sf)$, use Lemma \ref{lem:Sk(B)isCyclic} to write $[L] = \alpha_1 \cdots \alpha_k \cdot [\emptyset] \cdot \beta_1 \cdots \beta_l,$
where each $\alpha_i \in \mathbb{T}$ and $\beta_i \in \mathbb{B}$ and each $\alpha_i$ and $\beta_i$ consists of a single tangle.
Then,
\[
\Tr_{Sf}([L]) = \mathrm{ev}(\alpha_1) \cdots \mathrm{ev}(\alpha_k) \mathrm{ev}(\beta_1) \cdots \mathrm{ev}(\beta_l) \in \mathbb{S}f. 
\]

Suppose a face suspension is associated to the face $f$ and spans parts of tetrahedra $T_1$ and $T_2$. 
Label the edge cones from $T_1$ with $z, z', z''$ clockwise when viewed from the barycenter of $T_1$, 
and likewise label the edge cones from $T_2$ with $y, y', y''$ clockwise when viewed from the barycenter of $T_2$. 
An example of such a labeling is shown in Figure \ref{fig:FSM_labeling}. 

\begin{figure}[htbp]
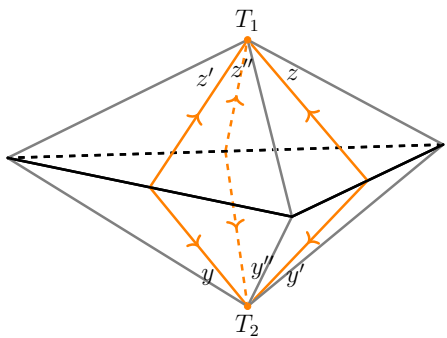

    \centering
    \includestandalone[scale=.8]{figures/fsm_labeling}
    \caption{A labeling of the edge cones of a face suspension used in obtaining the presentation given in this section.}
    \label{fig:FSM_labeling}
\end{figure}

Then 
\[
\mathbb{S}f \cong \frac{R \langle z^{\pm 1},z'^{\pm 1},z''^{\pm1} \rangle \otimes R \langle y^{\pm 1},y'^{\pm 1},y''^{\pm1}\rangle}{\langle zz' = Az'z, z'z'' = Az''z', z''z = Azz'',yy' = Ay'y, y'y'' = Ay''y', y''y = Ayy'' \rangle}.
\]

\subsection{Examples}

In Figure \ref{fig:fig8ComplementTriangulation}, we show a triangulation of the figure-8 knot complement along with two ribbon tangles $K_m$ and $K_b$. 
In that figure, we have also labeled the bare edges (and thus edge cones) with square root quantized shape parameters. 
In Figure \ref{fig:fig8_r3}, we show these same tangles and the figure-$8$ knot in $S^3$.

\begin{figure}
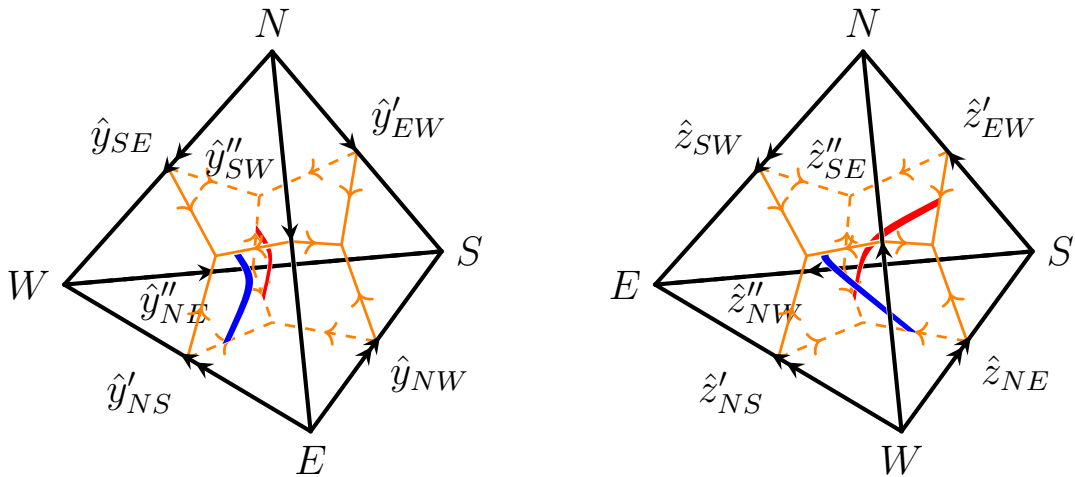

    \centering
    \includestandalone[scale=1.3]{figures/figure8_tet1}
    \hspace{.5cm}
    \includestandalone[scale=1.3]{figures/figure8_tet2}
    \caption{A triangulation of the figure-8 knot complement. 
    Vertex labels index the opposite face. 
    The gluing of the two tetrahedra is determined by gluing the appropriate edges based on the number of solid black arrows and their orientation. 
    The ribbon tangle $K_m$ is shown in red and the ribbon tangle $K_b$ is shown in blue.}
    \label{fig:fig8ComplementTriangulation}
\end{figure}

\begin{figure}[htbp]
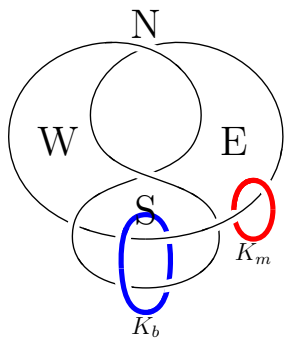

    \centering
    \includestandalone[scale=1.3]{figures/figure_8_knot}
    \caption{We show the figure-8 knot in $S^3$, along with unframed representatives of $K_b$ in blue and $K_m$ in red. 
    The faces of the triangulation can be identified with regions in the plane containing the figure-8 knot, and are labeled appropriately.}
    \label{fig:fig8_r3}
\end{figure}

\subsubsection{Quantum trace of $K_m^2$}
For our first example, we compute the quantum trace of a two cabling of $K_m$, denoted $K_m^2$. 
In Figure \ref{fig:km_in_face_suspensions} we show the way in which one representative of $K_m^2$ splits when we cut $Y$ into face suspensions. 
For simplicity, we only show the two face suspensions that contain components of $K_m^2$. 
Recall that the quantum trace is a sum over all possible states $\epsilon_1, \epsilon_2, \epsilon_3,$ and $\epsilon_4$. 

In the following computations, we will label the generators of the two copies of $\mathbb{S}f$ the shape parameter of the corresponding edge cone, indexed by the face. 
For example, the generator of $Sf_N$ associated to the edge cone corresponding to $Z$ will be referred to as $z_N$.

\begin{figure}[H]
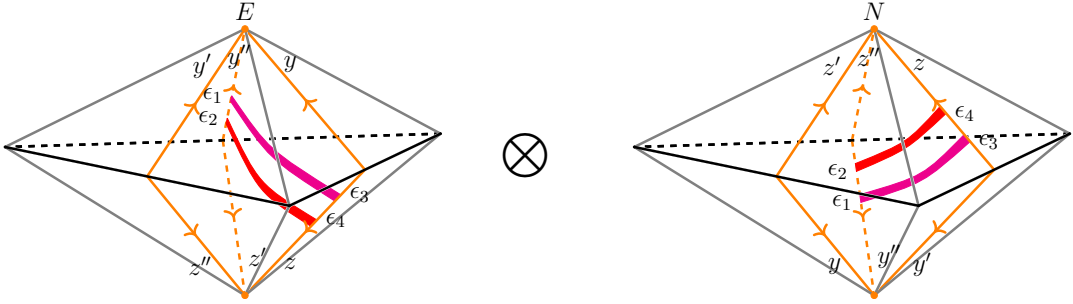

    \centering
    \begin{minipage}{.45\textwidth}  
        \centering
        \includestandalone[scale=.8]{figures/km_first_fs}
    \end{minipage}
    \begin{minipage}{.04\textwidth}
        \centering
        \begin{equation*}
            \bigotimes
        \end{equation*}
    \end{minipage}
    \begin{minipage}{.45\textwidth}
        \centering
        \includestandalone[scale=.8]{figures/km_second_fs}
    \end{minipage}
    \caption{$K^2_m$ after cutting $Y$ into face suspensions. 
    The two face suspensions correspond to the faces $E$ and $N$. } 
    \label{fig:km_in_face_suspensions}
\end{figure}

We use the skein relations to obtain Figure \ref{fig:km_in_face_suspensions_reduced}. 
There are $4$ compatible states that give non-zero elements: 

\begin{figure}[H]
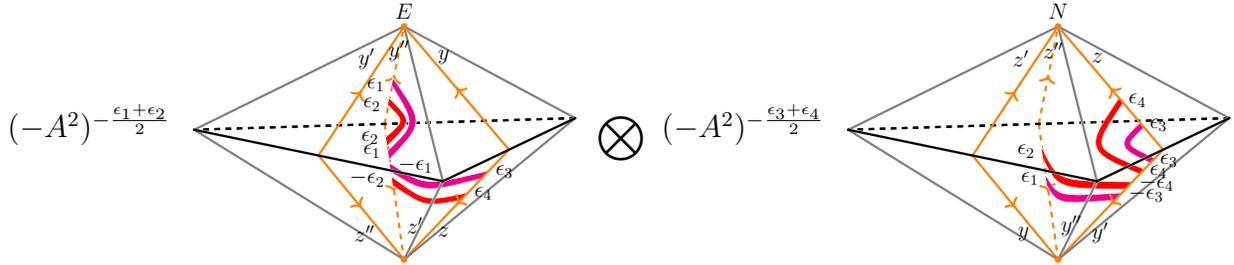

    \centering
    \begin{minipage}{.47\textwidth}  
        \centering
        $(-A^2)^{-\frac{\epsilon_1+\epsilon_2}{2}}$\includestandalone[scale=.7,valign=c]{figures/km_first_fs_reduced}
    \end{minipage}
    \begin{minipage}{.04\textwidth}
        \centering
        \begin{equation*}
            \bigotimes
        \end{equation*}
    \end{minipage}
    \begin{minipage}{.47\textwidth}
        \centering
        $(-A^2)^{-\frac{\epsilon_3+\epsilon_4}{2}}$\includestandalone[scale=.7,valign=c]{figures/km_second_fs_reduced}
    \end{minipage}
    \caption{$K^2_m$ after cutting $Y$ into face suspensions and using skein relations to identify each diagram as the action of some $\alpha \in \mathbb{T}^{\otimes 2} \otimes \mathbb{B}^{\otimes 3}$ on the empty skein.} \label{fig:km_in_face_suspensions_reduced}
\end{figure}

\begin{enumerate}
    \item $\epsilon_1=\epsilon_2=+$ and $\epsilon_3=\epsilon_4=-$: 
    Notice the coefficients in Figure \ref{fig:km_in_face_suspensions_reduced} cancel. 
    We obtain
    \[
    [z_E'^{-1} z_E^{-1}] [z_E'^{-1} z_E^{-1}] y_E'' z_E' y_E'' z_E' \otimes [y_N' y_N''][y_N' y_N''] z_N^{-1} y_N'^{-1} z_N^{-1} y_N'^{-1} 
    = y_E''^{2} z_E^{-2} \otimes y_N''^{2} z_N^{-2} 
    = \hat{y}''^{2}_{NE} \hat{z}_{NE}^{-2}
    = \hat{Y}'' \hat{Z}^{-1}.
    \]
    \item $\epsilon_1 = \epsilon_4 = +$ and $\epsilon_2 = \epsilon_3 = -$: 
    We find 
    \[
    [z_E' z_E] [z_E'^{-1} z_E^{-1}] y_E'' z_E' y_E''^{-1} z_E'^{-1} \otimes [y_N' y_N''] [y_N'^{-1} y_N''^{-1}] z_N^{-1} y_N'^{-1} z_N y_N' 
    = 1 \otimes 1.
    \]
    \item $\epsilon_1 = \epsilon_2 = -$ and $\epsilon_3 = \epsilon_4 = +$: 
    $\hat{Y}''^{-1} \hat{Z}$.
    \item $\epsilon_1 = \epsilon_4 = -$ and $\epsilon_2 = \epsilon_3 = +$: 
    $1 \otimes 1$. 
\end{enumerate}
Summing them up, we get
\[
\Tr_{\mathcal{T}}(K_m^2) = 2 + \hat{Y}'' \hat{Z}^{-1} + \hat{Y}''^{-1} \hat{Z} \;\in \iota(\QGM_{\mathcal{T}}(Y)) \subset \SQGM_{\mathcal{T}}(Y). 
\]

\subsubsection{Quantum trace of $K_b$}
We will now compute the quantum trace on $K_b$. 

In Figure \ref{fig:kb_in_face_suspensions}, we draw one representative of $K_b$ after cutting into face suspensions. 
In Figure \ref{fig:kb_in_face_suspensions_reduced}, we have used the skein relations to rewrite each tangle.

\begin{figure}[H]
    \centering
    \begin{minipage}{.45\textwidth}  
        \centering
        \includestandalone[scale=.8]{figures/kb_first_fs}
    \end{minipage}
    \begin{minipage}{.04\textwidth}
        \centering
        \begin{equation*}
            \bigotimes
        \end{equation*}
    \end{minipage}
    \begin{minipage}{.45\textwidth}
        \centering
        \includestandalone[scale=.8]{figures/kb_second_fs}
    \end{minipage}
    \caption{$K_b$ after cutting $Y$ into face suspensions. 
    The two face suspensions correspond to the faces $S$ and $N$. } 
    \label{fig:kb_in_face_suspensions}
\end{figure}

\begin{figure}[H]
    \centering
    \begin{minipage}{.47\textwidth}  
        \centering
        $(-A^2)^{-\frac{\epsilon_1}{2}}$\includestandalone[scale=.7,valign=c]{figures/kb_first_fs_reduced}
    \end{minipage}
    \begin{minipage}{.04\textwidth}
        \centering
        \begin{equation*}
            \bigotimes
        \end{equation*}
    \end{minipage}
    \begin{minipage}{.47\textwidth}
        \centering
        $(-A^2)^{-\frac{\epsilon_2}{2}}$\includestandalone[scale=.7,valign=c]{figures/kb_second_fs_reduced}
    \end{minipage}
    \caption{$K_b$ after cutting $Y$ into face suspensions and using skein relations to identify each diagram as the action of some $\alpha \in \mathbb{T}^{\otimes 2} \otimes \mathbb{B}^{\otimes 3}$ on the empty skein. } \label{fig:kb_in_face_suspensions_reduced}
\end{figure}

There are $3$ non-zero compatible states:
\begin{enumerate}
    \item $\epsilon_1 = +$ and $\epsilon_2 = -$: We calculate
    \[
    [z_S^{-1} z_S'^{-1}] y_S' z_S \otimes [y_N y_N'] z_N'^{-1} y_N^{-1}
    = A y_S' z_S'^{-1} \otimes z_N'^{-1} y_N' 
    = A \hat{z}_{NS}'^{-1} \hat{y}_{NS}'.
    \]
    \item $\epsilon_1 = -$ and $\epsilon_2 = +$:  
    \[
    [z_S z_S'] y_S'^{-1} z_S^{-1} \otimes [y_N^{-1} y_N'^{-1}] z_N' y_N 
    = A y_S'^{-1} z_S' \otimes z_N' y_N'^{-1} 
    = A \hat{y}_{NS}'^{-1} \hat{z}_{NS}'.
    \]
    \item $\epsilon_1 = \epsilon_2 = -$: 
    \[
    A^{2} [z_S z_S'^{-1}] y_S'^{-1} z_S^{-1} \otimes [y_N y_N'^{-1}] z_N'^{-1} y_N^{-1} 
    = A y_S'^{-1} z_S'^{-1} \otimes z_N'^{-1} y_N'^{-1} 
    = A \hat{z}_{NS}'^{-1} \hat{y}_{NS}'^{-1}.
    \]
\end{enumerate}
We claim the image of $K_b$ under our trace map lives in the submodule $\iota(\QGM_{\mathcal{T}}(Y))$ of $\SQGM_{\mathcal{T}}(Y)$. 
Some relations in $\SQGM_{\mathcal{T}}(Y)$ important for our purposes are:
\begin{enumerate}
    \item We have the following identities among the left actions 
    \begin{align*}
        1 + \hat{z}_{NS}'^{-2} &= -\hat{z}_{NW}''^{2}, \\
        [\hat{z}_{SW} \hat{z}_{NS}' \hat{z}_{NW}''] &= -A,\\
        [\hat{y}_{SE} \hat{y}_{NS}' \hat{y}_{NE}''] &= -A,\\
        [\hat{z}_{SW} \hat{z}_{EW}' \hat{z}_{SE}''] &= -A,\\
        [\hat{y}_{SE} \hat{y}_{EW}' \hat{y}_{SW}''] &= -A.
    \end{align*}
    The last four identities are central elements, so the same holds for right actions. 
    \item We have the following identity among the right actions 
    \[
    [\hat{y}''_{NE} \hat{y}''_{SW} \hat{y}'_{EW} \hat{z}''_{NW} \hat{z}''_{SE} \hat{z}'_{EW}]
    = A^2.
    \]
    From above, this is equivalent to identifying the right actions 
    \[
    \hat{y}_{NE}'' \hat{y}_{SE}^{-1} \hat{z}_{NW}'' \hat{z}_{SW}^{-1} 
    = A^{-1}.
    \]
\end{enumerate}
Then
\begin{align*}
\Tr_{\mathcal{T}}(K_b) &= A \hat{z}_{NS}'^{-1} \hat{y}_{NS}' + A \hat{y}_{NS}'^{-1} \hat{z}_{NS}' + A \hat{z}_{NS}'^{-1} \hat{y}_{NS}'^{-1}\\ 
&= A \hat{z}_{NS}'^{-1} \hat{y}_{NS}' + A (1 + \hat{z}_{NS}'^{-2}) \hat{z}_{NS}' \hat{y}_{NS}'^{-1}\\
&= A \hat{z}_{NS}'^{-1} \hat{y}_{NS}' - A \hat{z}_{NW}''^2 \hat{z}_{NS}' \hat{y}_{NS}'^{-1}\\ 
&= \hat{z}_{NW}'' \hat{z}_{SW} \hat{y}_{NE}''^{-1} \hat{y}_{SE}^{-1} - \hat{z}_{NW}'' \hat{z}_{SW}^{-1} \hat{y}_{NE}'' \hat{y}_{SE}\\ 
&= A \hat{z}_{SW}^{2} \hat{y}_{NE}''^{-2} - A \hat{y}_{SE}^{2}\\
&= A (\hat{Z} \hat{Y}''^{-1} + \hat{Y}),
\end{align*}
as desired.

In the conventions of \cite{AGLR}, $A = q_{AGLR}^{-1/4}$ and our shape parameters go to their inverses. 
After changing the framing by $-1$ (multiplying by $-A^{-3}$) and taking these differences in conventions into account, our result matches the conjecture given in that paper. 

\subsubsection{Quantum trace of $K^2_b$}

Finally, we compute the quantum trace on a $2$-cabling of $K_b$, shown in Figure \ref{fig:kb2_in_face_suspensions}.

\begin{figure}[H]
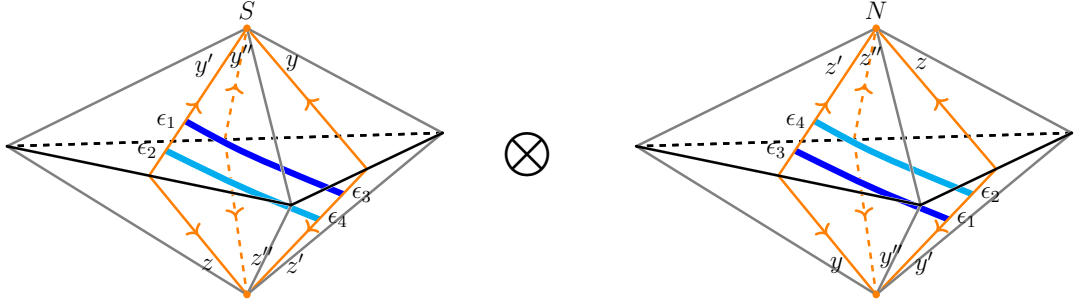

    \centering
    \begin{minipage}{.45\textwidth}  
        \centering
        \includestandalone[scale=.8]{figures/kb2_first_fs}
    \end{minipage}
    \begin{minipage}{.04\textwidth}
        \centering
        \begin{equation*}
            \bigotimes
        \end{equation*}
    \end{minipage}
    \begin{minipage}{.45\textwidth}
        \centering
        \includestandalone[scale=.8]{figures/kb2_second_fs}
    \end{minipage}
    \caption{$K^2_b$ after cutting $Y$ into face suspensions.} 
    \label{fig:kb2_in_face_suspensions}
\end{figure}

One can check there are $9$ compatible states giving non-zero elements, 
and we find 
\begin{align*}
    \Tr_{\mathcal{T}}(K^2_b) 
    &= A^4 \hat{z}'^{-2}_{NS} \hat{y}'^2_{NS} + A^4\hat{z}'^2_{NS} \hat{y}'^{-2}_{NS} + \hat{z}'^{-2}_{NS} \hat{y}'^{-2}_{NS} + (A^4 + 1) \hat{y}'^{-2}_{NS} + (A^4 + 1) \hat{z}'^{-2}_{NS} + 2\\
    &= A^4 \hat{Z}'^{-1} \hat{Y}' + A^4 (\hat{Z}' + A^{-4} \hat{Z}'^{-1} - 1 - A^{-4}) \hat{Y}'^{-1} + (A^4 + 1) \hat{Z}'' + 1 - A^4 \\
    &= A^4 \hat{Z}'^{-1} \hat{Y}' + A^4 \hat{Z}''^2 \hat{Z}' \hat{Y}'^{-1} + (A^4 + 1) \hat{Z}'' + 1 - A^4 \\
    &= A^4 (\hat{Z} \hat{Y}''^{-1} + \hat{Y})^2 + 1 - A^4.
\end{align*}

Surprisingly, even taking into account differences in conventions and possible differences in framing, this result \emph{does not} match the conjecture given in \cite{AGLR}.


\section{Future Directions} \label{sec:future_directions}
We conclude with some natural conjectures, as well as avenues for future research. 

\begin{conj}
The splitting maps $\sigma$ and $\overline{\sigma}$ are injective. 
\end{conj}

In particular, injectivity of $\overline{\sigma}$, together with Proposition \ref{prop:TensorProductOfQuantumTraceIsInjective} would imply: 
\begin{conj}
The 3d quantum trace map $\Tr_\mathcal{T}$ is injective. 
\end{conj}

As mentioned in Section \ref{subsec:qgm}, we also conjecture that
\begin{conj}\label{conj:qgmIsSubMod}
The natural homomorphism $\iota : \QGM_{\mathcal{T}}(Y) \rightarrow \SQGM_{\mathcal{T}}(Y)$ is injective. 
\end{conj}

For our next conjecture, assume that $Y = S^3 \setminus \mathcal{K}$ is an ideally triangulated knot complement. 
We say that a link $L \subset S^3 \setminus \mathcal{K}$ is \emph{even} if it represents the zero homology class in $H_1(S^3 \setminus \mathcal{K}, \mathbb{Z}_2) \cong \mathbb{Z}_2$. 
\begin{conj}\label{conj:EvenLink}
    The image of an even link under $\Tr_\mathcal{T}$ is contained in $\iota (\QGM_{\mathcal{T}}(Y)) \subset \SQGM_{\mathcal{T}}(Y)$.
\end{conj}
We have seen some evidence of Conjecture \ref{conj:EvenLink} in Section \ref{sec:examples}, as all of $K_m^2, K_b,$ and $K_b^2$ are even links. 

Beyond addressing these conjectures, here are some other promising future directions:
\begin{enumerate}
    \item How does the 3d quantum trace change as we vary the triangulation? 
    In the 2d case, for the Chekhov-Fock algebra of two different triangulations $\lambda$ and $\lambda'$, there is a coordinate change map $\Theta_{\lambda \lambda'}$ such that $\Theta_{\lambda \lambda'}\circ \Tr_\lambda = \Tr_{\lambda'}$. 
    The proof heavily relies on explicit coordinate change formulas developed by \cite{Hiatt} as one flips a diagonal of the triangulation. 
    To tackle the situation in 3d, one needs a handle on how square root quantum gluing modules associated to two different triangulations differing by a Pachner 2-3 move are related \cite{DG, Dim}.
    We expect our 3d quantum trace map to be natural with respect to changes of triangulation, in a sense analogous to the 2d case. 
    \item Currently, our 3d quantum trace map is defined for $3$-manifolds $Y$ without boundary (i.e.\ it only has cusps at infinity). 
    It should be relatively straightforward to generalize our construction to $3$-manifolds with any ideally triangulated boundary. 
    \item Once we allow ideally triangulated boundaries, as long as the boundary has no markings, the skein module of $Y$ will naturally be a module over the (ordinary) skein algebra of the boundary. 
    We expect our 3d quantum trace to be functorial in a sense that it is compatible with the 2d quantum trace on the skein algebra of the boundary. 
    It would be nice to check this. 
    \item How does the 3d quantum trace compare with the quantum UV-IR map \cite{NY, NY2, FN}? 
    This is the subject of a work in progress. 
    A related comparison in 2-dimensional setting was carried out in \cite{KLS}. 
    \item It should be possible to generalize our construction to $\mathrm{SL}_n$ quantum traces. 
    In the 2d case, this is addressed in \cite{NY2, Douglas, Kim, LS}, and ultimately solved in \cite{LY}. 
    Because the face suspension module is the tensor product of two triangle algebras, it seems the main ideas and techniques of \cite{LY} could be carried over into the 3d setting. 
    The 3d version of Fock-Goncharov coordinates \cite{FGa, FGb} are developed in \cite{GTZ, GGZ}. 
    \item Many ideally triangulated $3$-manifolds can be realized as a fibration over the circle, where the fiber is some ideally triangulated surface $\Sigma$ \cite{Gueritaud, McMullen}. 
    The image of links under the 3d quantum trace map whose projection onto the base of this fibration are trivial should roughly be the same as the image of that link under the 2d quantum trace map for $\Sigma$. 
    It would be illuminating to make a precise statement about this situation. 
    This question also seems intimately related to a conjecture posited and explored in \cite{BWY1, BWY2}.
    \item Another construction of the 3d quantum trace is currently being developed by S. Garoufalidis and T. Yu \cite{GY}. 
    It would be interesting to check if the two constructions are equivalent. 
\end{enumerate}

\bibliographystyle{alpha}
\bibliography{ref}

\begin{thebibliography}{GKRY16}

\bibitem[AB83]{AB}
M.~F. Atiyah and R.~Bott.
\newblock {The Yang-Mills Equations over Riemann Surfaces}.
\newblock {\em Philosophical Transactions of the Royal Society of London. Series A, Mathematical and Physical Sciences}, 308(1505):523--615, 1983.

\bibitem[AGLR22]{AGLR}
Prarit Agarwal, Dongmin Gang, Sangmin Lee, and Mauricio Romo.
\newblock Quantum trace map for 3-manifolds and a `length conjecture', 2022.

\bibitem[BFKB99]{BFK}
Doug Bullock, Charles Frohman, and Joanna Kania-Bartoszy\'{n}ska.
\newblock {Understanding the Kauffman bracket skein module}.
\newblock {\em Journal of Knot Theory and Its Ramifications}, 08(03):265--277, 1999.

\bibitem[Bon09]{Bonahon}
Francis Bonahon.
\newblock {\em Low-Dimensional Geometry, From Euclidean Surfaces to Hyperbolic Knots}.
\newblock AMS, 2009.

\bibitem[Bul97]{Bullock}
Doug Bullock.
\newblock Rings of sl2 (c)-characters and the kauffman bracket skein module.
\newblock {\em Commentarii Mathematici Helvetici}, 72(4):521--542, 1997.

\bibitem[BW11]{BW}
Francis Bonahon and Helen Wong.
\newblock Quantum traces for representations of surface groups in {${\rm SL}_2(\Bbb C)$}.
\newblock {\em Geom. Topol.}, 15(3):1569--1615, 2011.

\bibitem[BWY21]{BWY1}
Francis Bonahon, Helen Wong, and Tian Yang.
\newblock Asymptotics of quantum invariants of surface diffeomorphisms i: conjecture and algebraic computations, 2021.

\bibitem[BWY22]{BWY2}
Francis Bonahon, Helen Wong, and Tian Yang.
\newblock Asymptotics of quantum invariants of surface diffeomorphisms ii: The figure-eight knot complement, 2022.

\bibitem[BZBJ18]{BBJ}
David Ben-Zvi, Adrien Brochier, and David Jordan.
\newblock Integrating quantum groups over surfaces.
\newblock {\em J. Topol.}, 11(4):874--917, 2018.

\bibitem[CF00]{CF2}
L.~O. Chekhov and V.~V. Fock.
\newblock {Observables in 3D gravity and geodesic algebras}.
\newblock {\em Czech. J. Phys.}, 50:1201--1208, 2000.

\bibitem[CL22a]{CL}
Francesco Costantino and Thang T.~Q. L\^{e}.
\newblock Stated skein algebras of surfaces.
\newblock {\em J. Eur. Math. Soc. (JEMS)}, 24(12):4063--4142, 2022.

\bibitem[CL22b]{CL2}
Francesco Costantino and Thang T.~Q. L\^{e}.
\newblock {Stated skein modules of 3-manifolds and TQFT}, 2022.

\bibitem[Coo23]{Cooke}
Juliet Cooke.
\newblock Excision of skein categories and factorisation homology.
\newblock {\em Adv. Math.}, 414:Paper No. 108848, 51, 2023.

\bibitem[DG13]{DG}
Tudor~D. Dimofte and Stavros Garoufalidis.
\newblock {The Quantum content of the gluing equations}.
\newblock {\em Geom. Topol.}, 17:1253--1316, 2013.

\bibitem[Dim13]{Dim}
Tudor Dimofte.
\newblock {Quantum Riemann Surfaces in Chern-Simons Theory}.
\newblock {\em Adv. Theor. Math. Phys.}, 17(3):479--599, 2013.

\bibitem[Dou22]{Douglas}
Daniel~C. Douglas.
\newblock Quantum traces for $\mathrm{SL}_n(\mathbb{C})$: the case $n=3$, 2022.

\bibitem[FC99]{CF1}
Vladimir Fock and Leonid Chekhov.
\newblock {A quantum Teichmüller space}.
\newblock {\em Theoretical and Mathematical Physics}, 120:1245--1259, 09 1999.

\bibitem[FG06a]{FGb}
V.~V. Fock and A.~B. Goncharov.
\newblock {\em Cluster $\chi$-varieties, amalgamation, and Poisson---Lie groups}, pages 27--68.
\newblock Birkh{\"a}user Boston, Boston, MA, 2006.

\bibitem[FG06b]{FGa}
Vladimir Fock and Alexander Goncharov.
\newblock Moduli spaces of local systems and higher teichm\"{u}ller theory.
\newblock {\em Publications mathématiques de l’IHÉS}, 103(1):1–211, June 2006.

\bibitem[FN22]{FN}
Daniel~S. Freed and Andrew Neitzke.
\newblock {3d spectral networks and classical Chern-Simons theory}, 2022.

\bibitem[Gab17]{Gabella}
Maxime Gabella.
\newblock {Quantum Holonomies from Spectral Networks and Framed BPS States}.
\newblock {\em Commun. Math. Phys.}, 351(2):563--598, 2017.

\bibitem[GGZ15]{GGZ}
Stavros Garoufalidis, Matthias Goerner, and Christian Zickert.
\newblock {Gluing equations for $PGL(n, \mathbb{C})$-representations of 3-manifolds}.
\newblock {\em Algebr. Geom. Topol.}, 15(1):565--622, 2015.

\bibitem[GKRY16]{GKRY}
Dongmin Gang, Nakwoo Kim, Mauricio Romo, and Masahito Yamazaki.
\newblock {Aspects of Defects in 3d-3d Correspondence}.
\newblock {\em JHEP}, 10:062, 2016.

\bibitem[Gol84]{Gol1}
William~M Goldman.
\newblock The symplectic nature of fundamental groups of surfaces.
\newblock {\em Advances in Mathematics}, 54(2):200--225, 1984.

\bibitem[Gol86]{Gol2}
William~M Goldman.
\newblock Invariant functions on lie groups and hamiltonian flows of surface group representations.
\newblock {\em Inventiones mathematicae}, 85(2):263--302, June 1986.

\bibitem[GTZ15]{GTZ}
Stavros Garoufalidis, Dylan~P. Thurston, and Christian~K. Zickert.
\newblock {The complex volume of $\mathrm{SL}(n,\mathbb{C})$-representations of 3-manifolds}.
\newblock {\em Duke Math. J.}, 164(11):2099--2160, 2015.

\bibitem[Gu{\'e}06]{Gueritaud}
Fran{\c{c}}ois Gu{\'e}ritaud.
\newblock {On canonical triangulations of once-punctured torus bundles and two-bridge link complements}.
\newblock {\em Geometry \& Topology}, 10(3):1239 -- 1284, 2006.

\bibitem[GY]{GY}
Stavros Garoufalidis and Tao Yu.
\newblock in preparation.

\bibitem[Ha{\"i}22]{Haioun}
Benjamin Ha{\"i}oun.
\newblock Relating stated skein algebras and internal skein algebras.
\newblock {\em SIGMA Symmetry Integrability Geom. Methods Appl.}, 18:Paper No. 042, 39, 2022.

\bibitem[Hia10]{Hiatt}
Christopher Hiatt.
\newblock Quantum traces in quantum teichmüller theory.
\newblock {\em Algebraic and Geometric Topology - ALGEBR GEOM TOPOL}, 10:1245--1283, 06 2010.

\bibitem[Kas97]{Kashaev}
Rinat~M. Kashaev.
\newblock {Quantization of Teichm{\"u}ller Spaces and the Quantum Dilogarithm}.
\newblock {\em Letters in Mathematical Physics}, 43:105--115, 1997.

\bibitem[Kim22]{Kim}
Hyun~Kyu Kim.
\newblock ${\rm sl}_3$-laminations as bases for ${\rm pgl}_3$ cluster varieties for surfaces, 2022.

\bibitem[KLS23]{KLS}
Hyun~Kyu Kim, Thang T~Q L{\^{e}}, and Miri Son.
\newblock {SL}2 quantum trace in quantum teichmüller theory via writhe.
\newblock {\em Algebraic {\&} Geometric Topology}, 23(1):339--418, mar 2023.

\bibitem[L{\^{e}}18]{Le}
Thang T.~Q. L{\^{e}}.
\newblock Triangular decomposition of skein algebras.
\newblock {\em Geom. Topol.}, 9:591--632, 2018.

\bibitem[LS21]{LS}
Thang T.~Q. L{\^e} and Adam~S. Sikora.
\newblock {Stated SL(n)-Skein Modules and Algebras}.
\newblock 2021.

\bibitem[LY23]{LY}
Thang T.~Q. Lê and Tao Yu.
\newblock Quantum traces for $sl_n$-skein algebras, 2023.

\bibitem[McM96]{McMullen}
Curtis~T. McMullen.
\newblock {\em Renormalization and 3-Manifolds Which Fiber over the Circle (AM-142), Volume 142}.
\newblock Princeton University Press, Princeton, 1996.

\bibitem[NY20]{NY}
Andrew Neitzke and Fei Yan.
\newblock {$q$}-nonabelianization for line defects.
\newblock {\em J. High Energy Phys.}, (9):153, 65, 2020.

\bibitem[NY22]{NY2}
Andrew Neitzke and Fei Yan.
\newblock The quantum {UV}-{IR} map for line defects in {$\mathfrak{gl}(3)$}-type class {$S$} theories.
\newblock {\em J. High Energy Phys.}, (9):Paper No. 81, 50, 2022.

\bibitem[Pet41]{Pet}
H.~Petersson.
\newblock {Über eine Metrisierung der automorphen Formen und die Theorie der Poincaréschen Reihen}.
\newblock {\em Mathematische Annalen}, 117:453--537, 1940/1941.

\bibitem[PS00]{PS}
J{\'o}zef~H. Przytycki and Adam~S. Sikora.
\newblock On skein algebras and {${\rm Sl}_2({\bf C})$}-character varieties.
\newblock {\em Topology}, 39(1):115--148, 2000.

\bibitem[Thu80]{Thu}
William Thurston.
\newblock {\em {The Geometry and Topology of Three-Manifolds}}.
\newblock 1980.

\bibitem[Tur91]{Tur}
Vladimir~G. Turaev.
\newblock {Skein quantization of Poisson algebras of loops on surfaces}.
\newblock {\em Annales scientifiques de l'École Normale Supérieure}, 24(6):635--704, 1991.

\bibitem[Wei58]{Weil}
André Weil.
\newblock {Modules des surfaces de Riemann}.
\newblock {\em Séminaire Bourbaki}, 4:413--419, 1956-1958.

\end{thebibliography}
\end{document}